\documentclass[]{amsart}
\usepackage{amsmath}
\usepackage{amsfonts}
\usepackage{amssymb}
\usepackage{ mathrsfs }
\usepackage{amsthm}
\usepackage{enumitem}
\usepackage{graphicx,color,xcolor,tikz}
\usepackage[utf8]{inputenc}
\usepackage{hyperref}
\hypersetup{
    colorlinks=true,                         
    linkcolor=blue, % Couleur des liens internes
    citecolor=red, % Couleur des numéros de la biblio dans le corps
  } 
\usepackage[pagewise]{lineno}
\usepackage{tikz}

\newtheorem{Thm}{Theorem}
\newtheorem{Prop}{Proposition}

\def\R{\mathbb R}

\newcommand{\p}{\partial}

\title[A model of liquid-vapor interaction
    with metastability]{A nonisothermal thermodynamical model of
      liquid-vapor interaction with metastability}

\author{Hala Ghazi}
\address[Hala Ghazi]{Laboratoire Jean Leray, Universit\'e de Nantes \&
  CNRS UMR 6629, 
BP 92208, F-44322 Nantes Cedex 3, France}
\email{hala.ghazi@univ-nantes.fr}

\author{
Fran\c{c}ois James}
\address[Fran\c{c}ois James]{Institut Denis Poisson,
Universit\'e d'Orl\'eans \& CNRS UMR 7013, BP 6759, F-45067 Orl\'eans
Cedex 2, France}
\email{francois.james@univ-orleans.fr}

\author{H\'el\`ene Mathis}
\address[H\'el\`ene Mathis]{Laboratoire Jean Leray, Universit\'e de
  Nantes \& CNRS UMR 6629, 
BP 92208, F-44322 Nantes Cedex 3, France}
\email{helene.mathis@univ-nantes.fr}
\begin{document}

\maketitle

%--------------------------------------------------------
% Abstract
%--------------------------------------------------------
\begin{abstract}
The paper concerns the construction of a compressible liquid-vapor relaxation model which is
able to capture the
metastable states of the non isothermal van der Waals model as well as
saturation states. 
Starting from the Gibbs formalism, we propose a dynamical system which
complies with the second law of thermodynamics. Numerical simulations
illustrate the expected behaviour of metastable states: an initial
metastable condition submitted to a certain perturbation
may stay in the metastable state or reaches a saturation state.
The dynamical system is then coupled to the dynamics of the
compressible fluid using an Euler set of equations supplemented
by convection equations on the fractions of volume, mass and energy of
one of the phases. 
\end{abstract}
%--------------------------------------------------------



\tableofcontents

%--------------------------------------------------------
% Intro
%--------------------------------------------------------
\section{Introduction}

Metastable two-phase flows are involved in many industrial
applications, for instance in scenarii of safety accidents in
pressurized water reactors. They can also appear in everyday life. 
Warming water in a microwave with the maximum power may make the
liquid water being metastable: its temperature increases
above the saturation temperature; the water is the called
superheated.
The metastability corresponds then to a delay in vaporization.
Even a small perturbation of
the metastable water may lead to the brutal appearance of a vaporization wave.
In \cite{bartak90} an analogous phenomenon is highlighted. Liquid
water can be brought to a superheated state by means of a very rapid
depressurization. The depressurization is stopped suddenly by an
explosive nucleation causing, in its turn, an increase of the pressure.

As pointed out in \cite{DeLorenzo18}, such compressible two-phase
flows are characterized by three main difficulties. The first two
difficulties are linked to the dynamics of the fluid, namely the
compressibility of both phases and the presence of the moving
interface between them. The third difficulty lies in the modelling of
the thermodynamical exchanges which occur at the interface. 
The references \cite{saurel08} and \cite{Zein10} focus on the two
first difficulties and propose models coming from the Bear-Nunziato
model for compressible two-phase flows. The models are either 6 or 5
equations models, possibly including pressure and velocity interfacial
terms. Each phase possesses his own convex Equation of State (EoS),
namely a stiffened gas law (or a Mie-Gr\"uneisen generalization). 
Relaxation towards thermodynamical equilibrium is assumed to be
infinitely fast, so that metastable states appear far from the
vaporization fronts. In \cite{DeLorenzo18, DeLorenzo17} and
\cite{DeLorenzo2019}, the authors improve this approach by using the realistic tabulated law
IAPWS-IF97 EoS coupled with cubic interpolation and accurate HLLC-type
numerical scheme. 
They compare different models of a same hierarchy.
Starting from a
single-velocity six equations model with full disequilibrium, they
consider an
homogeneous equilibrium model where the liquid and the vapor are at
thermodynamical equilibrium (meaning stable) and a homogeneous
relaxation model in which the liquid is assumed to be metastable and
the vapor is at
saturation. Again emphasis is given to the two first difficulties of
compressible two-phase flows, the question of metastability being
addressed solely in the choice of the complex EoS. 

In the present paper, we focus on the third difficulty, namely
the modelling of thermodynamical
transfers and the appearance of metastable states.
As the dynamics of the flow is concerned, we adopt the strategy
proposed in \cite{DeLorenzo18, DeLorenzo17} and consider the
homogeneous relaxation model given in \cite{HelluyHurisse15} and
\cite{Hurisse17}. 
We assume that the two phases evolve with the same velocity and
consider the mass, momentum and energy conservation equation of the flow. 
The specificity is to assume that the two phases follow the same
non-convex EoS, namely a reduced form of the van der Waals equation. 
Because  the model involves a mixture pressure based on this cubic
equation,
the convective system is not strictly
hyperbolic, notably in the van der Waals spinodal zone. 
To get rid of this problem, the pressure is relaxed and depends on
additional quantities, which are the fractions of volume, mass and energy of
one of the phases. These fractions obey to convective equations with
relaxation terms
towards the thermodynamic equilibrium. The core of the paper is the
proper definition of these relaxation terms. 
To do so, we extend the method we proposed in \cite{james} in the
isothermal case and provide a characterization of thermodynamic
equilibria which are either saturation states, stable and metastable states.

In a first section, we recall some basic facts of thermodynamics in the
extensive and intensive form \cite{callen85}, notably the notion of
entropy. We focus on the van der Waals model, which is well-known
to depict stable and metastable states but is calssically used with a convexification
correction to properly depict saturation. It turns out that the
representation of metastable states of the van der Waals model is done
in the volume-pressure plane, although the equations of motion require
to manipulate phase diagram and EoS defined in the volume-energy plane.
A large part of Section \ref{sec:vdw} then concerns
the representation of stable, metastable and spinodal zone in the
volume-energy plane. 

In Section \ref{sec:optim} we investigate the thermodynamic stability of a
system described by the non convex EoS of van der Waals in its reduced
form.
As suggested in \cite[chap. 8]{callen85}, introducing heterogeneity
in a system is the hallmark of phase transition. Hence, in order to introduce
heterogeneity in the system, we decompose it in an arbitrary number
of subsystems depicted by the same nonconvex EoS. The second principle
of thermodynamics leads to a
constrained maximization problem on the mixture entropy. 
It turns out that the number of subsystems is limited to two, in accordance 
with the Gibbs phase rule.
Then the study of the optimization problem leads to two
possible kinds of maximizers, 
either saturation states or states corresponding to the identification of the two phases.
In the latter case, there is no distinction between the two phases and
all the states belonging to the van der Waals EoS
are possible maximizers, including the non-admissible (physically unstable) states of the
spinodal zone.
On the other hand, the saturation states correspond to the coexistence of the two phases at saturation, with
equality of the pressures, temperatures and chemical potentials of the two phases, corresponding to the convexification of the EoS.

Section \ref{sec:syst_dyn} provides a dynamical description of the
thermodynamic equilibrium and of its two kind of equilibrium states.
Following the approach developed in \cite{james} and \cite{ghazi19}
in the isothermal case, we introduce a dynamical system whose long-time
equilibria coincide with the maxima of the above optimization
problem, under a mixture entropy growth criterion.
We focus in this paper on a dynamical system on the fractions of
volume, mass and energy of the phase 1. The system is designed to recover the above two possible
equilibria: either saturation states or states corresponding to the
identification of the two phases. In the latter case, the equilibrium
is characterized by the equality of all the fractions which converge
asymptotically to some value belonging to $]0,1[$. Hence, as the two
phases identify, the fractions are not equal to $0$ or $1$, in contrast with
the Baer-Nunziato type two-phase models 
\cite{BN86}. This is one fundamental feature of the dynamical model we propose.
Another property stands in the attractivity of the equilibria and
their attraction basins. If the
energy-volume state of the mixture belongs to the spinodal zone, then
the corresponding equilibrium is a saturation state, whatever the
initial conditions of the dynamical system are. Thus the dynamical system
gets rid of unstable states of the spinodal zone by construction. 
On the other hand, if the mixture state belongs to a metastable zone,
there are two possible equilibria depending on the perturbation: 
either the identification of the
two phases to the mixture metastable state or a saturation state.
This interesting property was already highlighted in \cite{james,ghazi19} and is extended here to the non-isothermal case.
Numerical simulations illustrate the attraction of each equilibria and
typical trajectories of the dynamical system in the volume-energy
plane, volume-pressure plane and in the fractions domain.

Finally Section \ref{sec:homo-model} addresses the coupling
between the thermodynamics and the compressible dynamics of the
two-phase flows we are interested in. Following the approach in
\cite{HelluyHurisse15,Hurisse17,DeLorenzo18}, we consider that the fluid is homogeneous in the
sense that the two phases evolve with the same velocity. Then the
model is based on the 
conservation equations of total mass, momentum and energy. To close
the system, it is endowed with a complex equation of state
depending on the fractions of volume, mass and energy of one of the
phases. To ensure the return to the thermodynamic equilibrium, the
 evolution equations of the fractions admit relaxation source terms
derived from the dynamical system studied in Section \ref{sec:syst_dyn}.
Because the mixture pressure involves the van der Waals EoS, the
hyperbolicity is non strict. However it has been proved in
\cite{james} that the domains of hyperbolicity of the complete model
strongly depend on the attraction basins of the dynamical system.
In order to illustrate the dynamical behaviour of the model, we
provide a numerical scheme based on a fraction step approach: the
convective part is approximated by an explicit HLLC solver while the
source terms is integrated by a RK4 method.

%----------------------------------------------------------------
\section{Thermodynamic assumptions and the van der Waals EoS}
%----------------------------------------------------------------
\label{sec:Thermo}

%----------------------------------------------------------------
\subsection{Description of a single fluid}
%----------------------------------------------------------------

We consider a monocomponent fluid of mass $M \geq 0$, occupying a volume $V \geq 0$
with internal energy $E \geq 0$. Following the Gibbs formalism
\cite{Gibbs, callen85}, we introduce the extensive entropy $S$ of
the fluid as a function of its mass $M$,  volume $V$ and energy $E$:
\begin{equation}
  \label{eq:entropy}
  S:(M,V,E)\mapsto S(M,V,E).
\end{equation}
All the above quantities are said extensive, in the sense that
if the system is doubled,  then its mass, volume,
energy and entropy are doubled as well. 
Any extensive quantity is said
positively homogeneous of degree 1 (PH1) and satisfies
\begin{equation}
  \label{eq:PH1}
  \forall \lambda >0, \quad S(\lambda M, \lambda V, \lambda E)= \lambda
  S(M,V,E).
\end{equation}
We assume that the entropy
function $S$ belongs to $C^2(\R^+ \times \R^+ \times \R^+)$.
It allows to introduce intensive quantities, that are positively
 homogeneous functions of degree 0 (PH0), corresponding to
 derivatives of extensive functions.
From the gradient vector $\nabla S$ of the entropy $ S$, we
commonly define the pression $p$, the temperature $T$ and the chemical
potential $\mu$ by
\begin{equation}
  \label{eq:Intensive-quantities}
    \dfrac{1}{T}= \dfrac{\p S}{\p E} (M,V,E) , \quad
    \dfrac{p}{T}=  \dfrac{\p S}{\p V} (M,V,E),  \quad  
    \dfrac{\mu}{T}=-\dfrac{\p S}{\p M} (M,V,E),
\end{equation} 
leading to the fundamental thermodynamics extensive Gibbs relation
\begin{equation}
\label{eq:relation-thermo}
dS=-\dfrac{\mu}{T} d M+ \dfrac{p}{T} dV+  \dfrac{1}{T} d E.
\end{equation}
Standard thermodynamics requires that
\begin{equation}
  \label{eq:positivity-of-T}
  T =\left( \dfrac{\p S}{\p E}\right)^{-1}>0.
\end{equation}
Since the entropy $S$ is a PH1 function, it verifies the Euler
relation
\begin{equation}
  \label{eq:Euler-relation}
  S(M,V,E)=\nabla S(M,V,E) \cdot
  \begin{pmatrix}
    M\\
    V\\
    E
  \end{pmatrix},
\end{equation}
which, combined with the definitions \eqref{eq:Intensive-quantities},
gives
\begin{equation}
  \label{eq:Gibbs-relation}
  S(M,V,E)=   -\dfrac{ \mu M}{T}  +   \dfrac{pV}{T} +  \dfrac{E}{T}.
\end{equation}
Introducing the specific volume $\tau=V/M$ and the specific internal
energy $e=E/M$, and using the homogeneity of the extensive entropy
function, one can define the specific entropy $s$
\begin{equation}
\label{eq:intensive-entropy}
s(\tau,e)=S\left(1,  \dfrac{V}{M},  \dfrac{E}{M}\right)
=\dfrac{1}{M} S(M,V,E).
\end{equation}
We keep the same notations to denote the pressure and the temperature
expressed as functions of the specific volume and energy
\begin{equation}
  \label{eq:intensive-form}
  \dfrac{1}{T}= \dfrac{\p s}{\p e} (\tau,e), 
  \qquad
  \dfrac{p}{T} = \dfrac{\p s}{\p\tau} (\tau,e).
\end{equation}
The fundamental thermodynamics relation in its intensive form reads as
follow
\begin{equation}
  \label{eq:int_total_differential}
  Tds=d e+p d\tau.
\end{equation}
and the intensive counterpart of relation
\eqref{eq:Gibbs-relation} is
\begin{equation}
  \label{eq:gibbs-intensive}
  Ts=-\mu + p\tau +e.
\end{equation}

%----------------------------------------------------------------
\subsection{The van der Waals Equation of State}
%----------------------------------------------------------------
\label{sec:vdw}

In this work we focus  on a  non necessarily concave nor convex entropy function $s$.
 A common exemple is the van der Waals Equation of State (EoS), which
 entropy reads
\begin{equation}
 \label{eq:vdw-entropy}
  s(\tau,e)=C_v \ln \left(\dfrac{a}{\tau}+e \right)+R \ln(\tau-b)+s_0,
\end{equation} 
where $R$ is the universal constant of gas, $C_v>0$ the calorific
constant at constant volume, $s_0$  is the
entropy of reference, and $a$ and $b$ are the two nonnegative parameters
\cite{callen85,landau}. 

The entropy is well defined for $(\tau,e)\in (\mathbb R^+)^2$ such that
\begin{equation}
  \label{eq:dom_tau_e}
  \tau>b, \qquad \dfrac{a}{\tau}+e >0. 
\end{equation}
The corresponding definition domain of $s$ is denoted $D_s$:
\begin{equation}
  \label{eq:D_s}
  D_s:=\left\{(\tau,e) \in (\R^+)^2;\,   \tau>b \text{ and }
  \dfrac{a}{\tau}+e>0 \right\}. 
\end{equation}
According to relations \eqref{eq:intensive-form}, the van der Waals
temperature and pressure read
\begin{eqnarray}
  \label{eq:VDW_temperature}
  T(\tau,e)&=&\dfrac{1}{C_v} \left(e+ \dfrac{a}{\tau}\right),\\
  \label{eq:VDW_pressure}
  p(\tau,e) &=& \dfrac{R }{C_v(\tau -b)}\left(e+
    \dfrac{a}{\tau}\right)-\dfrac{a}{\tau^2}=
  \dfrac{RT(\tau,e)}{\tau-b}-\dfrac{a}{\tau^2}.
\end{eqnarray}
The van der Waals entropy is neither concave nor convex. Indeed the
coefficients of its Hessian matrix $H_s(\tau,e)$ are given by
    \begin{equation}\label{eq:hessian_s}
 \begin{cases}
\dfrac{\p^2 s}{\p\tau^2}(\tau,e)= \dfrac{1}{T(\tau,e)}\left(\dfrac{2a}{\tau^3} -
    \dfrac{a}{\tau^2}\dfrac{R}{C_v(\tau-b)}-\dfrac{RT(\tau,e)}{(\tau-b)^2}\right) +
    \dfrac{a}{C_v}\dfrac{p(\tau,e)}{\tau^2 T^2(\tau,e)},\\   
\dfrac{\p^2 s}{\p e^2}(\tau,e) = {}-\dfrac{C_v}{T^2(\tau,e)},\\
\dfrac{\p^2 s}{\p\tau \p e}(\tau,e) = \dfrac{C_v}{T^2(\tau,e)}\dfrac{a}{\tau^2}.
\end{cases}
    \end{equation}
Since the temperature $T$ is positive on $D_s$, one has
    \begin{equation}\label{eq:non_conv_property}
\dfrac{\p^2 s}{\p e^2} <0.
    \end{equation}
However the entropy function $s$ is not globally concave and its domain of concavity restricts to the set
where the determinant of $H_s$ is positive, that is
    \begin{equation}\label{eq:concav_domain}
D_c:=\left\{(\tau,e)\in 	D_s;\,   \dfrac{\p^2 s}{\p\tau^2}\dfrac{\p^2 s}{\p e^2}-\left(\dfrac{\p^2 s}{\p e\p\tau}\right)^2 > 0\right\}.
    \end{equation}

The non-concavity property of the van der Waals entropy makes it an
appropriate toy-model to represent liquid-vapor phase transition
\cite[chap.9]{callen85}. 
States belonging to the concavity region of the entropy refer to
stable and metastable liquid and vapor states. In contrast states
belonging to the non-concavity region are non-admissible states.
The purpose of this section is to precise the geometrical loci of
these states and provide
representations of the phase diagrams of the van der Waals EoS 
in both the $(\tau,p)$ and the $(\tau,e)$ planes.

In all the representations given in the sequel,
we use a reduced form of the EoS, as the one proposed
in \cite{Fechter_Munz17}, with the parameters 
\begin{equation}
  \label{eq:vdw_param}
a= 1, \quad 
b=0.5, \quad
R=0.5, \quad 
C_v= 3, \quad 
s_0=0.
\end{equation}

Usually the metastable zones of the van der Waals EoS are defined and
observable in the $(\tau,p)$ plane at constant temperature.
This implies to manipulate the entropy and the pressure as
functions of the volume $\tau$ and the temperature $T$.
Adapting relations \eqref{eq:VDW_temperature}, \eqref{eq:VDW_pressure}
and \eqref{eq:vdw-entropy} leads to
\begin{equation}
  \label{eq:pT_sT}
  \begin{aligned}
    s(\tau,T) &= C_v \ln ( C_vT) + R \ln(\tau-b)+s_0,\\
    p(\tau,T) &= \dfrac{RT}{\tau-b}-\dfrac{a}{\tau^2}.
  \end{aligned}
\end{equation}

\begin{figure}[h]
  \centering
  \includegraphics[width=\linewidth]{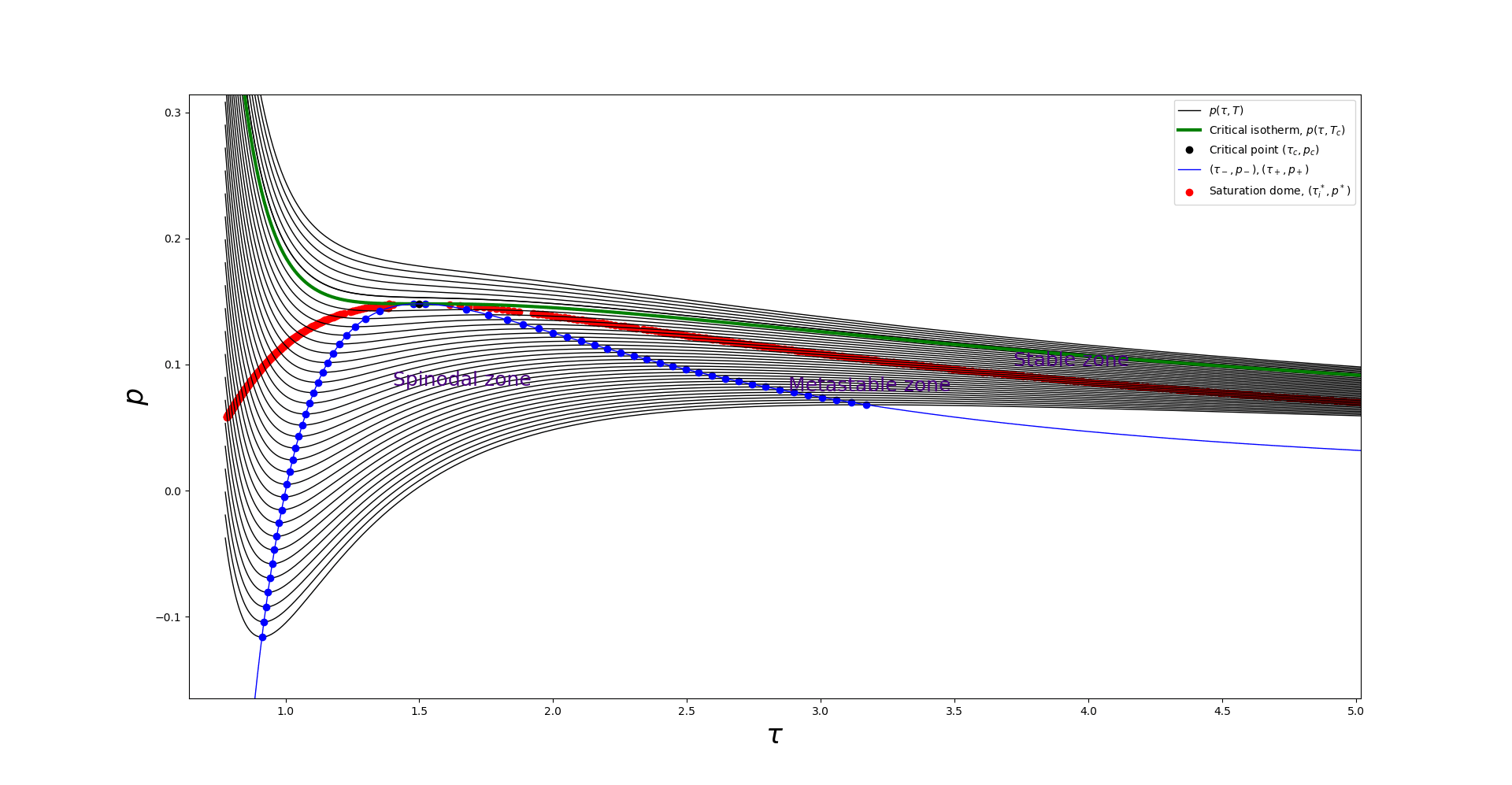}
  \caption{Isothermal curves of the van der Waals EoS in the
    $(\tau,p)$ plane. Isothermal
    curves $p(\tau,T)$ are plotted in black. The isothermal
    curve at critical temperature $T=T_c$ is plotted in green. Below the
    critical isothermal curve, the pressure is not monotone with
    respect to the specific volume and
    increases in the spinodal zone of non admissible states. 
    This zone is delimited by the blue
    curve representing the set of minima $(\tau_-,p(\tau_-,T))$ and maxima
    $(\tau_+,p(\tau_+,T))$ of the pressure for each temperature $T<T_c$.
    The Maxwell equal area rule construction allows to replace
    the non physically admissible increasing branch of an isothermal
    curve by computing two volumes $\tau_1^*$ and $\tau_2^*$ at each
    temperature $T<T_c$, such that $p(\tau_1^*,T)=p(\tau_2^*,T)$. The set
    of these volumes is represented in red in the graph and
    corresponds to the saturation dome. 
    The states belonging to decreasing branches of isothermal curves,
    below the saturation dome (in red) and above the spinodal zone (in
    blue), are called metastable states.}
  \label{fig:vdw_tau_p}
\end{figure}

We represent in Figure~\ref{fig:vdw_tau_p} the isothermal curves $(\tau,p(\tau,T))$ (black
lines) in the $(\tau,p)$ plane for fixed temperatures $T$. 
There exists a unique critical temperature $T_c$ for which the pressure admits a unique
inflection point $(\tau_c,p(\tau_c,T_c))$, called the critical point. For the reduced van der Waals law, $T_c=1$.
For supercritical temperature $T>T_c$, 
the pressure is a strictly decreasing function of the specific volume. 
Below the critical isothermal curve, 
for $T<T_c$, the pressure is an increasing
function of the volume between the minimum $(\tau_-,p(\tau_-,T))$ and
the maximum $(\tau_+,p(\tau_+,T))$. This increasing branch refers
to non physically admissible states.
The critical isothermal curve is plotted in green in
Figure~\ref{fig:vdw_tau_p}.
The set of minima and maxima is plotted in blue in
Figure~\ref{fig:vdw_tau_p} 
and delimits the spinodal zone.
Actually the spinodal zone in the $(\tau,p)$ plane  corresponds to the
zone $D_s \setminus D_c$ in the plane $(\tau,e)$ where the entropy
function is not concave.

At a given temperature $T<T_c$, it is classical to replace the non
admissible increasing branch of the pressure by a specific isobaric line
satisfying the Maxwell equal area rule. Such a construction
defines two volumes $\tau_1^*$ and $\tau_2^*$, for each temperature
$T<T_c$, such that  $p(\tau_1^*, T) = p(\tau_2^*,T)$.
Their set, represented in red in Figure~\ref{fig:vdw_tau_p},
is called the saturation dome.
The states belonging to decreasing branches of isothermal curves,
below the saturation dome (in red) and above the spinodal zone (in
blue) are called metastable states.

The purpose of this section is to provide a representation of the
saturation dome, spinodal and metastable zones in the
$(\tau,e)$ plane.
\begin{figure}[h]
  \centering
  \includegraphics[width=\linewidth]{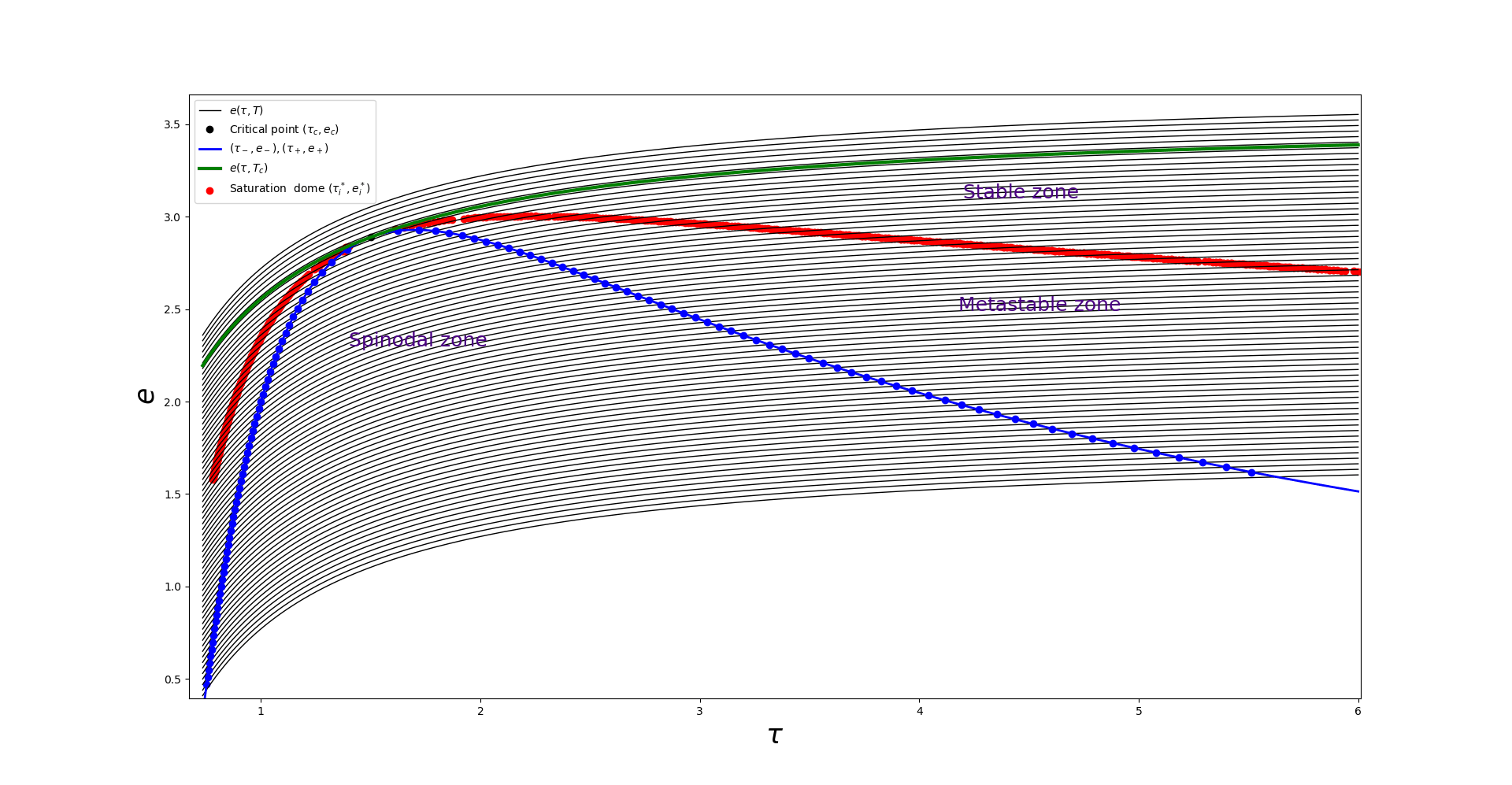}
  \caption{Isothermal curves of the van der Waals EoS in the
    $(\tau,e)$-plane.
    The black lines
    correspond to isothermal curves
    $e(\tau,T)$. The isothermal
    curve at the critical temperature $T=T_c$ is plotted in
    green. States belonging to the zone above the critical isothermal
    curve are supercritical states. The
    spinodal zone is delimited by the blue curve, which is the graph
    of the function $g.$ defined in \eqref{eq:function-g}.
    The saturation dome
    is represented by the set of red points. Stable states belong to
    the areas below the critical isothermal curve (in green) and above
    the saturation dome (in red). The metastable areas correspond to
    zones above the spinodal zone (in blue) and below the saturation
    dome (in red).}
  \label{fig:vdw_tau_e}
\end{figure}
We represent in Figure~\ref{fig:vdw_tau_e} the isothermal curves in
the $(\tau,e)$ plane. 
The spinodal zone corresponds to the domain where the concavity of the entropy
function $(\tau,e)\mapsto s(\tau,e)$ changes.
According to the definition \eqref{eq:hessian_s} of the Hessian matrix
$H_s(\tau,e)$ of the
entropy $s$, this domain is delimited by the set of states
$(\tau,e)\in D_s$ such that
\begin{equation}
  \label{eq:det_Hs}
  \text{det}(H_s)(\tau,e)=0.
\end{equation}
Solving \eqref{eq:det_Hs} allows to define the spinodal zone $Z_{\text{Spinodal}}\subset D_s$
\begin{equation}
  \label{eq:zone_spinodale}
  Z_{\text{Spinodal}}:=\{(\tau,e)  \in D_s; e<g(\tau) \},
\end{equation}
where 
\begin{equation} 
  \label{eq:function-g}
  g(\tau ) = \displaystyle \frac{2aC_{v}(\tau - b )^{2}}{R\tau^{3}}
  -\displaystyle \frac{a}{\tau}.
\end{equation}
The critical isothermal curve (green curve) admits a unique intersection point with
the graph of $g$ which turns to be the critical point $(\tau_c,e_c=g(\tau_c))$. 
The Maxwell construction, which is usually defined in the
$(\tau,p)$-plane, admits its counterpart in the
$(\tau,e)$-plane. Actually 
the construction of the
concave hull of the van der Waals entropy function
$(\tau,e)\mapsto
s$
is equivalent to
the Maxwell equal area rule construction \cite{caro2004,faccanoni12}.
An analogous proof, based on the 
properties of the Legendre transform, is available in \cite{HM10}.
In practice, the computation of the concave hull of the graph of $(\tau,e)\mapsto
s$ boils down to the construction of a ruled surface. 
For any point $(\tau,e)$, this ruled surface contains
a segment which is bitangent to the graph of $(\tau,e)\mapsto s$ in two
points denoted $(\tau_1^*,e_1^*)$ and $(\tau_2^*,e_2^*)$.
The set of points $((\tau_1^*,e_1^*),(\tau_2^*,e_2^*))$ defines the
saturation dome in the $(\tau,e)$ plane and is  represented
in red in
Figure~\ref{fig:vdw_tau_e}. Note that the computation of the points
$((\tau_1^*,e_1^*),(\tau_2^*,e_2^*))$ is not explicit and requires the resolution of a
nonlinear system \cite{caro2004,faccanoni12,HM10}. 
However if we assume that the set of the red dots is the graph of 
a function $g^*:\tau\to g^*(\tau)$, then the curve
$(\tau,g^*(\tau))$ defines the saturation dome $  Z_\text{Saturation}$, that is
\begin{equation}
  \label{eq:saturation_zone}
  Z_\text{Saturation}:=\{(\tau,e) \in D_s; \, e=g^*(\tau)\}.
\end{equation}
Thus the metastable states $Z_\text{Metastable}$ corresponds to the states belonging to the
saturation domain but outside the spinodal zone
\begin{equation}
  \label{eq:zone_metastable}
  Z_\text{Metastable}:=\{(\tau,e) \in D_s; \, g(\tau)<e<g^*(\tau)\}.
\end{equation}
Finally the stable zones, either stable liquid or stable vapor states, correspond to states below the critical
isotherm curve and above the saturation dome
\begin{equation}
  \label{eq:zone_stable}
  Z_\text{Stable}:=\{(\tau,e) \in D_s; \, g^*(\tau)<e<e(\tau, T_c) \}.
\end{equation}

%--------------------------------------------------------------------
 \section{Thermodynamics of equilibria for a multicomponent system}  
%--------------------------------------------------------------------
\label{sec:optim}

We consider a system of mass $M> 0$, volume $V>0$ and energy $E>0$
which is composed of
$I$ subsystems. 
Each subsystem $i=1,\dots,I$ is
characterized by its mass
$M_i \geq 0$, its volume $V_i \geq 0$ and its energy
$E_i \geq 0$.
Moreover, we assume that each subsystem $i$ follows the same non
concave entropy $S(M_i,V_i,E_i)$, namely the van der Waals EoS in its
extensive setting.
The conservations of mass and energy require that
\begin{equation}
  \label{eq:extensive_constraints}
  M=\sum_{i=1}^{I} M_i, \quad E=\sum_{i=1}^{I} E_i.
\end{equation}
Furthermore, we suppose that all the subsystems are immiscible
and that no vacuum appears, in the sense that
\begin{equation}
  \label{eq:extensive_constraints_V}
  V=\sum_{i=1}^{I} V_i.
\end{equation}
The entropy of the system is
the sum of the partial entropies of each subsystem:
\begin{equation*}
(M_{i},V_{i},E_i)_{i=1,\dots,I}\mapsto  \sum \limits_{i=1}^{I} S(M_{i},V_{i},E_i).
\end{equation*}
According to the second principle of thermodynamics, the entropy of the
multicomponent system achieves its maximum at Thermodynamic equilibrium.
Considering a state vector $(M,V,E)$ of the multicomponent system, the 
equilibrium entropy is
\begin{equation}
  \Sigma(M,V,E)=  \sup\limits_{(M_{i},V_{i},E_i)\in (\mathbb R^+)^3} \sum
  \limits_{i=1}^{I} S(M_{i},V_{i},E_i),
  \label{eq:equi_entrop_ext}
\end{equation}
under the constraints
\eqref{eq:extensive_constraints}-\eqref{eq:extensive_constraints_V}.

We now turn to the intensive formulation of the maximization problem.
In the following, we denote $\varphi_i=M_i/M
\in [0,1]$ the mass fraction, $\alpha_i=V_i/V\in [0,1]$ the volume
fraction and $\xi_i=E_i/E\in [0,1]$ the energy fraction.
Given $\tau=V/M$ and $e=E/M$ the specific volume and specific energy of the
multicomponent system, the specific volume of the subsystem $i=1,\dots,I$ is 
$\tau_i=V_i/M_i=\alpha_i\tau/\varphi_i
\geq 0$ and its specific energy is $e_i=E_i/M_i=z_ie/\varphi_i\geq 0$.

The conservation of mass and energy and the volume constraints read now
\begin{equation}
  \label{eq:int_constraints}
  \sum \limits_{i=1}^{I}\varphi_{i}=1,\quad
  \sum \limits_{i=1}^{I}\varphi_{i} \tau_i= \tau,\quad
  \sum \limits_{i=1}^{I}\varphi_{i} e_i= e.
\end{equation}

Using the homogeneity property of the extensive entropy function $S$, the
definitions of the mass fractions $\varphi_i$ and phasic intensive
quantities
$\tau_i$ and $e_i$, the intensive form of the equilibrium entropy of the system
is, for any state vector $(\tau,e)$
\begin{equation} 
\sigma(\tau,e)= \sup\limits_{(\tau_i,e_i)\in (\mathbb R^+)^2}
\sum
\limits_{i=1}^{I}\varphi_{i}s(\tau_{i},e_i),
\label{eq:equi_entrop_int}
\end{equation}
under the constraints \eqref{eq:int_constraints}.

%--------------------------------------------------------------------
\subsection{The Gibbs phase rule}
\label{sec:number_I}
%--------------------------------------------------------------------
For the moment the number $I$ of subsystems, 
potentially present at the thermodynamic equilibrium, is not determined.
Actually the theorem of Caratheodory gives a first estimate on the
number of subsystems $I$. We recall the theorem statement and
refer to \cite{rockafellar, HUL01} for a detailed proof.
\begin{Thm}{(Theorem of Caratheodory)}
  \label{thm:Caratheodory}
  Let $A$ be a subset in $\R ^n $ and $\text{conv}(A)$ the set of all
  the convex combinations of elements of $A$.
  Then every point $ x \in \text{conv} (A)$
  can be represented as a convex combination of $ (n + 1) $ points of
  $A$.
\end{Thm}
In the present context, the theorem provides the following first bound.
\begin{Prop}
  \label{prop:I_leq_3}
  Consider the  maximization problem \eqref{eq:equi_entrop_int} 
  under the constraints
  \eqref{eq:int_constraints}.
  The number of subsystems which may coexist at equilibrium is
  $I\leq 3$.
\end{Prop}

\begin{proof}
  Consider $(\tau,e)\in (\mathbb R^+)^2\mapsto -s(\tau,e)\in \mathbb R$.
  According to Caratheodory's theorem,
  the convex hull of the epigraph of $-s$ at any point $(\tau,e)\in
  (\mathbb R^+)^2$ is
  \begin{equation}
    \label{eq:conv}
    (conv \, (-s)) (\tau,e) = \inf \sum_{i=1}^{3} -\lambda_i
    s(\tau_i,e_i),
  \end{equation}
  with
  $\sum_{i=1}^{3} \lambda_i \tau_i =\tau$ 
  and
  $\sum_{i=1}^{3} \lambda_i e_i =e$
  where the infimum is taken over all the expressions of $(\tau,e)$  as a
  convex combinations of three points $(\tau_i,e_i)$, $i=1,2,3$.
  Now considering $\lambda_i=\varphi_i$, we recover the intensive
  constraints \eqref{eq:int_constraints} and the maximization problem
  \eqref{eq:equi_entrop_int}
  is equivalent to the determination of the concave hull of $s$.
\end{proof}
As a consequence of Caratheodory's Theorem, 
at the most three phases
remain at thermodynamic equilibrium. 
This result is in total agreement with the Gibbs phase
rule. Indeed, considering a single component system, the Gibbs
phase rule states that  the number of phases is $I=3-F$, where $F\geq
0$ is the degree of freedom \cite{ball, mortimer}.

Actually when considering the van der Waals EoS, the admissible number
of subsystems present at Thermodynamic equilibrium restricts to at
most 2.

\begin{Thm}
  \label{thm:I=3}
  Consider the maximization problem \eqref{eq:equi_entrop_int},
  and assume that the entropy function $s$ verifies the inequality
  \eqref{eq:non_conv_property}. Then 
  \begin{equation}
    I(\tau,e)<3 \quad \forall (\tau,e)\in D_s.
  \end{equation}
\end{Thm}

\begin{proof}
Assume that $I=3$ and consider a point $X\in D_s$. 
Then $X$ belongs to a simplex of dimension 2. 
On the one hand, inside this simplex, the concave hull of $s$, denoted
$\text{conc}(s)$ is an affine function.
It follows that the partial derivatives of $\text{conc}(s)$,
$\partial_\tau \text{conc}(s)(X)$ and
$\partial_e \text{conc}(s)(X)$
are constant.
On the other hand, at the boundaries of the simplex,
the concave hull $\text{conc}(s)$ is tangent to the surface
$(\tau,e)\mapsto s$. Hence
$\partial_e \text{conc}(s)(X)= \partial_e s(X)$, 
which leads to  a
contradiction with property \eqref{eq:non_conv_property}.
\end{proof}
According to Theorem~\ref{thm:I=3}, 
the maximization process using the van der Waals EoS does not allow
the coexistence of more than two phases and prevents from the
modelling of a triple point.

%---------------------------------------------------------------
\subsection{Maxima of the constrained optimization problem}
%---------------------------------------------------------------

From now on we consider $I=2$ and consider the optimization problem
\begin{equation}
  \label{eq:max_entrop_2}
  \sigma(\tau,e) = \max
  \mathscr{S}(\varphi_1,\varphi_2,\tau_1,\tau_2,e_1,e_2),
\end{equation}
where
\begin{equation}
  \label{eq:S_mix}
  \mathscr{S}(\varphi_1,\varphi_2,\tau_1,\tau_2,e_1,e_2)=\varphi_1
  s(\tau_1,e_1)+\varphi_2s(\tau_2,e_2),
\end{equation}
under the constraints
\begin{equation}
  \label{eq:cons_2}
  \varphi_1+\varphi_2=1, \quad
  \varphi_1\tau_1+\varphi_2\tau_2=\tau,\quad
  \varphi_1e_1+\varphi_2e_2=e.
\end{equation}
Note that if $\rho_1\neq \rho_2$ and $e_1\neq e_2$,
\eqref{eq:cons_2} imply that
the mass fractions $\varphi_i$, $i=1,2$ satisfy
\begin{equation}
  \label{eq:mass-fraction}    
  \varphi_1=
  \dfrac{\tau-\tau_2}{\tau_1-\tau_2}=
  \dfrac{e-e_2}{e_1-e_2}
  , \quad
  \varphi_2=
  \dfrac{\tau-\tau_1}{\tau_2-\tau_1}=
  \dfrac{e-e_1}{e_2-e_1}.
\end{equation}
On the other hand, if $\tau=\tau_1=\tau_2$ and $e=e_1=e_2$, the mass
fraction is undetermined.
In order to preserve the positivity of the fractions, we assume that
\begin{equation}
  \label{eq:H**}   
  (\tau,e) \in [\min(\tau_1,\tau_2), \max(\tau_1,\tau_2)] \times
  [\min(e_1,e_2), \max(e_1,e_2)].
\end{equation}

Introducing the Lagrange multipliers $\lambda_\varphi$, $\lambda_\tau$
and $\lambda_e$ associated to the constraints \eqref{eq:cons_2},
we define the Lagrangian 
\begin{equation}
  \label{eq:Lag}
  L(\lambda_{\varphi},\lambda_{\tau},\lambda_e,x)=
  \mathscr{S} (x)+\lambda_{\varphi} \beta_{\varphi}(x) +\lambda_{\tau}
  \beta_{\tau}(x)+\lambda_e \beta_e(x),
\end{equation}
with
$x=(\varphi_1,\varphi_2,\tau_1,\tau_2,e_1,e_2)$ and 
\begin{equation}
\label{eq:cond_3}
  \begin{cases}
    \beta_{\varphi}(x)=\varphi_{1}+\varphi_{2}-1,\\
    \beta_\tau(x)=\varphi_{1} \tau_{1}+\varphi_{2} \tau_{2}-\tau,\\
    \beta_{e}(x)=\varphi_{1} e_{1}+\varphi_{2} e_{2}-e.
  \end{cases}
\end{equation}
Since $\mathscr{S}$ is $C^1$
and the conditions \eqref{eq:cond_3}
are affine, 
we obtain straightforwardly the optimality conditions for the maxima
in the problem \eqref{eq:max_entrop_2}-\eqref{eq:cons_2}:
\begin{subequations}
  \label{eq:optimality-system}
  \begin{gather}
    \label{eq:optimality-system1}
    s(\tau_{1},e_1)+\lambda_{\varphi}+\lambda_{\tau}\tau_{1}+\lambda_{e}e_1=0,\\
    \label{eq:optimality-system2}
    s(\tau_{2},e_2)+\lambda_{\varphi}+\lambda_{\tau}\tau_{2}+\lambda_{e}e_2=0,\\
    \label{eq:optimality-system3}
    \varphi_{1}\displaystyle \dfrac{p(\tau_1,e_1)}{T
      (\tau_{1},e_1)}+\varphi_{1}\lambda_{\tau}=0,\\
    \label{eq:optimality-system4}
    \varphi_{2}\displaystyle \dfrac{p(\tau_2,e_2)}{T
      (\tau_{2},e_2)}+\varphi_{2}\lambda_{\tau}=0,\\
    \label{eq:optimality-system5}
    \varphi_{1}\displaystyle \dfrac{1}{T(\tau_{1},e_1)}+\varphi_{1}\lambda_{e}=0,\\
    \label{eq:optimality-system6}
    \varphi_{2}\displaystyle \dfrac{1}{T(\tau_{2},e_2)}+\varphi_{2}\lambda_{e}=0.
  \end{gather}
\end{subequations} 
We now turn to the determination of
the maxima of the problem \eqref{eq:max_entrop_2}. 
It turns out that it 
involves the notion of relative entropy, which is defined, for any
two states $a,b\in (\mathbb R^+)^2$ by
\begin{equation}
  \label{eq:relatentrop}
  s(a|b)=s(a)-s(b)-\nabla s(b)\cdot (a-b).
\end{equation}

\begin{Prop}
\label{prop:opti_equi}
The maxima of the problem \eqref{eq:max_entrop_2}-\eqref{eq:cons_2} are
\begin{enumerate}
\item \label{it:pure} \textbf{Identification of phases 1 and 2:}
  \begin{itemize}
  \item$\tau_1=\tau_2=\tau$ and $e_1=e_2=e$, $\varphi_i$ undetermined,
  \item $\varphi_1=0$, $\varphi_2=1$, $(\tau_2,e_2)=(\tau,e)$ and $(\tau_1,e_1)$ solution to 
    \begin{equation}
      \label{eq:pure1}
      \begin{cases}
        s\big((\tau_1,e_1)|(\tau,e)\big) &= 0,\\
        \mu(\tau_1,e_1)/T(\tau_1,e_1) &= \mu(\tau,e)/T(\tau,e).
      \end{cases}
    \end{equation}
  \item$\varphi_1=1$, $\varphi_2=0$, $(\tau_1,e_1)=(\tau,e)$ and $(\tau_2,e_2)$ solution to 
    \begin{equation}
      \label{eq:pure2}
      \begin{cases}
        s\big((\tau_2,e_2)|(\tau,e)\big) &= 0,\\
        \mu(\tau_2,e_2)/T(\tau_2,e_2) &= \mu(\tau,e)/T(\tau,e).
      \end{cases}
    \end{equation}
  \end{itemize}
\item\label{it:mix} \textbf{Saturation states:} 
there exists a unique couple of points\\
\centerline{ $M_1^*=(\tau_1^*,e_1^*, s(\tau_1^*,e_1^*))$ and
$M_2^*=(\tau_2^*,e_2^*,s (\tau_2^*,e_2^*))$}\\
 with $\tau \in [\min(\tau_1^*, \tau_2^*), \max(\tau_1^*,\tau_2^*)]$
and $e \in [\min(e_1^*,e_2^*), \max(e_1^*, e_2^*)]$ given by 
\eqref{eq:mass-fraction}, satisfying
\begin{equation}
  \label{eq:egalite}
  \begin{cases}
    p(\tau_1^*,e_1^*)=p(\tau_2^*,e_2^*),\\
    \mu(\tau_1^*,e_1^*)=\mu(\tau_2^*,e_2^*),\\
    T(\tau_1^*,e_1^*)=T(\tau_2^*,e_2^*),
  \end{cases}
\end{equation}
such that $M=(\tau,e,\sigma(\tau,e))$ belongs to the line segment 
$(M_1^*,M_2^*)=\{zM_1^*+(1-z)M_2^*, \; z\in [0,1]\}$ contained in the
concave hull $\text{conc}(s)$.
\end{enumerate} 
\end{Prop}

\begin{proof}
  The first case $\tau=\tau_1=\tau_2$ and $e=e_1=e_2$ is straightforward.
  We focus on the case $\varphi_1=0$. 
  The mass conservation constraint induces $\varphi_2=1$ and thus
  $\tau_2=\tau$ and $e_2=e$.
  Then the optimality conditions
  \eqref{eq:optimality-system4} and \eqref{eq:optimality-system6} 
  give
  \begin{equation*}
    \lambda_\tau=-p(\tau,e)/T(\tau,e),
    \quad
    \lambda_e=-1/T(\tau,e).
  \end{equation*}
  Associated with the conditions
  \eqref{eq:optimality-system1}
  and   \eqref{eq:optimality-system2},
  the definition of the relative entropy \eqref{eq:relatentrop} and
  the definition of the chemical potential \eqref{eq:gibbs-intensive},
  one determines $(\tau_1,e_1)$ as the solution of \eqref{eq:pure1}.
  The same holds for the case $\varphi_1=1$.

  We now consider the saturation case. It is characterized by
  $\varphi_1\varphi_2\neq 0$.
  The optimization procedure also reads as a convexification of
  $(\tau,e)\mapsto s(\tau,e)$ in the sense that the graph of
  $(\tau,e)\mapsto \sigma(\tau,e)$ is the concave hull of
  $(\tau,e)\mapsto s(\tau,e)$, see the definition \eqref{eq:conv}.
  Then for any saturation state $(\tau,e)$, the graph  of $\sigma$
  contains a segment $(M_1^*M_2^*)$ passing through $(\tau,e,\sigma(\tau,e))$.
  The characterization \eqref{eq:egalite} of the points $M_1^*$ and
  $M_2^*$ derives from the optimality conditions.
  Combining \eqref{eq:optimality-system5} and
  \eqref{eq:optimality-system6} gives the temperatures equality
  \begin{equation*}
    \dfrac{1}{T(\tau_1,e_1)} =  \dfrac{1}{T(\tau_2,e_2)}.
  \end{equation*}
    Similarly using \eqref{eq:optimality-system3} and \eqref{eq:optimality-system4}, 
  yields
  \begin{equation*}
    \dfrac{p(\tau_1,e_1)}{T(\tau_1,e_1)} = \dfrac{p(\tau_2,e_2)}{T(\tau_2,e_2)}.
  \end{equation*}
  Finally \eqref{eq:optimality-system1} and \eqref{eq:optimality-system2},
  combined with the definition of the chemical potential
  \eqref{eq:gibbs-intensive}, give
  \begin{equation*}
    \dfrac{\mu(\tau_1,e_1)}{T(\tau_1,e_1)} = \dfrac{\mu(\tau_2,e_2)}{T(\tau_2,e_2)}.
  \end{equation*}
  We now address the uniqueness of the segment $(M_1^*M_2^*)$.
  Outside the spinodal zone, the van der Waals entropy is a concave
  and increasing function with respect to $\tau$ and $e$. Then there
  is a bijection between $(p,T)$ and $(\tau_i,e_i)$, $i=1,2$. Define
  $\widetilde M_1^*=(\widetilde \tau_1^*,\widetilde e_1^*)$ 
  and   $\widetilde M_2^*=(\widetilde \tau_2^*,\widetilde e_2^*)$.
  If $(\tau,\sigma(\tau,e))\in (M_1^*M_2^*)\cap(\widetilde
  M_1^*\widetilde M_2^*)$, since $(p,T,\mu)$ are constant along
  $(M_1^*M_2^*)$ and $(\widetilde M_1^*\widetilde M_2^*)$, then
  $M_i^*=\widetilde M_i^*$ and the segments coincide.
\end{proof}

Notice that, for a given saturation state $(\tau,e)$, the quadruplet 
$(\tau_1^*,\tau_2^*,e_1^*,e_2^*)$ satisfies also
\begin{equation}
  \label{eq:coexistence_relat}
  s((\tau_1^*,e_1^*)|(\tau_2^*,e_2^*)) 
=   s((\tau_2^*,e_2^*)|(\tau_1^*,e_1^*)) 
=0.
\end{equation}

We emphasize that the necessary conditions in Proposition
\ref{prop:opti_equi}
include all equilibrium states, regardless of their stability.
In particular we recover in item \eqref{it:pure} the complete van der
Waals EoS, including physically unstable states (spinodal zone),
and liquid and vapor metastable and stable states.

To proceed further the classical method consists in studying the local
concavity of the mixture entropy out equilibrium $\mathscr{S}$ introduced
in \eqref{eq:S_mix}. We adopt here the approach proposed in
\cite{james}.
We introduce a relaxation towards the equilibrium states by means of a
dynamical system.

%-------------------------------------------------------
\section{Dynamical system and attraction bassins}
%-------------------------------------------------------
\label{sec:syst_dyn}

The goal of this section is to introduce time dependence to create a
dynamical system 
able to  characterize all the equilibrium states including the 
metastable states.
To build the appropriate dynamical system, we impose two basic
criteria:
\begin{itemize}
\item long-time equilibria coincide with the maxima given by the
  optimality conditions in Proposition~\ref{prop:opti_equi}.
\item the mixture entropy increases along trajectories.
\end{itemize}

Fix $(\tau,e)$ a state vector of the system.
The maximization problem 
applies to six variables under the three constraints~\eqref{eq:cons_2}.
Hence it is sufficient to  reduce the variables from six to three.
We consider the vector of volume, mass and energy fractions 
$\mathbf r =(\alpha,\varphi,\xi)$.
Then the phasic specific energies and volumes are now functions of
$\mathbf r =(\alpha,\varphi,\xi)\in ]0,1[^3$ with
\begin{equation}
  \label{eq:volumes_energies_fractions}
  \begin{aligned}
    \tau_1(\mathbf r)&=\dfrac{\alpha \tau}{\varphi},\quad
    \tau_2(\mathbf r)=\dfrac{(1-\alpha)\tau}{1-\varphi},\\ 
    e_1(\mathbf r)&=\dfrac{\xi e}{\varphi},\quad
    e_2(\mathbf r)=\dfrac{(1-\xi)e}{1-\varphi}.
\end{aligned}
\end{equation}
The formulas in \eqref{eq:volumes_energies_fractions} do not suggest
any natural order in the volumes nor energies. Besides it is possible that the
phasic specific volumes (resp. energies) coincide. Indeed, if $\alpha=\varphi=\xi\in
]0,1[$, then $\tau_1(\mathbf{r})=\tau_2(\mathbf{r})=\tau$ and
$e_1(\mathbf{r})=e_2(\mathbf{r})=e$.
Hence the constraint
\eqref{eq:H**} still remains.

In this context, the mixture entropy of the system becomes a function
of $\mathbf r$, still denoted $\mathscr{S}$:
\begin{equation}
  \label{eq:Sr}
  \mathscr{S}(\mathbf r) = \varphi s(\tau_1(\mathbf r),e_1(\mathbf r))
  +(1-\varphi)s(\tau_2(\mathbf r),e_2(\mathbf r)).
\end{equation}
Using the relations \eqref{eq:volumes_energies_fractions} and
the expressions \eqref{eq:intensive-form}  of the partial
derivatives of the entropy function,
the gradient of $\mathcal S$ reads
\begin{equation}
\label{eq:nabla_S}
  \nabla_\mathbf r  \mathscr{S} (\mathbf r)=
  \begin{pmatrix}
    \tau\dfrac{ p(\tau_1 (\mathbf r),e_1 (\mathbf r))}{T(\tau_1 (\mathbf r),e_1 (\mathbf r))} - \tau \dfrac{
      p(\tau_2 (\mathbf r),e_2 (\mathbf r))}{T(\tau_2 (\mathbf r),e_2 (\mathbf r))}\\
    -\dfrac{\mu(\tau_1 (\mathbf r),e_1 (\mathbf r))}{T(\tau_1 (\mathbf r),e_1 (\mathbf r))}+
    \dfrac{\mu(\tau_2 (\mathbf r),e_2 (\mathbf r))}{T(\tau_2 (\mathbf r),e_2 (\mathbf r))}\\
    \dfrac{e}{T(\tau_1 (\mathbf r),e_1 (\mathbf r))}- \dfrac{e}{T(\tau_2 (\mathbf r),e_2 (\mathbf r))}
  \end{pmatrix}.
\end{equation}
Observe that both $\mathscr{S}$ and $\nabla_\mathbf r  \mathscr{S}$
are defined only for $(\alpha,\varphi,\xi)\in ]0,1[^3$.

We wish to construct a dynamical system which complies with the
entropy growth criterion in the sense that
entropy increases along the trajectories \textit{i.e.}
$\text{d}/\text{dt}\mathscr{S}(\mathbf{r}(t))\geq 0$.
A naive choice is to choose $\dot{\mathbf r}$ close to $\nabla_\mathbf
r  \mathscr{S}$. We introduce the
following dynamical system:
\begin{equation}
  \label{eq:fraction-model}
  \begin{cases}
    \dot{\alpha}(t)=\alpha(1-\alpha) \tau\Big(
    \dfrac{p(\tau_1(\mathbf{r}),e_1(\mathbf{r}))}{T(\tau_1(\mathbf{r}),e_1(\mathbf{r}))} - 
    \dfrac{p(\tau_2(\mathbf{r}),e_2(\mathbf{r}))}{T(\tau_2(\mathbf{r}),e_2(\mathbf{r}))}\Big),\\
    \dot{\varphi}(t)=\varphi (1-\varphi) \Big(
    \dfrac{\mu(\tau_2(\mathbf{r}),e_2(\mathbf{r}))}{T(\tau_2(\mathbf{r}),e_2(\mathbf{r}))}-
    \dfrac{\mu(\tau_1(\mathbf{r}),e_1(\mathbf{r}))}{T(\tau_1(\mathbf{r}),e_1(\mathbf{r}))} \Big),\\
    \dot{\xi}(t)=\xi(1-\xi)e\Big(
    \dfrac{1}{T(\tau_1(\mathbf{r}),e_1(\mathbf{r}))}-
    \dfrac{1}{T(\tau_2(\mathbf{r}),e_2(\mathbf{r}))}\Big).
\end{cases}
\end{equation}

\begin{Prop}
  \label{prop:sys_dyn_prop}
  The dynamical system \eqref{eq:fraction-model} satisfies the
  following properties.
  \begin{enumerate}
  \item \label{it:prop1} If $\mathbf{r}(0)\in ]0,1[^3$ then, for all time
    $t>0$, one has
    $\mathbf{r}(t) \in ]0,1[^3$.
  \item \label{it:prop2} If the condition \eqref{eq:H**} is satisfied at $t=0$, 
    then $(\tau,e)$ belongs to the segment
    $[(\tau_1(\mathbf{r}) ,e_1(\mathbf{r}))(t),
    (\tau_2(\mathbf{r}),e_2(\mathbf{r}))(t)]$ for all time $t>0$.
  \item \label{it:prop3} The mixture entropy increases along the trajectories 
    \begin{equation}
      \label{eq:Sdot}
      \dfrac{d}{dt}\mathscr{S}(\mathbf r(t))\geq 0.
    \end{equation}
  \end{enumerate}
\end{Prop}

\begin{proof}
  The multiplicative term $\alpha(1-\alpha)$ in the first equation of
  \eqref{eq:fraction-model} ensures that the right-hand side vanishes
  if $\alpha=0$ or $\alpha=1$. The same holds for the equations on
  the remaining fractions $\varphi$ and $\xi$. This proves
  \eqref{it:prop1}.
Item \eqref{it:prop2} is nothing but a reformulation of \eqref{eq:mass-fraction}.
Finally, the time derivative of the mixture entropy writes
  \begin{equation*}
    \dfrac{d}{dt}\mathscr{S}(\mathbf r(t))
    =
    \nabla \mathscr{S}(\mathbf r)\cdot  \dot{\mathbf r}(t).
  \end{equation*}
  Since $\tau>0$, $e>0$ and the fractions belong to $]0,1[$, it
  follows that the item
  \eqref{it:prop3} holds true.
\end{proof}

%-------------------------------------------------------
% Attraction bassins
%-------------------------------------------------------
\subsection{Equilibria and attractivity}
\label{sec:equil-attr}

In the sequel, we let $\mathbb F(\mathbf r)=(\mathbb F^\alpha,\mathbb
F^\varphi, \mathbb F^\xi)$ be the right-hand side of
\eqref{eq:fraction-model}, such that
\begin{equation}
  \label{eq:RHS}
  \begin{cases}
        \dot{\alpha}(t) = \mathbb F^\alpha(\mathbf{r}),\\
        \dot{\varphi}(t) = \mathbb F^\varphi(\mathbf{r}),\\
        \dot{\xi}(t) = \mathbb F^\xi(\mathbf{r}).
  \end{cases}
\end{equation}

\begin{Prop}[Equilibrium states]
\label{prop:equilibrium_states}
The equilibrium states for the dynamical system
\eqref{eq:fraction-model} are 
\begin{enumerate}
\item\label{it:coexistence} saturation states: 
  either $\mathbf
  r^*=(\alpha^*,\varphi^*,\xi^*)$ or 
  $\mathbf{r}^\#= (1-\alpha^*,
  1-\varphi^*, 1-\xi^*)$, 
  with $\alpha^*\neq\varphi^*\neq\xi^*\in
  ]0,1[$, defined by~\eqref{eq:volumes_energies_fractions} such that 
  $\tau_i^*=\tau_i(\mathbf r^*)= \tau_i(\mathbf{r}^\#)$ 
  and $e_i^*=e_i(\mathbf r^*)=e_i(\mathbf{r}^\#)$,
  $i=1,2$, corresponding
  to the characterization~\eqref{eq:egalite} of
  Proposition~\ref{prop:opti_equi}-~\eqref{it:mix}.
\item\label{it:pureq} Identification of phases 1 and 2: $\overline{\mathbf r}=(\beta,\beta,\beta)$,
  $\beta\in]0,1[$ such that
  $\tau_1(\overline{\mathbf{r}})=\tau_2(\overline{\mathbf{r}})=\tau$ and 
  $e_1(\overline{\mathbf{r}})=e_2(\overline{\mathbf{r}})=e$.
\end{enumerate}
\end{Prop}

The equilibria of the dynamical system are
given by $\mathbb F(\mathbf r)=0$.
In the case of equilibria \eqref{it:coexistence}, consider that
$\alpha\neq\varphi\neq \xi$. Then, according to the
Proposition~\ref{prop:opti_equi}-~\eqref{it:mix},
there exists a unique triplet $r^*=(\alpha^*,\varphi^*,\xi^*)$ such
that the characterization~\eqref{eq:egalite} holds. 
It turns out that
$\mathbf{r}^\#$ is also an equilibrium of the system.
If the equilibrium
$r^* $ corresponds to $\tau_1^*<\tau<\tau_2^*$, $e_1^*<e<e_2^*$,
then the equilibrium $\mathbf{r}^\#$ corresponds to
$\tau_1^*>\tau>\tau_2^*$ and $e_1^*>e>e_2^*$ and conversely.
In the case of equilibria \eqref{it:pureq}, the two phases coincide, in
the sense that $\tau_1=\tau_2=\tau$ and $e_1=e_2=e$.
The determination of the constant $\beta\in  ]0,1[$ depends on the
initial data of the dynamical system \eqref{eq:fraction-model}.
We emphasize that the equilibrium states $\mathbf r
=(\beta,\beta,\beta)$ are valid for all states $(\tau,e)$ and go over
the van der Waals surface. 

To go further and identify the physically admissible equilibrium
states, we must investigate their stability and attractivity.

\begin{Prop}[Attractivity]
The equilibrium states are classified as follow:
\begin{itemize}
\item The saturation states $\mathbf{r}^*=(\alpha^*,\varphi^*,\xi^*)$
  and $\mathbf{r}^\#=(1-\alpha^*,1-\varphi^*,1-\xi^*)$
  are attractive points,
\item The equilibrium $\overline{\mathbf{r}} =
  (\beta,\beta,\beta)\in ]0,1[$, corresponding to the identification
  of the two phases, is strongly degenerate.
\end{itemize}
\end{Prop}

\begin{proof}
In the sequel and for sake of
readability, we denote $p_i:=p(\tau_i(\mathbf{r}),e_i(\mathbf{r}))$,
$T_i:=T(\tau_i(\mathbf{r}),e_i(\mathbf{r}))$ and
$\mu_i:=\mu(\tau_i(\mathbf{r}),e_i(\mathbf{r}))$.
The goal now is to find the spectrum of the Jacobian  matrix  of $\mathbb{F}$ denoted by 
\begin{equation*}
  %\label{eq:DF}
  D_\mathbf{r} \mathbb F(\mathbf{r}):=
  \begin{pmatrix}
    \p_\alpha \mathbb F^\alpha(\mathbf{r}) & 
    \p_\varphi \mathbb F^\alpha(\mathbf{r}) & 
    \p_\xi \mathbb F^\alpha(\mathbf{r}) \\
    \p_\alpha \mathbb F^\varphi(\mathbf{r}) & 
    \p_\varphi \mathbb F^\varphi(\mathbf{r}) & 
    \p_\xi \mathbb F^\varphi(\mathbf{r}) \\
    \p_\alpha \mathbb F^\xi(\mathbf{r}) & 
    \p_\varphi \mathbb F^\xi(\mathbf{r}) & 
    \p_\xi \mathbb F^\xi(\mathbf{r}) 
  \end{pmatrix}.
\end{equation*}
First consider the equilibrium $
\overline{\mathbf{r}}=(\beta,\beta,\beta)\in ]0,1[^3$, which
corresponds to the identification of the two phases.
In that case, the Jacobian matrix $D_\mathbf{r} \mathbb F(\overline{\mathbf{
r}})$ reads
\begin{equation}
  \label{eq:JF_rbar}
  D_\mathbf{r} \mathbb F(\mathbf{\overline{r}})=
  \begin{pmatrix}
    \tau \,\partial_\tau ( \dfrac{p}{T})(\tau,e) & 
    -\tau \,\partial_\tau ( \dfrac{p}{T})(\tau,e)
    -e\, \partial_e ( \dfrac{p}{T})(\tau,e) & 
    e\, \partial_e (  \dfrac{p}{T})(\tau,e)\\
    -\tau \, \partial_\tau (  \dfrac{\mu}{T})(\tau,e) & 
    \tau \, \partial_\tau (  \dfrac{\mu}{T})(\tau,e)
    +e\, \partial_e( \dfrac{\mu}{T})(\tau,e) &
    - e\, \partial_e( \dfrac{\mu}{T})(\tau,e) \\
    \tau \, \partial_\tau (  \dfrac{1}{T})(\tau,e) &
    -\tau \, \partial_\tau (  \dfrac{1}{T})(\tau,e)-e\, \partial_e (
    \dfrac{1}{T})(\tau,e)& 
    e\, \partial_e (  \dfrac{1}{T})(\tau,e)
\end{pmatrix}.
\end{equation}
Since the middle column is the sum of the two remaining columns, then
the determinant of $D_\mathbf{r} \mathbb F(\mathbf{\overline{r}})$ is
zero and the Jacobian matrix admits a null eigenvalue. Hence
the equilibrium $\overline{\mathbf{r}}$ is a  strongly degenerate.

As the 
saturation equilibrium $\mathbf{r}^*=(\alpha^*,\varphi^*,\xi^*)$ is
concerned,
the coefficients of the Jacobian matrix $D_\mathbf{r} \mathbb
F(\mathbf{{r}}^*)$ do not simplify much and obtaining an explicit
formulation of its eigenvalues is out of reach.
So, we turn to the numerical illustration of the spectrum
$\{\lambda_1,\lambda_2,\lambda_3 \} $ 
of the matrix $D_\mathbf{r} \mathbb
F(\mathbf{{r}}^*)$ for some saturation states
$\mathbf{r}^*=(\alpha^*,\varphi^*,\xi^*)$ with the van der Waals EoS
with parameters \eqref{eq:vdw_param}.

$$\begin{tabular}{|c|c|c|c|c|c|c|c|c|c|c|c|}
\hline
$\mathbf{r}^*$& $(\tau,e)$ &$(\tau_1^*,e_1^*)$& $(\tau_2^*,e_2^*)$ &$\lambda_1$ &
$\lambda_2$ & $\lambda_3$   \\
\hline
(0.71, 0.29, 0.39) & (1.99, 2.1)&(4.76, 2.80)& (0.82, 1.80) & -8.443 & -1.290 &-0.061\\
\hline
(0.94, 0.73, 0.82) & (3.9, 2.49)&(5.13, 2.77) &(0.81, 1.73)& -5.713& 2.048& -0.055  \\
\hline
(0.76, 0.22, 0.35) & (2.39, 1.59)&(8.23, 2.56) & (0.73, 1.32) & -8.477 & -2.835 & -0.110  \\
\hline
(0.64,  0.14, 0.24) & (1.79, 1.49)&(8.25, 2.56) &  (0.73, 1.32)
 &-9.044& -2.405 & -0.097\\
\hline
(0.68, 0.25, 0.35) & (1.89, 1.99)&(5.11, 2.77) & (0.81, 1.73) & -8.660& -1.368 & -0.065 \\
\hline
\end{tabular}
$$
One observes numerically that, for these saturation equilibria
$\mathbf{r}^*$,  the Jacobian matrix $ D_{\mathbf{r}} \mathbb F
(\mathbf{r}^*)$ admits three negative eigenvalues, which means that
these equilibria are attractive. The same hold true for the
equilibrium $\mathbf{r}^\#$.
\end{proof}

To complete the study of equilibrium states, in particular to cope
with the degenerate state $\overline{\mathbf{r}}$, corresponding to
the identification of the two phases, we investigate
the attraction basins of $\overline{\mathbf{r}}$, $\mathbf{r}^*$ and $\mathbf{r}^\#$.
We introduce the following functions with index $I$ for Identification and
$S$ for Saturation:
\begin{equation}
  \label{eq:lyap}
  \begin{aligned}
    G_S(\mathbf{r}) &=  -\mathscr{S}(\mathbf{r}) + \text{conc}(s)(\tau,e),\\
    G_I(\mathbf{r}) &= -\mathscr{S}(\mathbf{r}) + s(\tau,e),
  \end{aligned}
\end{equation}
where $(\text{conc}s)(\tau,e)$ refers to the concave hull of the
function $s$, see the definition \eqref{eq:conv}.
\begin{Prop}
  \label{prop:lyap}
  The basins of attraction of the equilibrium states are the
  following:
  \begin{itemize}
  \item In the spinodal zone, with $(\tau,e)\in Z_{\text{Spinodal}}$,
    $G_S$ is a Lyapunov function on the whole domain
    $(\alpha,\varphi,\xi)\in ]0,1[^3$.
  \item In the liquid or vapor stable zones, with $(\tau,e)\in
    Z_{\text{Stable}}$, $G_I$ is a Lyapunov function of the whole
    domain $(\alpha,\varphi,\xi)\in ]0,1[^3$.
  \end{itemize}
\end{Prop}

\begin{proof}
  The two functions are candidate to 
  be a Lyapunov function, since
  \begin{itemize}
  \item by construction $G_S(\mathbf{r}^*)=0$.
    Indeed, denoting
    $p^* = p(\tau_1^*,e_1^*)=p(\tau_1^*,e_2^*)$, 
    $T^* = T(\tau_1^*,e_1^*)=T(\tau_1^*,e_2^*)$, 
    and $\mu^* = \mu(\tau_1^*,e_1^*)=\mu(\tau_1^*,e_2^*)$,
    one has
    \begin{equation}
      \label{eq:GC_lyap1}
      \begin{aligned}
        G_S(\mathbf{r}^*) &= -\varphi^*s(\tau_1^*,e_1^*)
        -(1-\varphi^*)s(\tau_2^*,e_2^*)
        + p^*/T^* \tau +
        e/T^* + \mu^*/T^*\\
        & = p^*/T^* (\tau -\varphi^*\tau_1^* -(1-\varphi^*)\tau_2^*) +
        1/T^* (e-\varphi^* e_1^* - (1-\varphi^*) e_2^*) \\
        &\quad+ \mu^*/T^* (1-\varphi^*-(1-\varphi^*))\\
        &=0,
      \end{aligned}
    \end{equation}
    using the Gibbs relation \eqref{eq:gibbs-intensive}.
    The same holds for the equilibrium $\mathbf{r}^\#$.
    Similarly $G_I(\overline{\mathbf{r}})=0$;
  \item it holds $\nabla_\mathbf{r}
    G_S(\mathbf{r})=\nabla_\mathbf{r}
    G_I(\mathbf{r})=-\nabla_\mathbf{r}\mathscr{S}(\mathbf{r})$. Then we
    obtain as well $\nabla_\mathbf{r}
    G_S(\mathbf{r}^*)= \nabla_\mathbf{r}
    G_S(\mathbf{r}^\#)=\nabla_\mathbf{r}  G_S(\overline{\mathbf{r}})=0$,
    according to \eqref{eq:nabla_S};
  \item for the same reason, and using \eqref{eq:Sdot}, we have
    \begin{equation*}
      \dfrac{\text{d}}{\text{dt}}G_S(\mathbf{r}(t))=
      \dfrac{\text{d}}{\text{dt}}G_I(\mathbf{r}(t))
      \leq 0.
    \end{equation*}
\end{itemize}
It remains to check the positivity of $G_S$ and $G_I$ in a neighborhood of
$\mathbf{r}^*$, $\mathbf{r}^\#$ and $\overline{\mathbf{r}}$ respectively, depending on
the domain the state $(\tau,e)$ belongs to.

\noindent\textbf{Saturation} with $(\tau,e)\in Z_{\text{Spinodal}}$. By definition of
$\text{conc}(s)$, $\text{conc}(s)(\tau,e) > \mathscr S(\mathbf{r})$
for $\mathbf{r}\neq
\mathbf{r}^*$ (or equivalently $\mathbf{r}\neq
\mathbf{r}^\#$) . Hence $G_S(\mathbf{r})>0$.

\noindent\textbf{Stable states} with $(\tau,e)\in
Z_{\text{Stable}}$. We make use again of the concave hull of $s$
\begin{equation*}
  \begin{aligned}
    G_I(\mathbf{r}) &= -\varphi s\left(\alpha/\varphi \tau,\xi/\varphi
      e \right) - (1-\varphi) s \left( (1-\alpha)/(1-\varphi) \tau,
      (1-\xi)/(1-\varphi) e \right) - s(\tau,e)\\
    &\geq - \varphi \text{conc }(s)(\alpha/\varphi\tau,\xi/\varphi e)
    \\
    &\quad -
    (1-\varphi) \text{conc
    }(s)((1-\alpha)/(1-\varphi)\tau,(1-\xi)/(1-\varphi) e)\\
    &\quad -s(\tau,e)\\
    &\geq -  \text{conc }(s)(\alpha/\varphi\tau +
    (1-\alpha)/(1-\varphi)\tau,\xi/\varphi e+ (1-\xi)/(1-\varphi)
    e)-s(\tau,e)\\
    &= -  \text{conc }(s)(\tau,e)-s(\tau,e).
  \end{aligned}
\end{equation*}
In the liquid or vapor stable zones, $(\tau,e)$ belongs to the convex hull of the
graph of $s$, that is $\text{conc}(s)(\tau,e)=s(\tau,e)$.
Then $G_I(\mathbf{r}) \geq 0$ and the equality occurs if
$\mathbf{r}=\overline{\mathbf{r}}$.
\end{proof}
When considering the metastable regions with $(\tau,e)\in
Z_{\text{Metastable}}$, there are two basins of attraction, numerically
illustrated in Section \ref{sec:metast-zone}, see Figure
\ref{fig:meta_phase}.
Unlike in the spinodal zone, the function $G_I$ is non-negative in a
neighborhood of $(\tau,e)$, provided that $(\tau,e)$ belongs to a zone
of strict concavity of $s$. It means that both $\overline{\mathbf{r}}$
and $\mathbf{r}^*$ are reachable. The two basins of attraction are
separated by an unstable manifold, which is difficult to determine
theoretically and numerically as well.
It is already tough in the  isothermal framework, see \cite{james},
\cite{ghazi:tel-01976189} and\cite{ghazi19}.
In the latter reference  the determination of the basins of
the metastable states is more precise, even if it is not
explicit, the basins being defined through the application of the
implicit function theorem.

%------------------------------------------------------------------------------
\subsection{Numerical illustrations}
\label{sec:num_dyn_sys}
%------------------------------------------------------------------------------

This section provides numerical simulations to illustrate the
behavior of the dynamical system \eqref{eq:fraction-model} and the
attraction of each possible equilibrium states studied in Propositions
\ref{prop:equilibrium_states} and \ref{prop:lyap}.

The computations correspond to the reduced van der Waals EoS, with
parameters \eqref{eq:vdw_param}.
Cauchy problems for the system \eqref{eq:fraction-model} are solved
using a BDF method for stiff
problems available in the Python ODE-solver package.
The numerical results are computed
for a large computational time  $T_f = 200$s.
For each test case, the state $(\tau,e)$ of the total system
is picked either in the spinodal zone $Z_\text{Spinodal}$, in 
stable zones $Z_\text{Stable}$ 
or in a metastable
zones $Z_\text{Metastable}$, as depicted in Figure \ref{fig:vdw_tau_e}.
We provide the associated vector field in the $(\alpha,
\varphi,\xi)$ phase space and plot some trajectories in the phase
space starting from arbitrary initial state $\mathbf{r}(0)$ in order to
illustrate the attractivity of the equilibria.
Several complementary trajectories are represented in the planes $(\tau,e)$ and
$(\tau,p)$.

%------------------------------------------------------------------------------
\subsubsection{Spinodal zone}
\label{sec:spinodal-zone}
%------------------------------------------------------------------------------
The purpose is to illustrate the fact that, for any initial data
$\mathbf{r}(0)\in ]0,1[^3$, if the state $(\tau,e)$ belongs to the
spinodal zone $Z_\text{Spinodal}$, the corresponding attraction points are either
$\mathbf{r}^*$ or $\mathbf{r}^\#$,
that is the system achieves a saturation state 
of the saturation dome, see Proposition \ref{prop:lyap}.

We consider the state $(\tau,e)=(2,2.5)$ belonging to the spinodal zone.
The vector field of the dynamical system \eqref{eq:fraction-model} is
represented in Figure \ref{fig:spinodal_phase} by light blue
arrows. For some random initial conditions $\mathbf{r}(0)\in
]0,1[^3$ (representing by green or yellow dots), 
the corresponding trajectories converge either towards the
point $\mathbf{r}^*=(\alpha^*,\varphi^*,\xi^*)$ (green lines
converging towards the green
star) or towards
$\mathbf{r}^\#=(1-\alpha^*,1-\varphi^*,1-\xi^*)$ (yellow lines
converging towards the
yellow star). In both case, the asymptotic state corresponds to the
unique state $(\tau_i^*, e_i^*)$, $i=1,2$, defined by
\eqref{eq:egalite}, which belongs to the
saturation dome, see Proposition \ref{prop:opti_equi}-\eqref{it:mix}.
\begin{figure}[htpb]
  \centering
  \includegraphics[width=1.\linewidth]{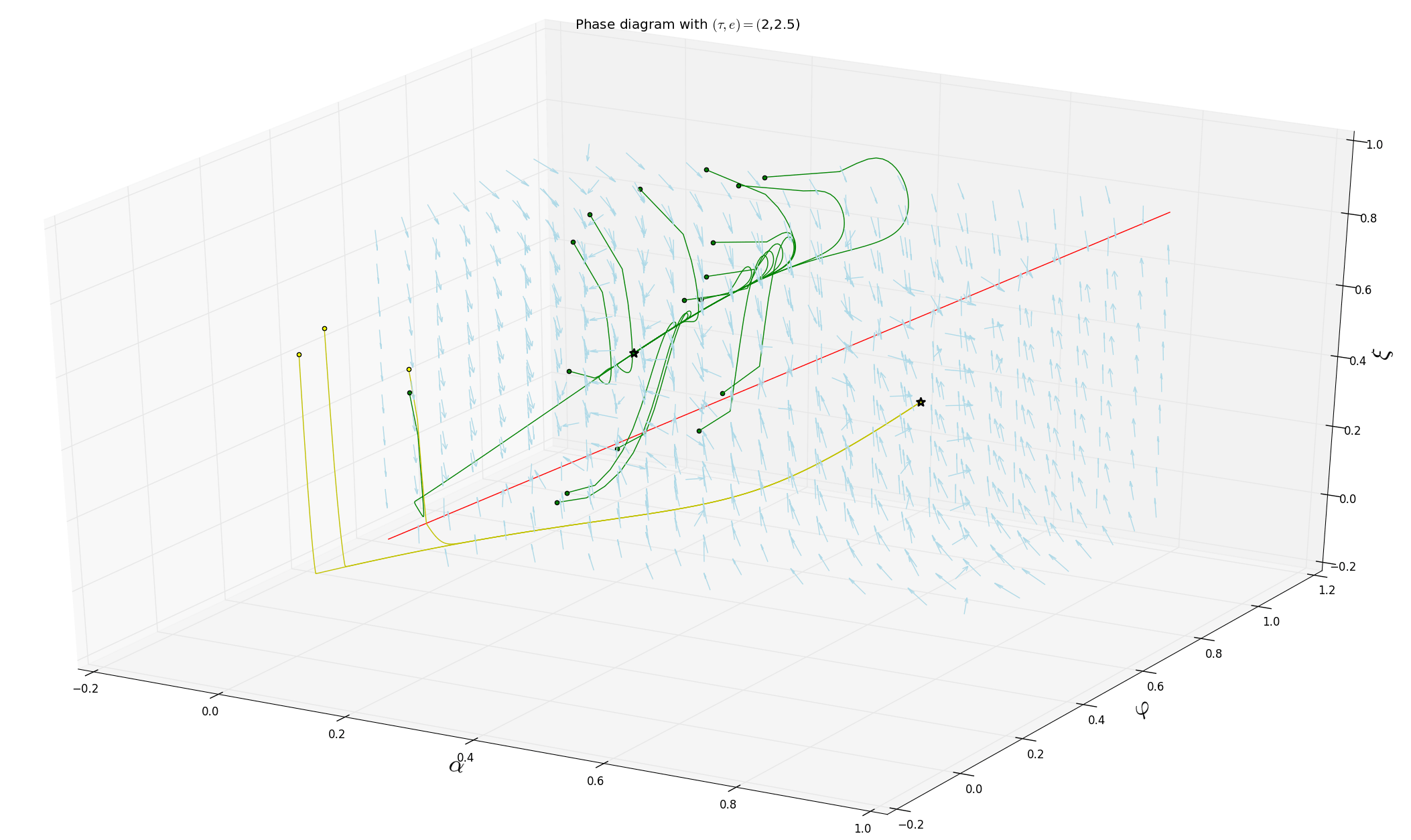}
  \caption{Spinodal zone: vector field of the dynamical system
    \eqref{eq:fraction-model} (light blue arrows). The red line
    corresponds to the line $\alpha=\varphi=\xi$.
    Depending on the
    initial condition $\mathbf{r}(0)$, the trajectories converge
    either towards the equilibrium
    $\mathbf{r}^*=(\alpha^*,\varphi^*,\xi^*)$ (green lines) or towards
    $\mathbf{r}^\#=(1-\alpha^*,1-\varphi^*,1-\xi^*)$ (yellow
    lines). In both case, the asymptotic regime corresponds to the
    state
    $(\tau_i^*, e_i^*)$, $i=1,2$, defined by
    \eqref{eq:egalite},
    belonging to the saturation dome.}
\label{fig:spinodal_phase}
\end{figure}

\begin{figure}[htpb]
  \centering
  \includegraphics[width=\linewidth]{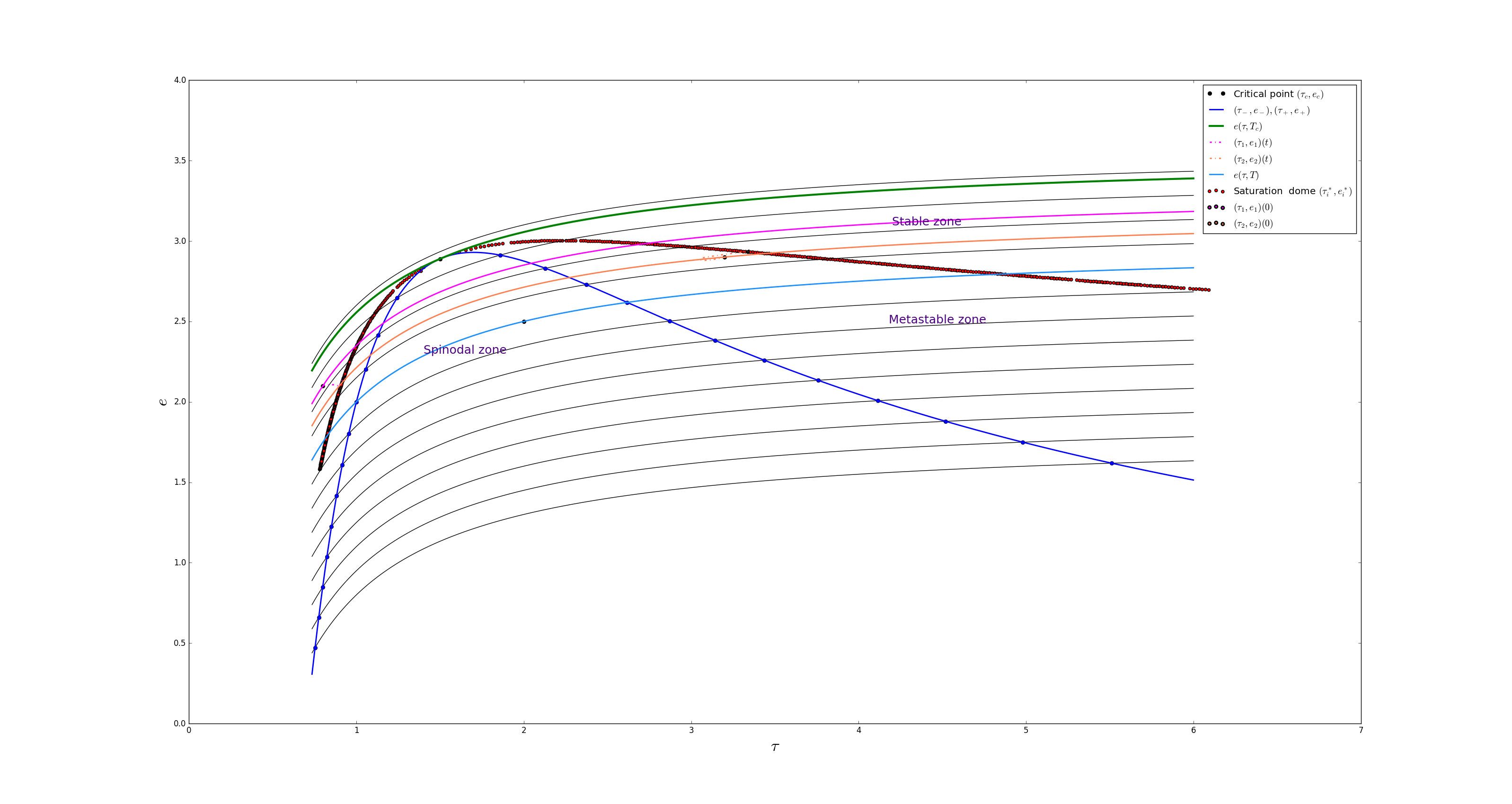}
  \includegraphics[width=0.8\linewidth]{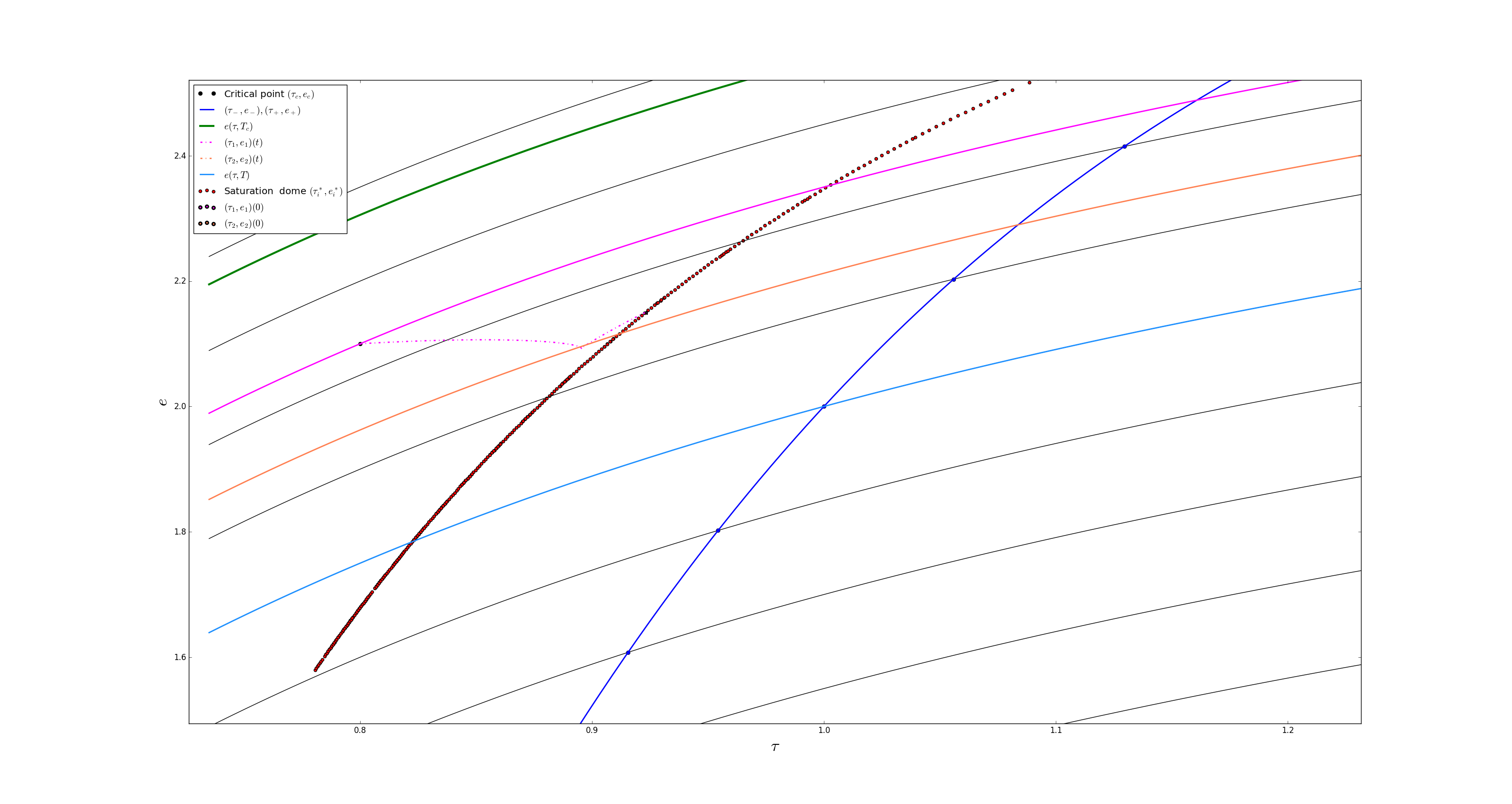}
  \includegraphics[width=0.8\linewidth]{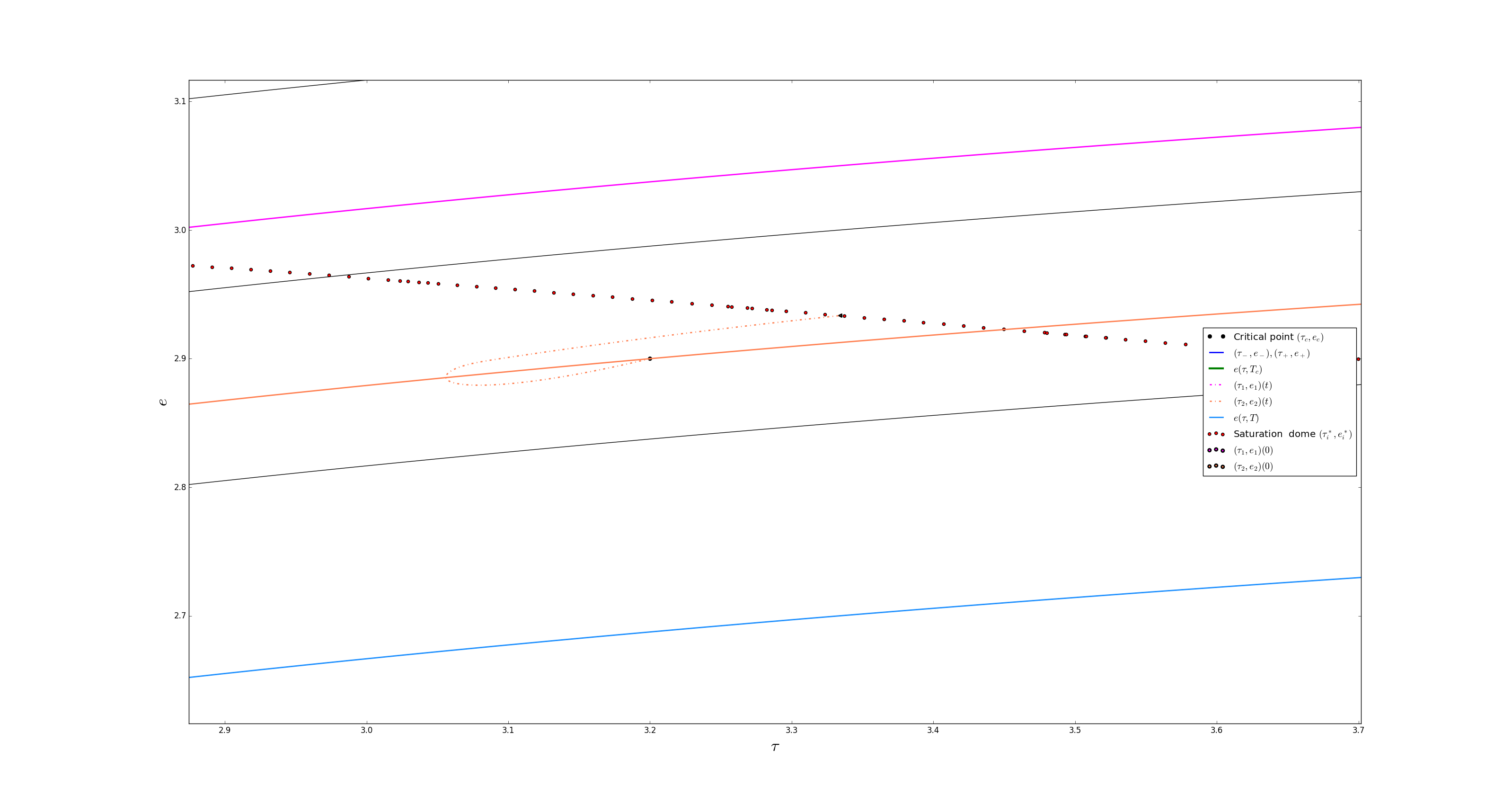}
  \caption{Spinodal zone, from top to bottom. Trajectories of the dynamical system
    \eqref{eq:fraction-model} in the $(\tau,e)$ plane. Starting from an
    initial state $(\tau_1(\mathbf{r}), e_1(\mathbf{r}))(0)$
    in the stable liquid region (on the magenta isothermal
    curve),
    the trajectory $(\tau_1(\mathbf{r}),
    e_1(\mathbf{r}))(t)$ is represented with a dashed magenta line and converges
    towards the saturation dome. The
    trajectory $(\tau_2(\mathbf{r}),
    e_2(\mathbf{r}))(t)$ is represented in orange.
    Middle and bottom figures: zoom of trajectories $(\tau_1(\mathbf{r}),
    e_1(\mathbf{r}))(t)$ and $(\tau_2(\mathbf{r}),
    e_2(\mathbf{r}))(t)$ respectively.}
  \label{fig:spinodal_etat1liqpur_tau_e}
\end{figure}
 
\begin{figure}[htpb]
  \centering
  \includegraphics[width=\linewidth]{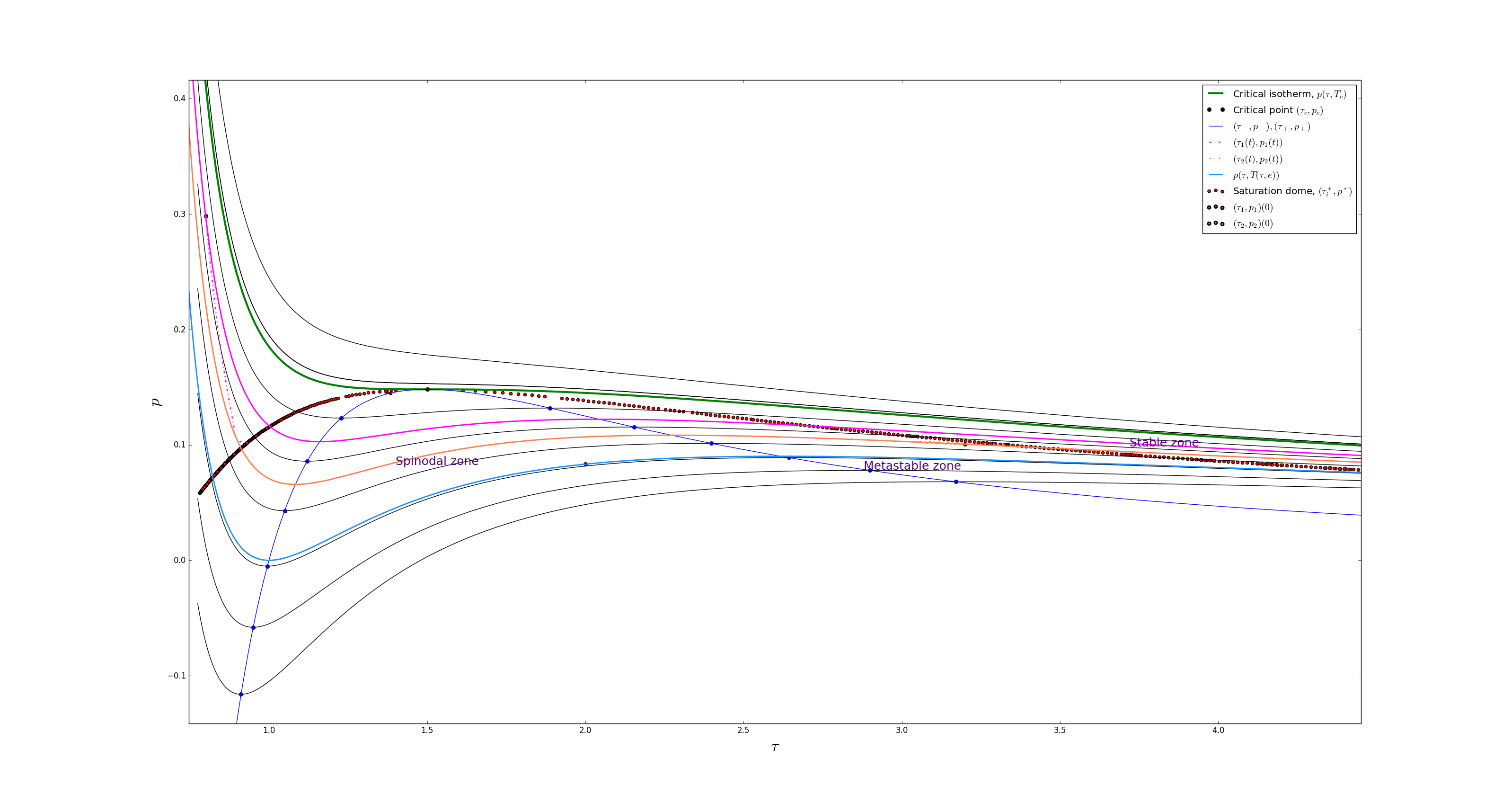}
  \includegraphics[width=0.8\linewidth]{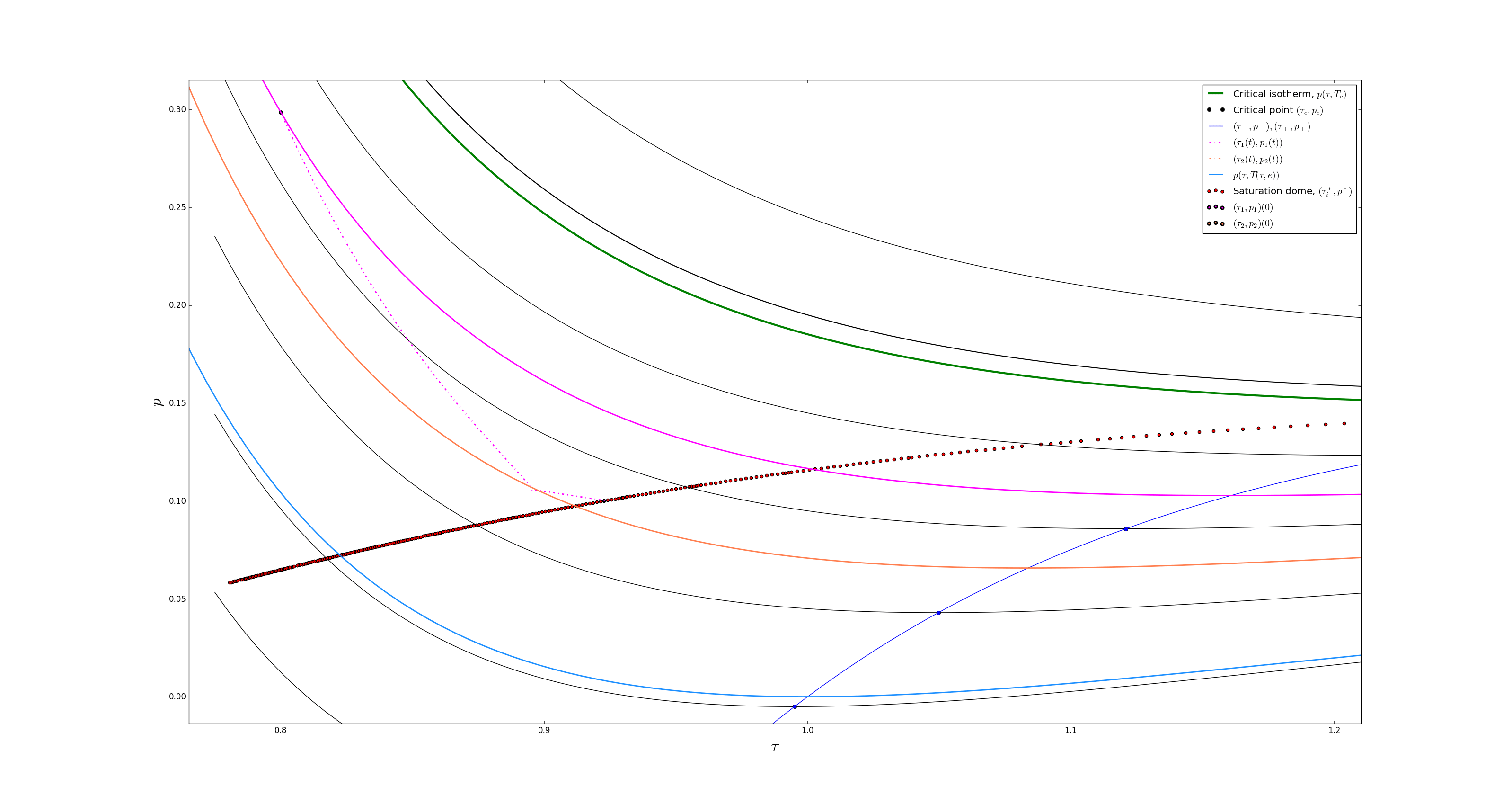}
  \includegraphics[width=0.8\linewidth]{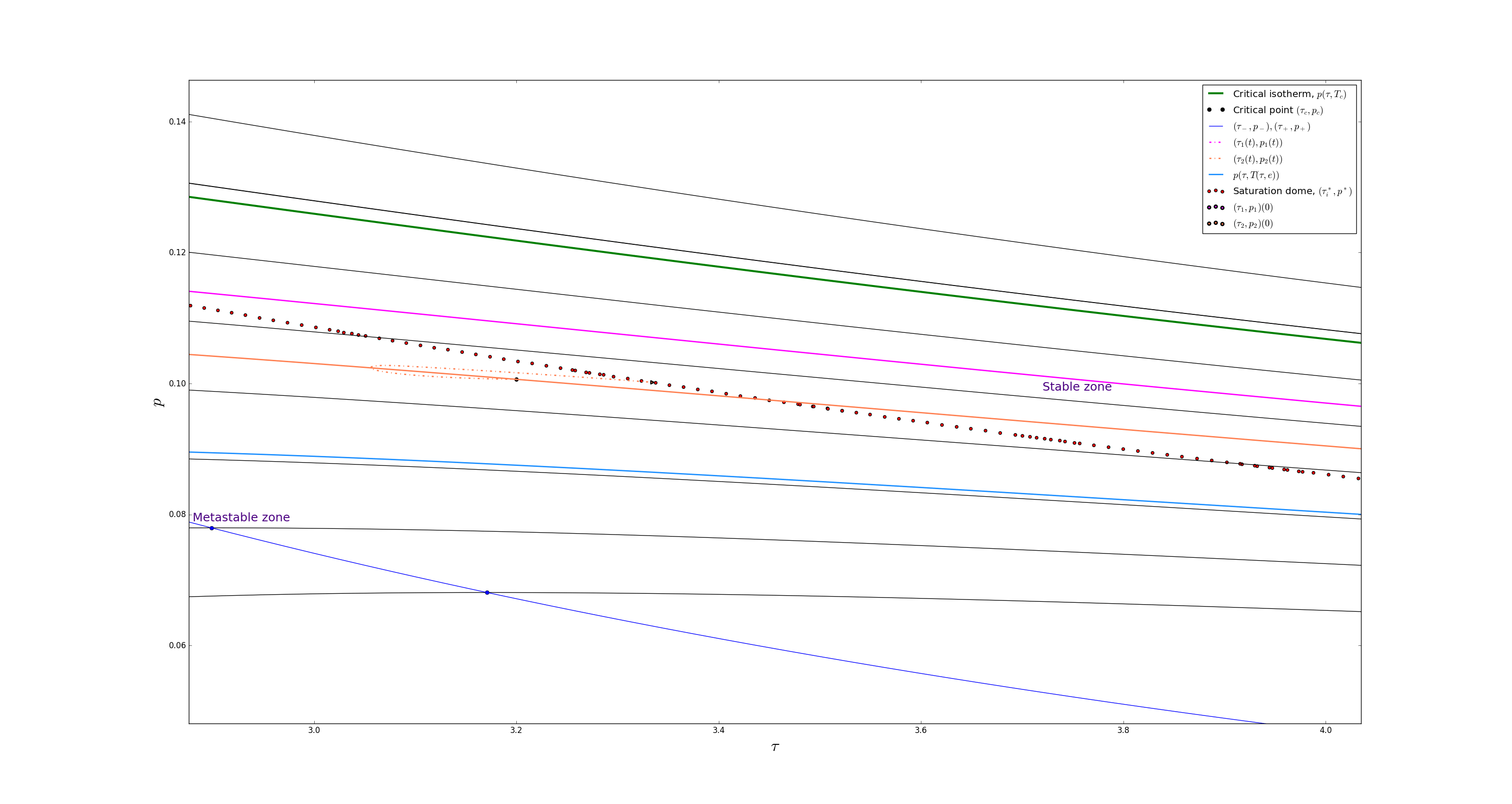}
  \caption{Spinodal zone, from top to bottom. Trajectories of the dynamical system
    \eqref{eq:fraction-model} in the $(\tau,p)$ plane. Starting from an
    initial state $(\tau_1(\mathbf{r}), p(\tau_1(\mathbf{r}),e_1(\mathbf{r})))(0)$
    in the liquid region (on the isothermal curve
    in magenta), the trajectory $(\tau_1(\mathbf{r}),
    p(\tau_1(\mathbf{r}),e_1(\mathbf{r})))(t)$ 
    is represented with a dashed magenta line and converges
    towards the saturation dome. The
    trajectory $(\tau_2(\mathbf{r}),
    p(\tau_2(\mathbf{r}),e_2(\mathbf{r})))(t)$  is represented in orange.
    Middle and bottom figures: zoom of trajectories $(\tau_1(\mathbf{r}),
    p(\tau_1(\mathbf{r}),e_1(\mathbf{r})))(t)$  and $(\tau_2(\mathbf{r}),
    p(\tau_2(\mathbf{r}),e_2(\mathbf{r})))(t)$ respectively.}
  \label{fig:spinodal_etat1liqpur_tau_p}
\end{figure}

In Figures \ref{fig:spinodal_etat1liqpur_tau_e} and   \ref{fig:spinodal_etat1liqpur_tau_p}
are plotted trajectories corresponding to the initial condition 
\begin{equation}
  \label{eq:spinodal_IC}
  \mathbf{r}(0) = (0.2,0.5,0.42),
\end{equation}
which corresponds to a state $(\tau_1(\mathbf{r}),e_1(\mathbf{r}))(0)  =
(0.8,2.1)$ belonging to the stable liquid zone with
$p_1(0)=0.2986$, $T_1(0)=1.1166$ and $\mu_1(0)=2.561$, and a state 
$(\tau_2(\mathbf{r}),e_2(\mathbf{r}))(0)  =(3.2,2.9)$ belonging to a
metastable vapor state with
$p_2(0)=0.1006$, $T_2(0)=1.0708$ and $\mu_2(0)=2.2736$.
Focusing on Figure \ref{fig:spinodal_etat1liqpur_tau_e}-top, the trajectory 
$(\tau_1(\mathbf{r}),e_1(\mathbf{r}))(t)$ is represented with a dashed
magenta line.
One observes that the trajectory starts
from the magenta subcritical isothermal curve, goes through the 
stable liquid zone and converges towards
a point of the saturation dome, see Figure
\ref{fig:spinodal_etat1liqpur_tau_e}-middle for a zoom of the
trajectory.
The trajectory 
$(\tau_2(\mathbf{r}),e_2(\mathbf{r}))(t)$ (dashed orange line) is
similar, except that it remains in the metastable vapor zone before
converging towards a point of the saturation dome. 
The saturation asymptotic state is characterized by the fractions
\begin{equation}
  \label{eq:spinodal_Tf}
    \mathbf{r}(T_f) =(0.255, 0.55, 0.47),
\end{equation}
with $(\tau_1,e_1)(T_f)=(0.923,2.15)$, $(\tau_2,e_2)(T_f)=(3.33,2.93)$
and $p_1(T_f)=p_2(T_f)=0.1$, $T_1(T_f)=T_2(T_f)=1.077)$, $\mu_1(T_f)=\mu_2(T_f)=2.284$.
See
Figure~\ref{fig:spinodal_etat1liqpur_tau_e}-bottom for a zoom of the trajectory.
Figures \ref{fig:spinodal_etat1liqpur_tau_p} represent the complementary
trajectories plotted in the $(\tau,p)$ plane.

%------------------------------------------------------------------------------
\subsubsection{Stable phase zone}
\label{sec:pure_stable_zone}
%------------------------------------------------------------------------------

The purpose is to illustrate the attraction of the line
$\alpha=\varphi=\xi$ 
for any initial data $\mathbf{r}(0)\in ]0,1[^3$,
as soon as the state $(\tau,e)$ belongs to a
stable phase zone. The corresponding equilibrium is then the
equilibrium
$(\tau_1(\overline{\mathbf{r}}),e_1(\overline{\mathbf{r}}))=
(\tau_2(\overline{\mathbf{r}}),e_2(\overline{\mathbf{r}}))=
(\tau,e)$, see Proposition \ref{prop:lyap}.

We consider a state $(\tau,e)=(3,3.1)$ belonging to the stable vapor zone.
The vector field of the dynamical system \eqref{eq:fraction-model} is
represented in Figure \ref{fig:pure_phase} by light blue
arrows. For some random initial conditions $\mathbf{r}(0)\in
]0,1[^3$ (represented by green dots), 
the corresponding trajectories (green lines) converge towards points belonging to the
line $\alpha=\varphi=\xi$ plotted in red. Then the asymptotic states
are such that $(\tau_1(\overline{\mathbf{r}}),e_1(\overline{\mathbf{r}}))=
(\tau_2(\overline{\mathbf{r}}),e_2(\overline{\mathbf{r}}))=
(\tau,e)$.

\begin{figure}[htpb]
  \centering
  \includegraphics[width=\linewidth]{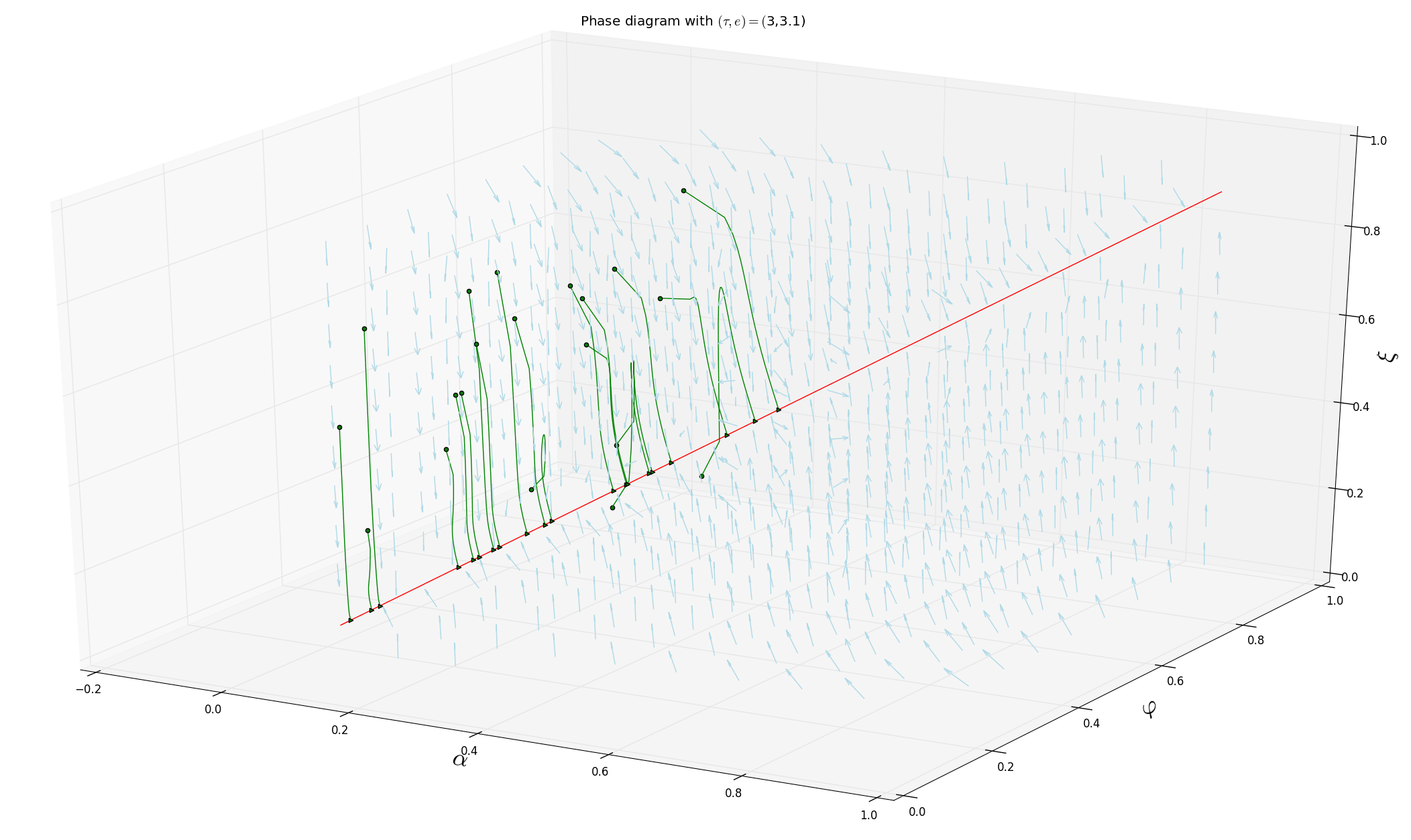}
  \caption{Stable phase zone: vector field of the dynamical system
    \eqref{eq:fraction-model} (light blue arrows). The red line
    corresponds to the line $\alpha=\varphi=\xi$.
    For any 
    initial condition $\mathbf{r}(0)$, the trajectories converge
    towards a point belonging to the line $\alpha=\varphi=\xi$,
    corresponding to  the state $(\tau,e)$.}.
\label{fig:pure_phase}
\end{figure}

\begin{figure}[htpb]
  \centering
  \includegraphics[width=\linewidth]{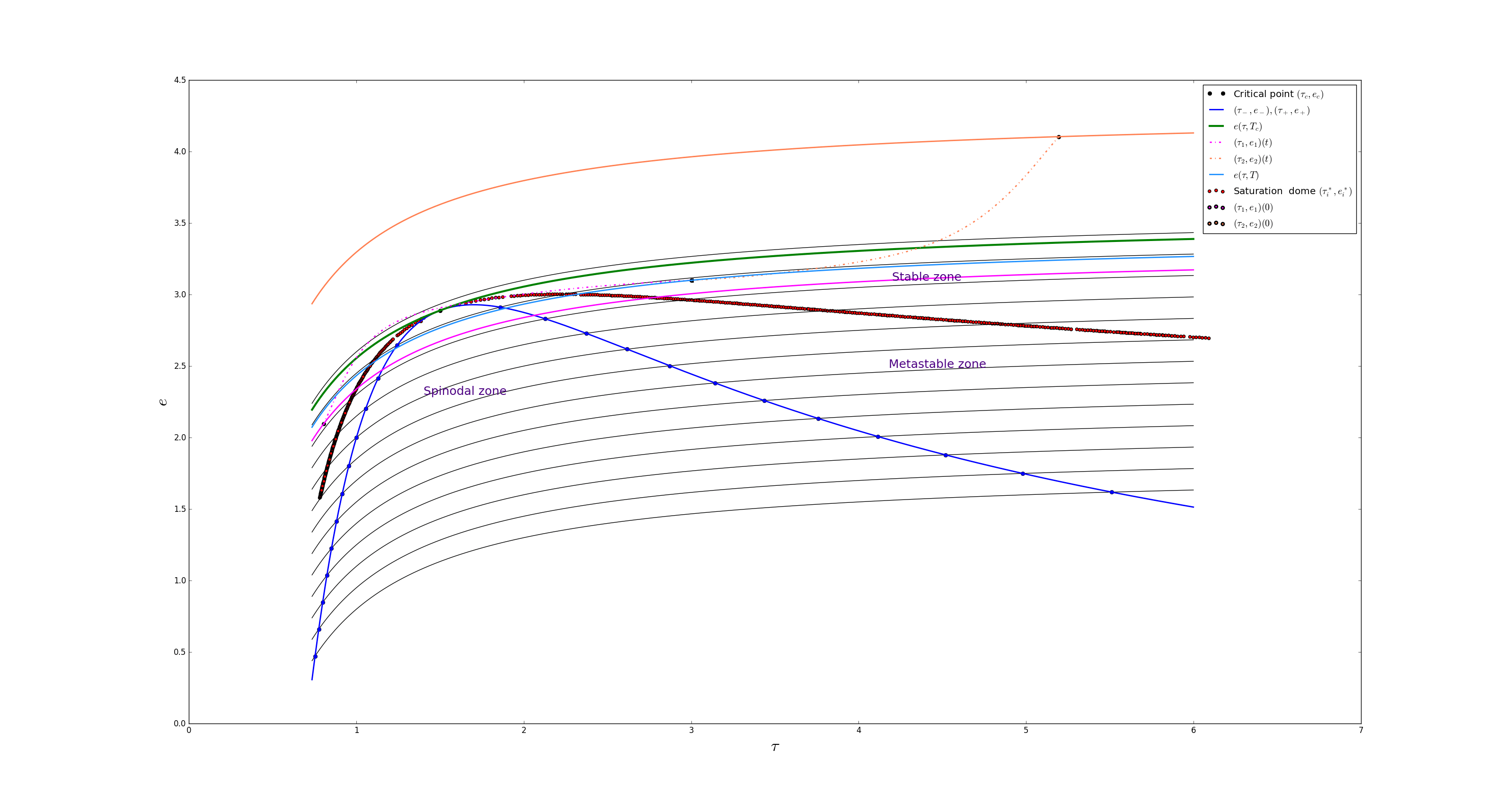}
 \caption{Stable phase zone. Trajectories of the dynamical system
    \eqref{eq:fraction-model} in the $(\tau,e)$ plane. Starting from an
    initial state $(\tau_1(\mathbf{r}), e_1(\mathbf{r}))(0)$
    in the stable liquid region (on the isothermal curve
    in magenta), the trajectory $(\tau_1(\mathbf{r}),
    e_1(\mathbf{r}))(t)$ is represented with a dashed magenta line and converges
    towards the state $(\tau,e)$. The
    trajectory $(\tau_2(\mathbf{r}(t)),
    e_2(\mathbf{r}))(t)$ is represented in orange.}
  \label{fig:pure_tau_e}
\end{figure}
 
\begin{figure}[htpb]
  \centering
  \includegraphics[width=\linewidth]{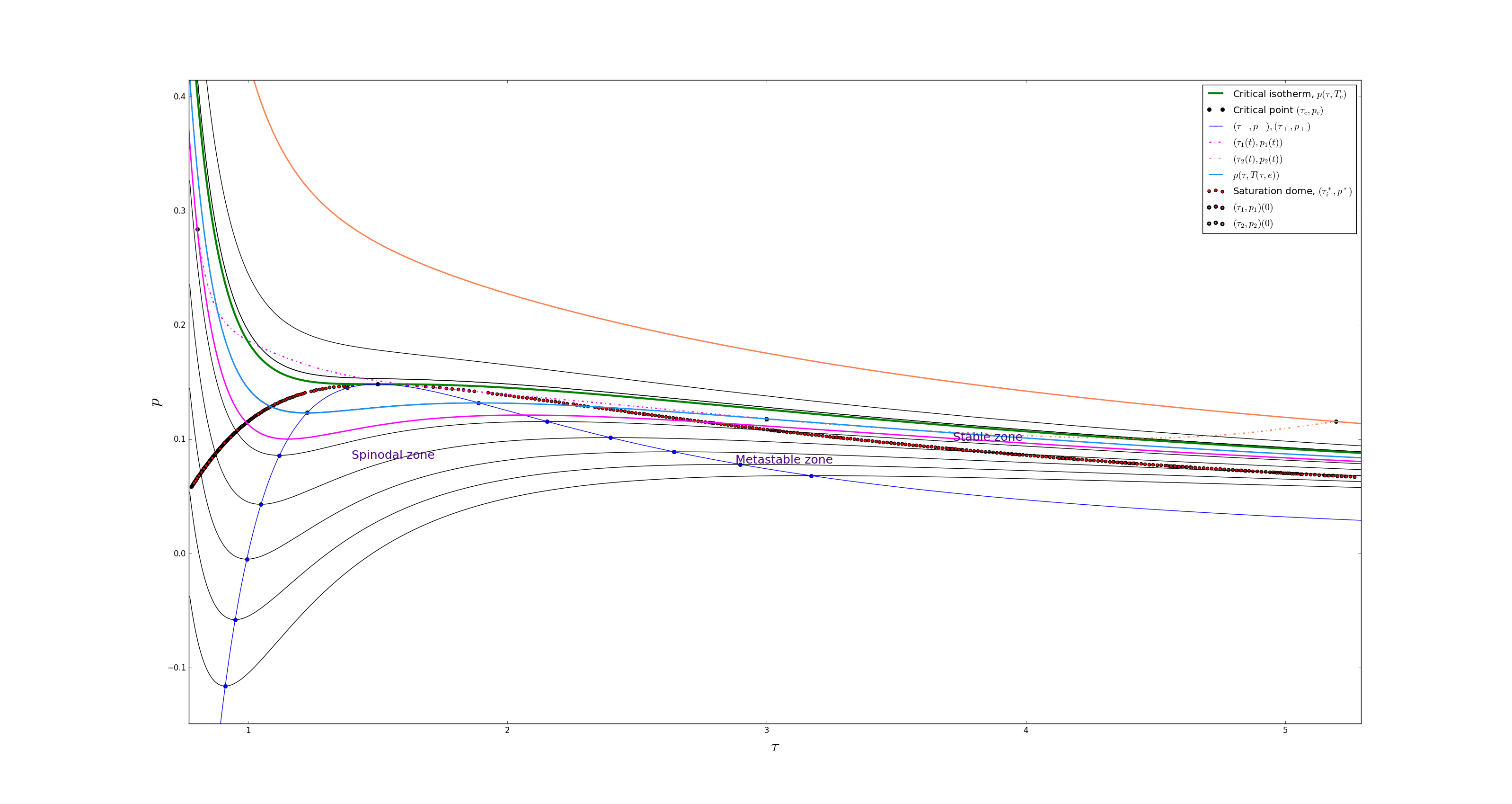}
  \includegraphics[width=0.8\linewidth]{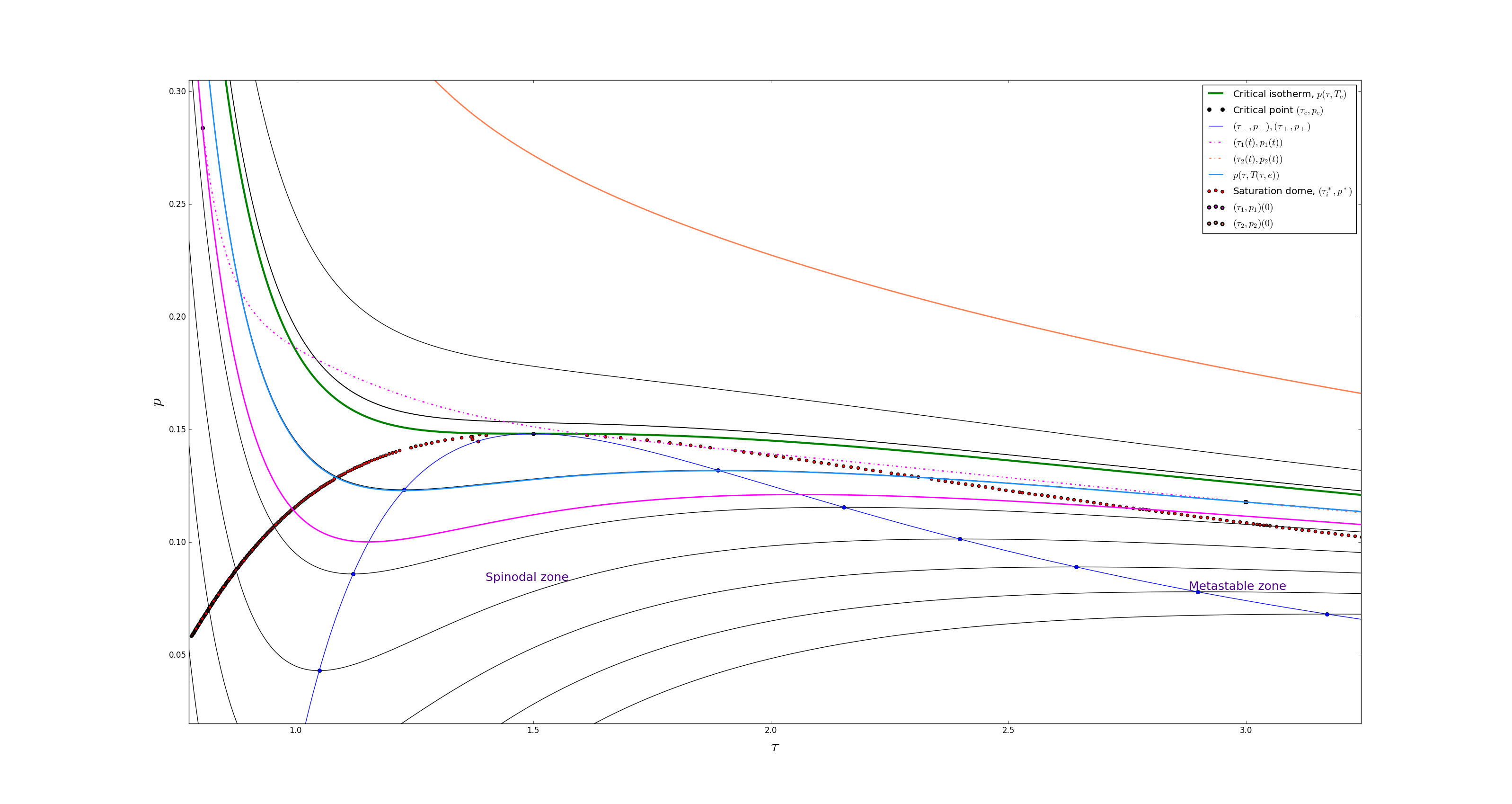}
  \includegraphics[width=0.8\linewidth]{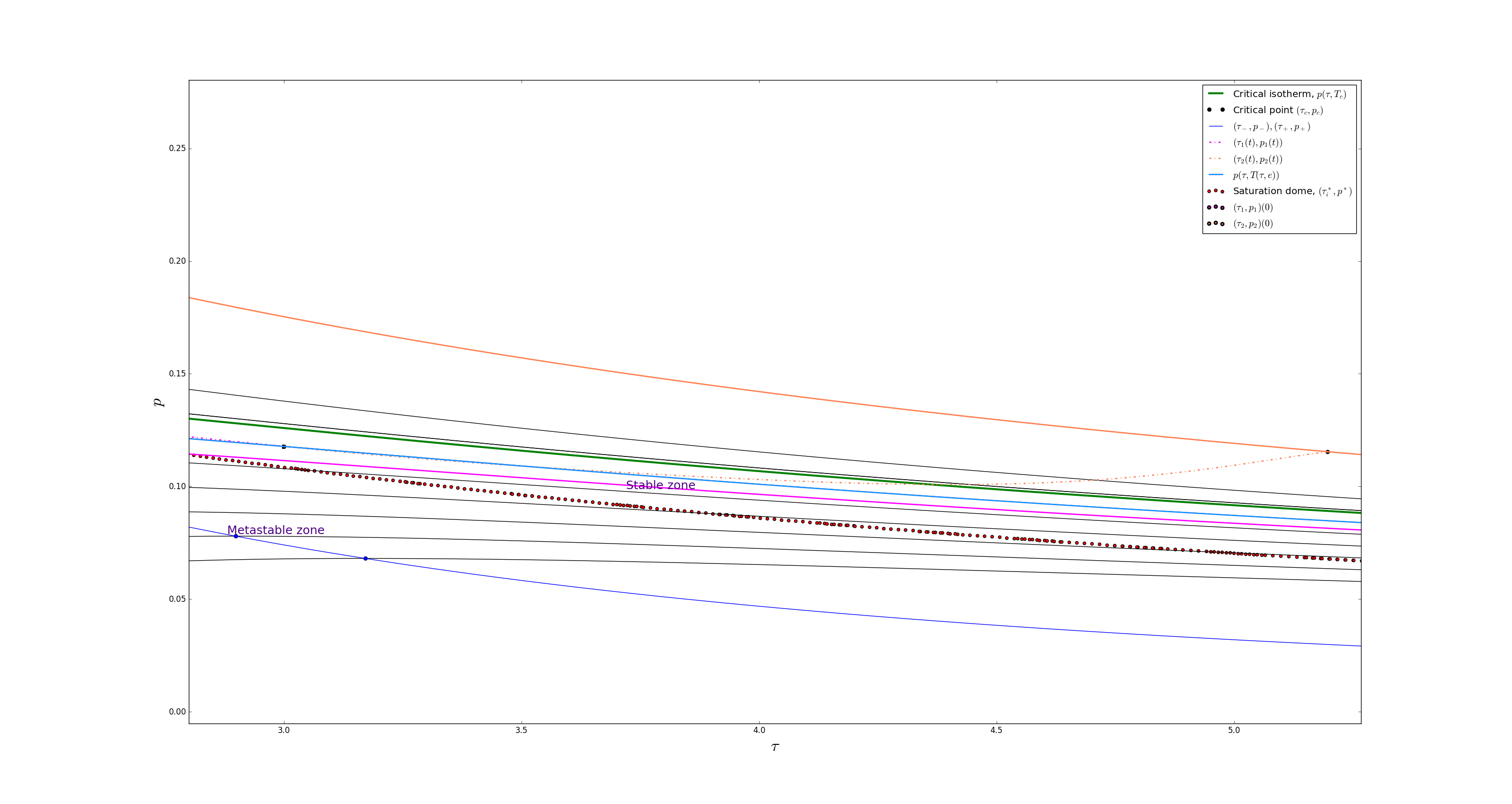}
  \caption{Stable phase zone, from top to bottom. Trajectories of the dynamical system
    \eqref{eq:fraction-model} in the $(\tau,p)$ plane. Starting from an
    initial state $(\tau_1(\mathbf{r}), p(\tau_1(\mathbf{r}),e_1(\mathbf{r})))(0)$
    in the stable liquid region (on the isothermal curve
    in magenta), the trajectory $(\tau_1(\mathbf{r}),
    p(\tau_1(\mathbf{r}),e_1(\mathbf{r})))(t)$ is represented with a
    dashed magenta line and converges
    towards the point $(\tau,e)$. The
    trajectory $(\tau_2(\mathbf{r}),
    p(\tau_2(\mathbf{r}),e_2(\mathbf{r})))(t)$ is represented in orange.
    Middle and bottom figures: zoom of trajectories $(\tau_1(\mathbf{r}),
    p(\tau_1(\mathbf{r}),e_1(\mathbf{r})))(t)$ and $(\tau_2(\mathbf{r}),
    p(\tau_2(\mathbf{r}),e_2(\mathbf{r})))(t)$ respectively.}
  \label{fig:pure_tau_p}
\end{figure}

In Figures \ref{fig:pure_tau_e} and  \ref{fig:pure_tau_p}
are plotted trajectories with the initial condition 
\begin{equation}
  \label{eq:pure_IC}
  \mathbf{r}(0) = (0.134,0.5,0.338),
\end{equation}
which corresponds to a state $(\tau_1(\mathbf{r}),e_1(\mathbf{r}))(0)  =
(0.8,2.1)$ belonging to the stable liquid zone, and a state 
$(\tau_2(\mathbf{r}),e_2(\mathbf{r}))(0)  =(5.196,4.1044)$
corresponding to a supercritical state.
Focusing on Figure \ref{fig:pure_tau_e}, the trajectory 
$(\tau_1(\mathbf{r}),e_1(\mathbf{r}))(t)$ is represented with a dashed
magenta line.
One observes that it starts
from the magenta subcritical isothermal curve, goes over the critical
point entering the supercritical zone, and finally converges towards
the point $(\tau,e)$. The trajectory 
$(\tau_2(\mathbf{r}),e_2(\mathbf{r}))(t)$ (dashed orange line) is
similar, going from the supercritical zone to the stable vapor zone
and finally
converging towards the point $(\tau,e)$.
Figures \ref{fig:spinodal_etat1liqpur_tau_p} represent the same
trajectories plotted in the $(\tau,p)$ plane.
One observes that the trajectory of
$(\tau_1(\mathbf{r}),p(\tau_1(\mathbf{r}),e_1(\mathbf{r})))(t)$ starts from the stable
liquid zone, crosses the critical isothermal curve twice before
converging towards the point $(\tau,p(\tau,e))$.

%------------------------------------------------------------------------------
\subsubsection{Metastable zone}
\label{sec:metast-zone}
%------------------------------------------------------------------------------
The purpose is to illustrate the fact that, if the state $(\tau,e)$ belongs to a
metastable zone, for any initial data
$\mathbf{r}(0)\in ]0,1[^3$, there exist two possible attraction
points.

We consider a state $(\tau,e)=(3.2,2.5)$ belonging to the metastable
vapor zone
with $p(\tau,e)=0.0759$ and $T(\tau,e)=0.9375$.
The vector field of the dynamical system \eqref{eq:fraction-model} is
represented in Figure \ref{fig:meta_phase} by light blue
arrows. For some random initial conditions $\mathbf{r}(0)\in
]0,1[^3$ (represented by green or yellow dots), 
the complementary trajectories (green or yellow lines) converge
towards 
\begin{itemize}
\item either an attraction point which lies on to the line
  $\alpha=\varphi=\xi$ (yellow trajectories). In that case the
  asymptotic state satisfies
  \begin{equation*}
    (\tau_1(\overline{\mathbf{r}}),e_1(\overline{\mathbf{r}}))=
    (\tau_2(\overline{\mathbf{r}}),e_2(\overline{\mathbf{r}}))=
    (\tau,e),
  \end{equation*}
  and remains 
  metastable, see Propositions \ref{prop:opti_equi}-\eqref{it:pure} and
  \ref{prop:equilibrium_states}-\eqref{it:coexistence};
\item either the attraction point
  $\mathbf{r}^*=(\alpha^*,\varphi^*,\xi^*)$ (green trajectories) which
  concurs with the
  unique state $(\tau_i^*, e_i^*)$, $i=1,2$, defined by
  \eqref{eq:egalite}, which belongs to the
  saturation dome, see Propositions \ref{prop:opti_equi}-\eqref{it:mix} and
  \ref{prop:equilibrium_states}-\eqref{it:pureq}.
\end{itemize}

\begin{figure}[htpb]
  \centering
  \includegraphics[width=\linewidth]{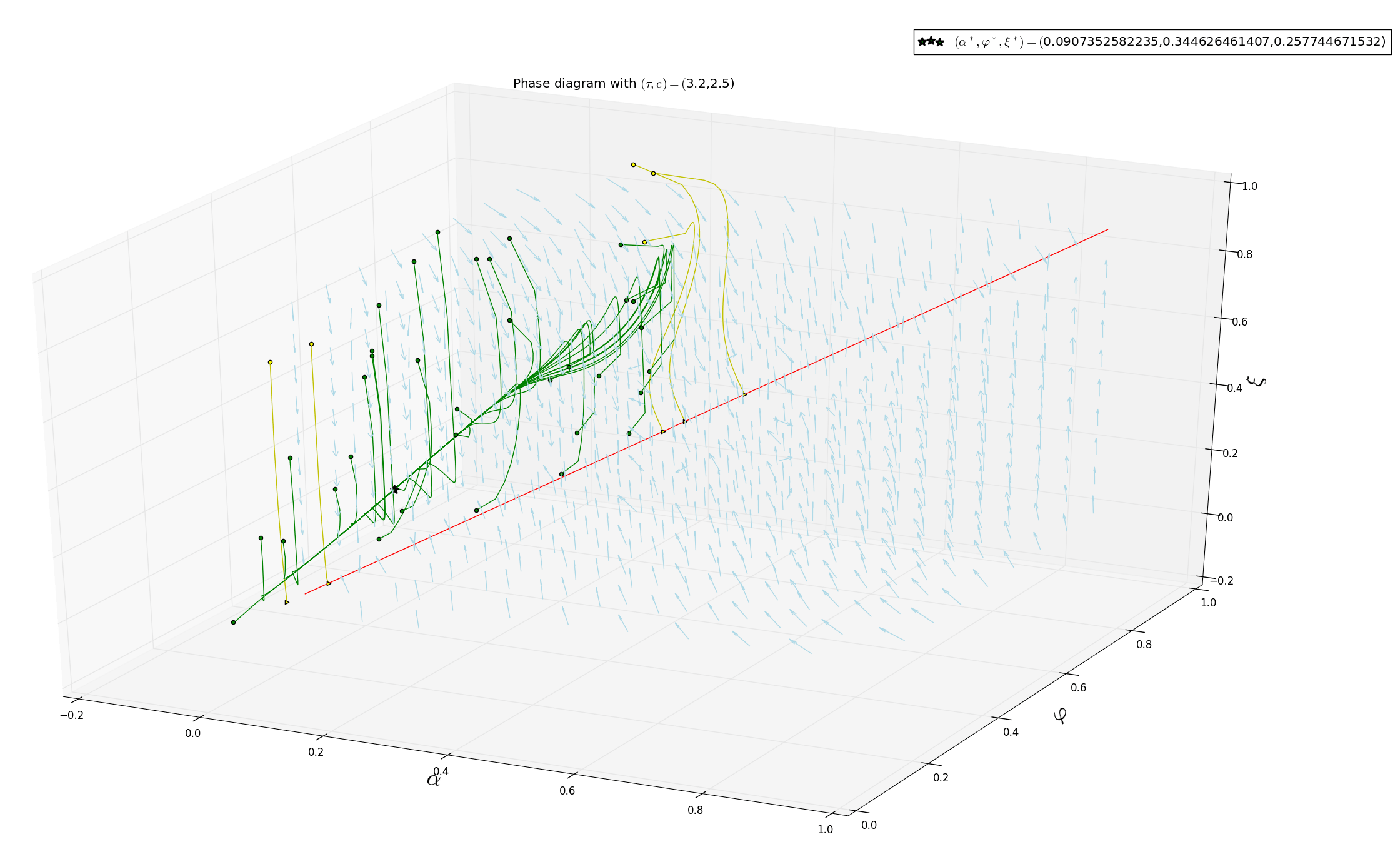}
  \caption{Metastable zone: vector field of the dynamical system
    \eqref{eq:fraction-model} (light blue arrows). The red line
    corresponds to the line $\alpha=\varphi=\xi$.
    For any 
    initial condition $\mathbf{r}(0)$, the trajectories converge
    either towards a point belonging to the line $\alpha=\varphi=\xi$,
    corresponding to the state $(\tau,e)$ (yellow trajectories), or to the
    point $\mathbf{r}^*=(\alpha^*,\varphi^*,\xi^*)$, which concurs with a
    state belonging to the saturation dome (green trajectories).}
\label{fig:meta_phase}
\end{figure}

%------------------------------------------------------------------------------
\textit{Metastable state and perturbation within the phase}.
%------------------------------------------------------------------------------
In Figures \ref{fig:meta_pure_tau_e} and \ref{fig:meta_pure_tau_p}
the represented trajectories correspond to a realization of the
dynamical system for the initial condition
\begin{equation}
  \label{eq:meta_pure_IC}
  \mathbf{r}(0) =(0.5,0.5,0.55).
\end{equation}
It boils down  to an initial state
$(\tau_1(\mathbf{r}),e_1(\mathbf{r}))(0)=(3.2,2.75)$ in the
metastable vapor zone with $p_1(0)=0.091$, $T_1(0)=1.02$,
$\mu_1(0)=2.15$,
and to an initial state
$(\tau_2(\mathbf{r}),e_2(\mathbf{r}))(0)=(3.2,2.25)$ belonging to the
spinodal zone with
$p_2(0)=0.06$, $T_2(0)=0.85$, $\mu_2(0)=1.75$. 
Notice that in this case, it holds
$e_1(\mathbf{r})(0)>e>e_2(\mathbf{r})(0)$.
The perturbation is small enough to ensure that the trajectories
$(\tau_i(\mathbf{r}),e_i(\mathbf{r}))(t)$ converge towards the point
$(\tau,e)$ in the metastable zone. The asymptotic state is
characterized by the fractions
\begin{equation}
  \label{eq:meta_pure_Tf}
  \mathbf{r}(T_f) =(0.499,0.499,0.499),
\end{equation}
with $p_1(T_f)=p_2(T_f)=0.0759$ and
$T_1(T_f)=T_2(T_f)=0.9374$. 
\begin{figure}[htpb]
  \centering
  \includegraphics[width=\linewidth]{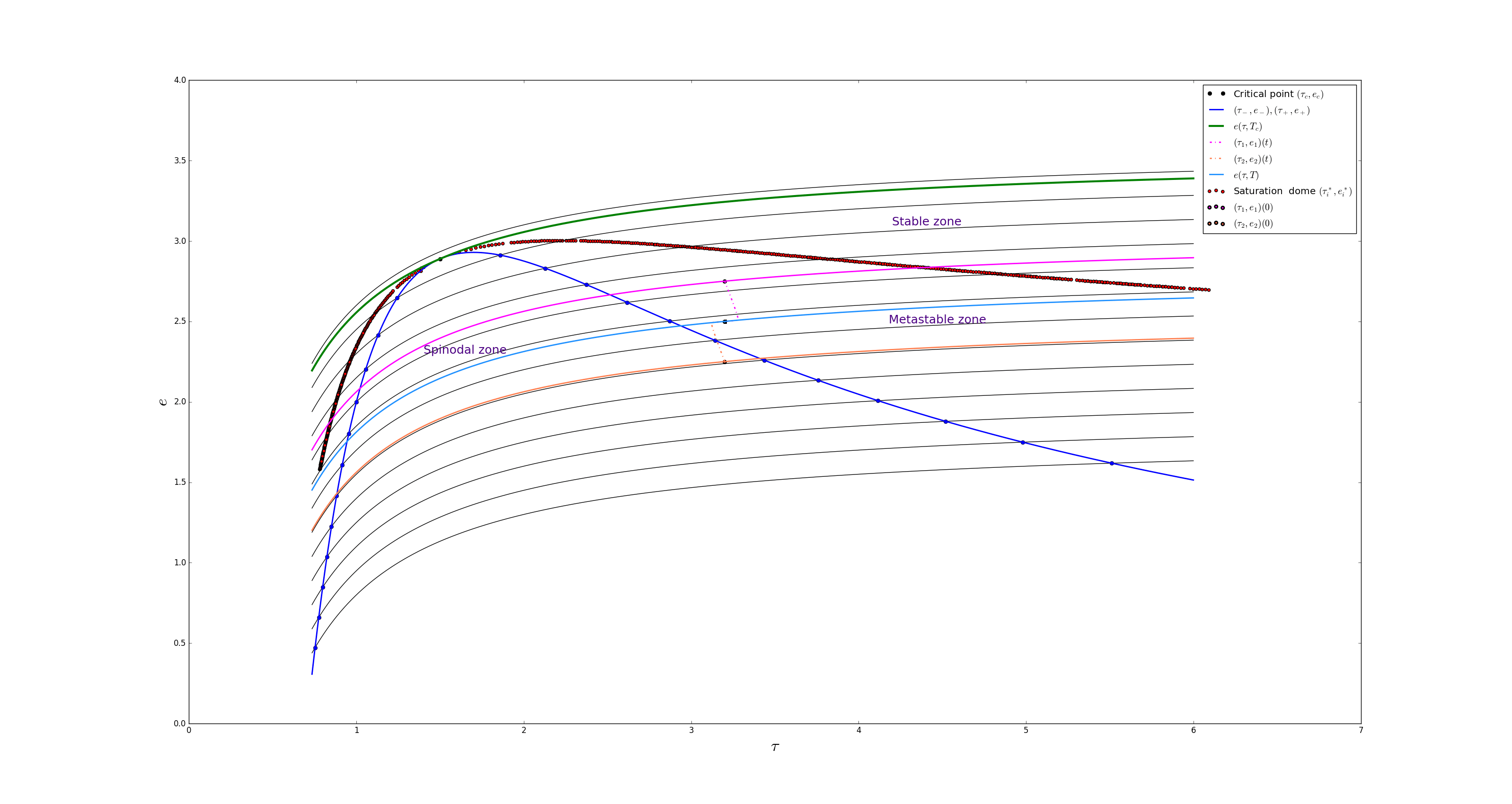}
  \includegraphics[width=\linewidth]{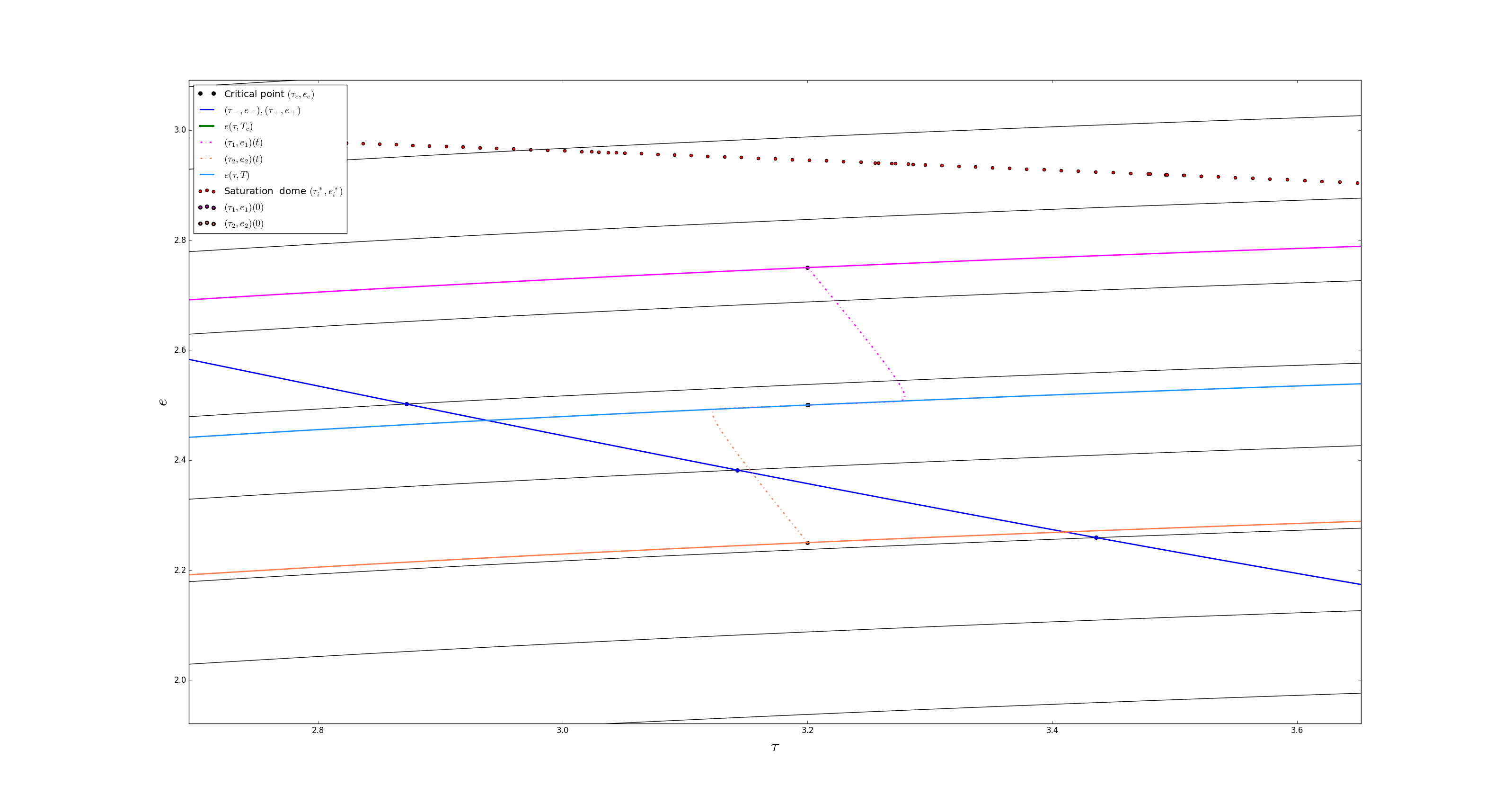}
 \caption{Metastable state and perturbation within the phase. Top figure:
   trajectories of the dynamical system
    \eqref{eq:fraction-model} in the $(\tau,e)$ plane. Starting from an
    initial state $(\tau_1(\mathbf{r}), e_1(\mathbf{r}))(0)$
    in the metastable vapor region (on the magenta isothermal curve),
    the trajectory $(\tau_1(\mathbf{r}),
    e_1(\mathbf{r}))(t)$ is represented with a dashed magenta line and converges
    towards the state $(\tau,e)$. The
    trajectory $(\tau_2(\mathbf{r}),
    e_2(\mathbf{r}))(t)$ is represented in orange and starts with an
    initial condition in the spinodal zone. 
    Bottom figure: zoom of trajectories $(\tau_i(\mathbf{r}),e_I(\mathbf{r}))(t)$.}
  \label{fig:meta_pure_tau_e}
\end{figure}
 
\begin{figure}[htpb]
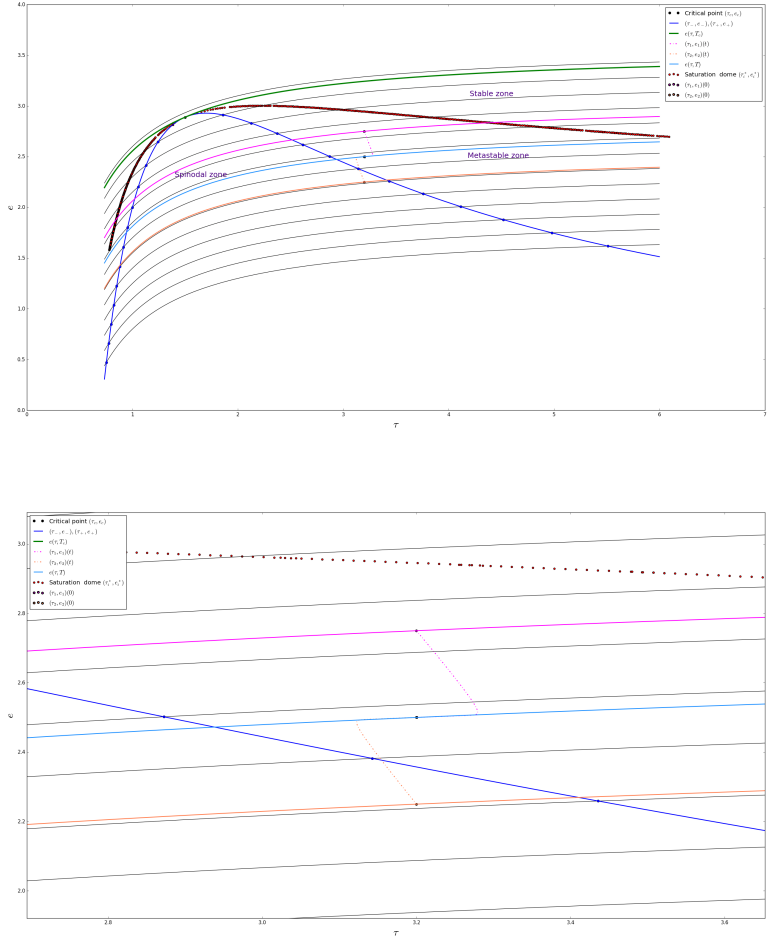

  \centering
  \includegraphics[width=\linewidth]{meta_pure_tau_e.png}
  \includegraphics[width=\linewidth]{meta_pure_tau_e_zoom.png}
  \caption{Metastable state and perturbation within the phase. Top
    figure: trajectories of the dynamical system 
    \eqref{eq:fraction-model} in the $(\tau,p)$ plane. Starting from an
    initial state $(\tau_1(\mathbf{r}), p(\tau_1(\mathbf{r}), e_1(\mathbf{r})))(0)$
    in the metastable vapor region (on the isothermal curve
    in magenta), the trajectory $(\tau_1(\mathbf{r}),
    p(\tau_1(\mathbf{r},e_1(\mathbf{r}))))(t)$ is represented with a dashed magenta line and converges
    towards the point $(\tau,e)$. The
    counterpart for the state $(\tau_2(\mathbf{r}),
    p(\tau_2(\mathbf{r}),e_2(\mathbf{r})))(t)$ is represented in orange, and starts from
    an initial datum in the spinodal zone.
    Bottom figure: zoom of trajectories $(\tau_1(\mathbf{r}),
    p(\tau_1(\mathbf{r}),e_1(\mathbf{r})))(t)$ and $(\tau_2(\mathbf{r}),
    p(\tau_2(\mathbf{r}),e_2(\mathbf{r})))(t)$ respectively.}
  \label{fig:meta_pure_tau_p}
\end{figure}

%------------------------------------------------------------------------------
\textit{Metastable state and perturbation outside the phase}.
%------------------------------------------------------------------------------
We provide in Figures \ref{fig:meta_maxwell_tau_e} and \ref{fig:meta_maxwell_tau_p}
the trajectories of the dynamical system for an initial condition
\begin{equation}
  \label{eq:meta_maxwell_IC}
  \mathbf{r}(0) = (0.16,0.5,0.328).
\end{equation}
It corresponds to an initial state
$(\tau_1(\mathbf{r}),e_1(\mathbf{r}))(0)=(0.8,2.1)$ in the
stable liquid zone and an initial state
$(\tau_2(\mathbf{r}),e_2(\mathbf{r}))(0)=(5.376,3.36)$ belonging to the
stable vapor zone. 
The perturbation is large enough to ensure that the trajectories
$(\tau_i(\mathbf{r}),e_i(\mathbf{r}))(t)$ converge towards a state
belonging to the saturation dome.

\begin{figure}[htpb]
  \centering
  \includegraphics[width=\linewidth]{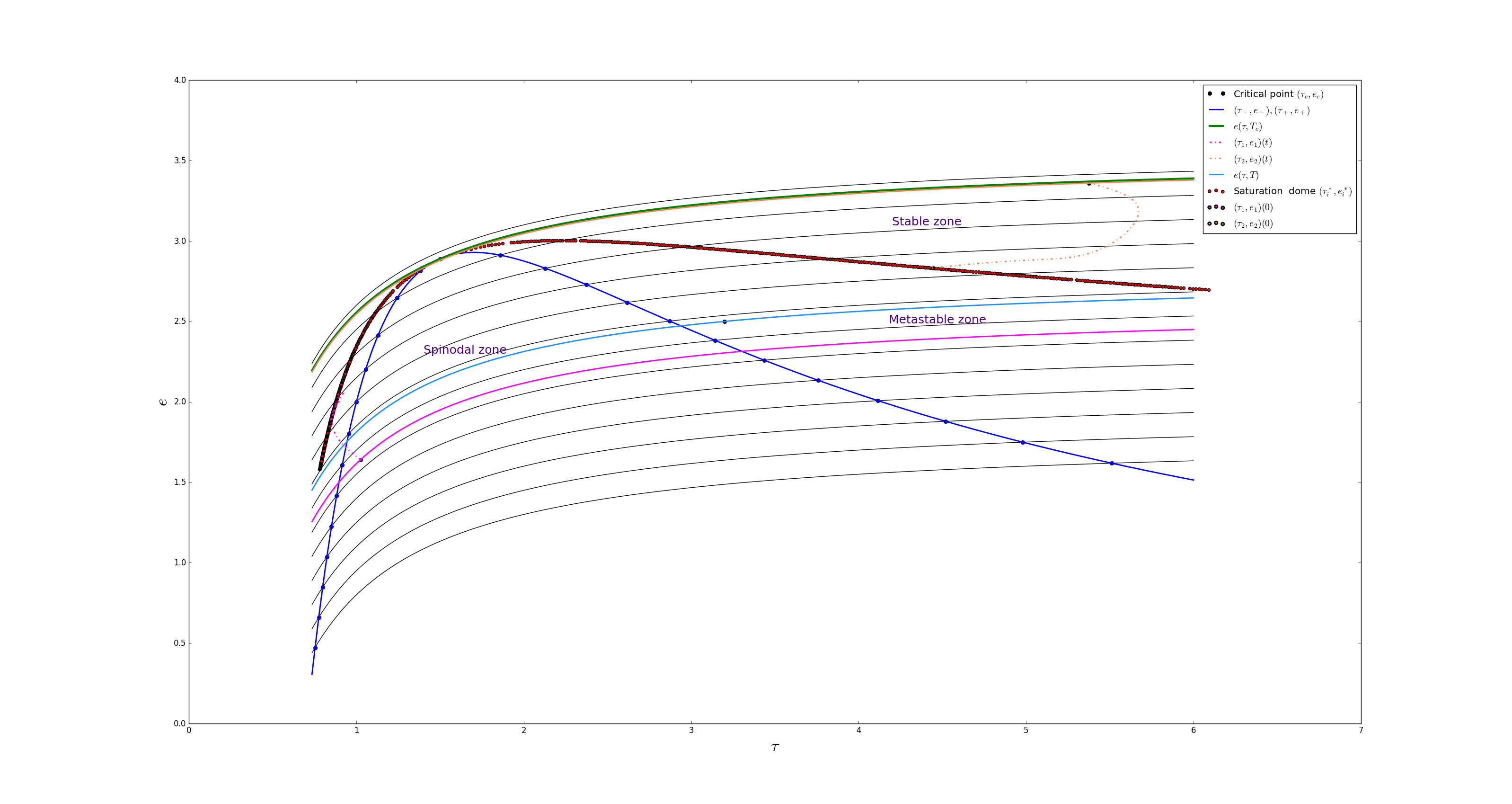}
  \includegraphics[width=0.8\linewidth]{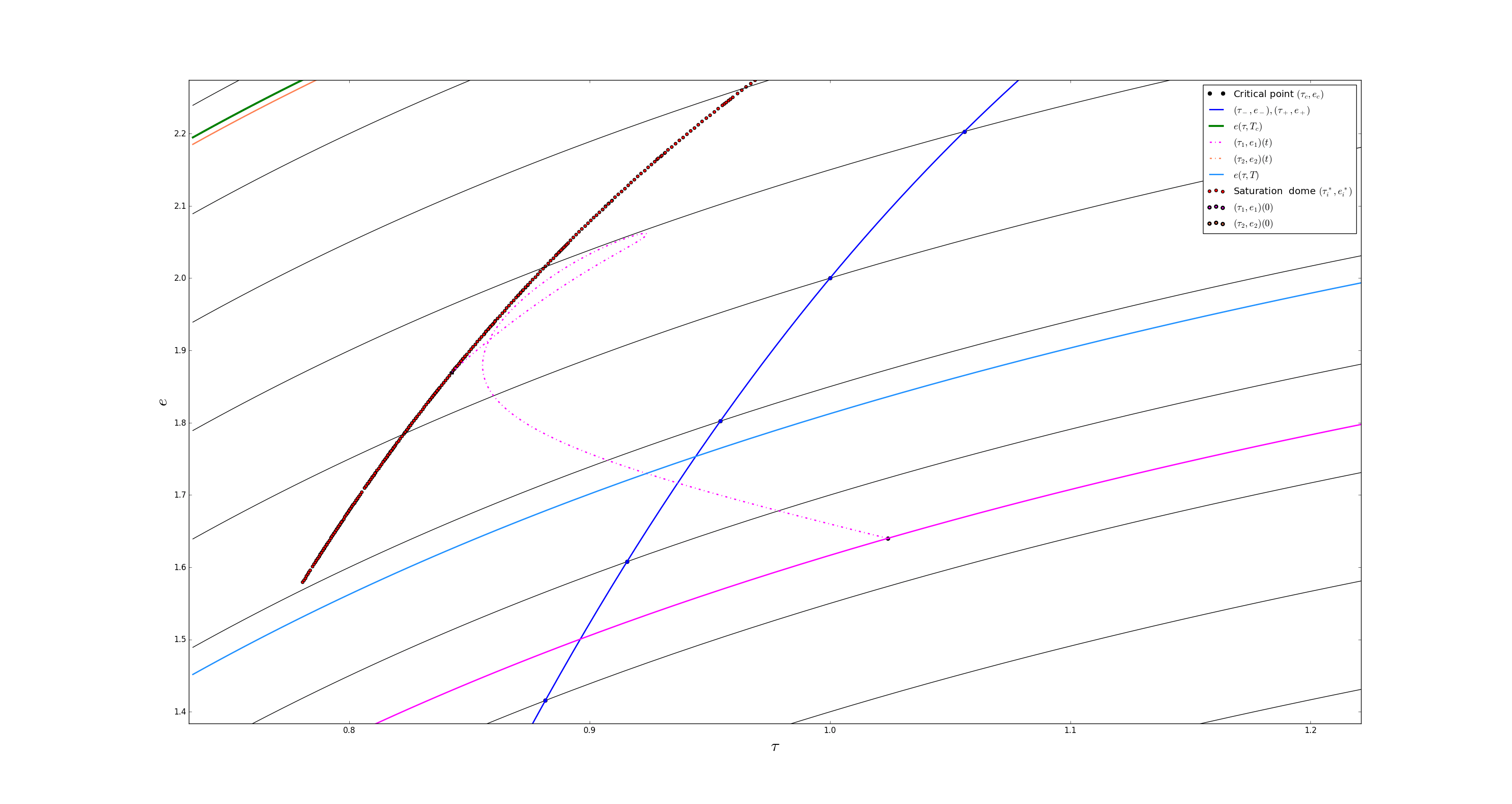}
  \includegraphics[width=0.8\linewidth]{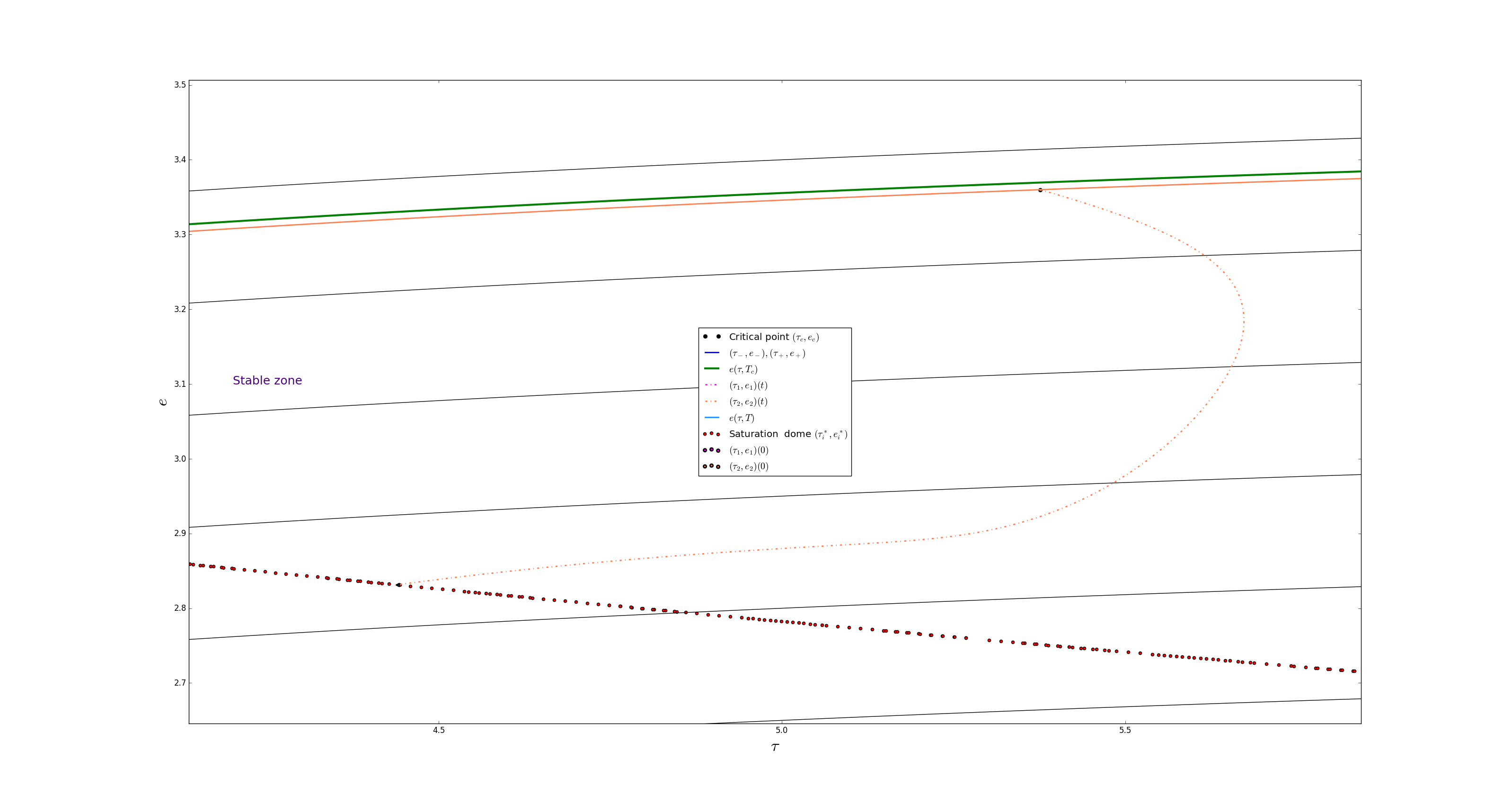}
 \caption{Metastable state and perturbation outside the phase.
   From top to bottom: trajectories of the dynamical system
    \eqref{eq:fraction-model} in the $(\tau,e)$ plane. Starting from an
    initial state $(\tau_1(\mathbf{r}), e_1(\mathbf{r}))(0)$
    in the metastable vapor region (on the magenta isothermal curve),
    the trajectory $(\tau_1(\mathbf{r}),
    e_1(\mathbf{r}))(t)$ is represented with a dashed magenta line and converges
    towards the state $(\tau,e)$. The
    trajectory $(\tau_2(\mathbf{r}),
    e_2(\mathbf{r}))(t)$ is represented in orange and starts with
    an initial condition in the spinodal zone.
    Middle and bottom figures: zoom of trajectories
    $(\tau_i(\mathbf{r}),
    e_i(\mathbf{r}))(t)$.}
  \label{fig:meta_maxwell_tau_e}
\end{figure}
 
\begin{figure}[htpb]
  \centering
  \includegraphics[width=\linewidth]{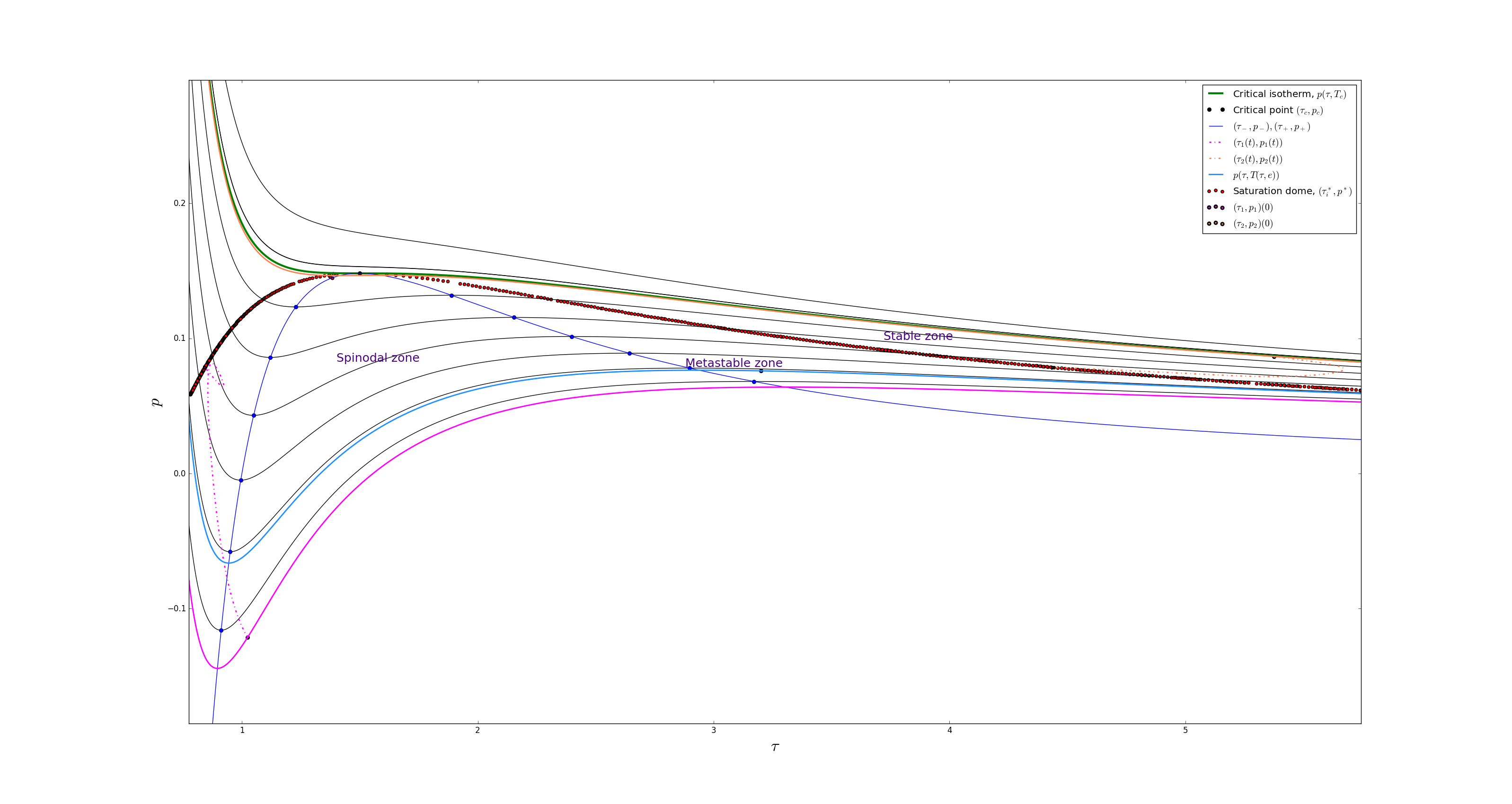}
  \includegraphics[width=0.8\linewidth]{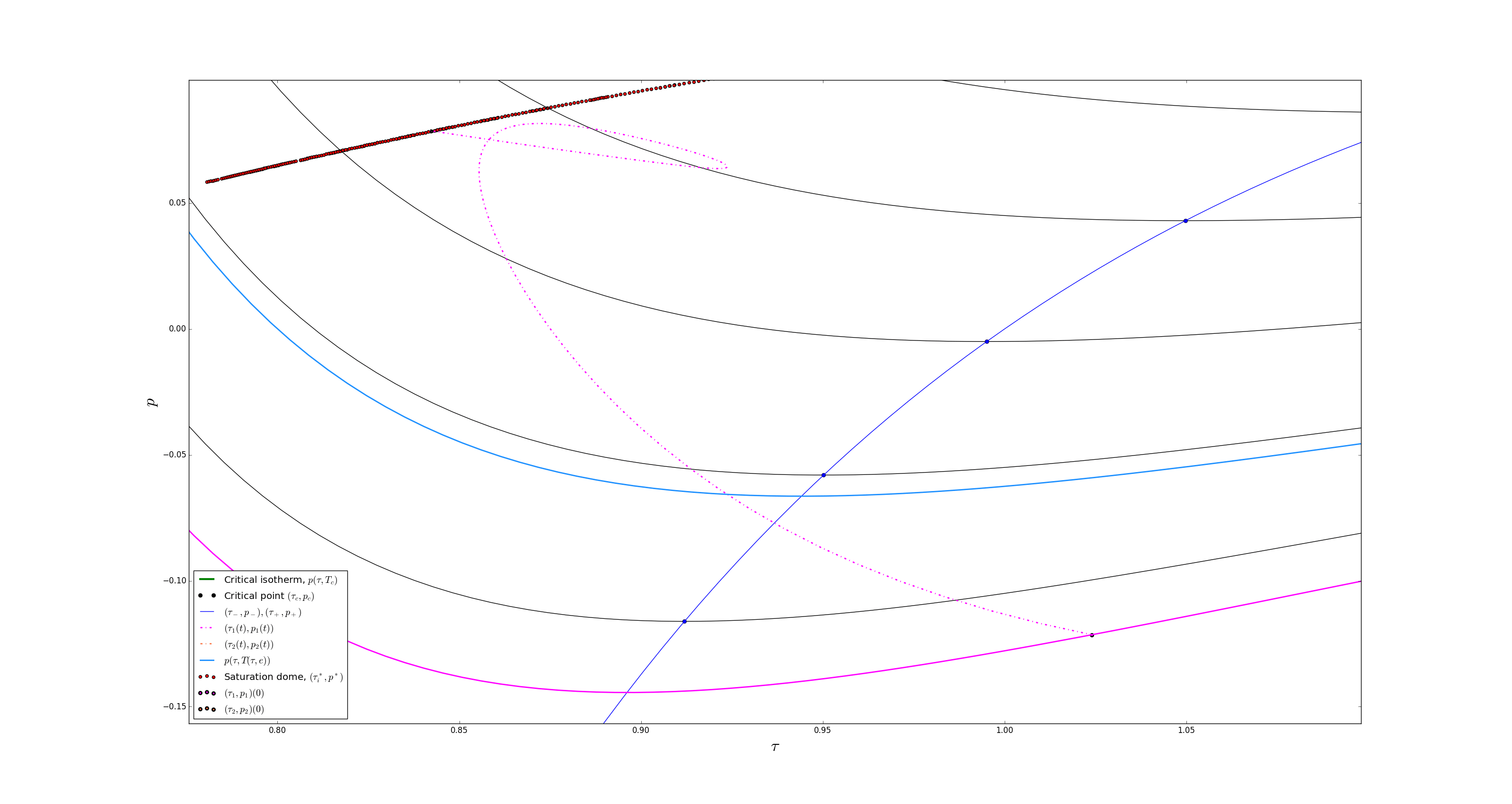}
  \includegraphics[width=0.8\linewidth]{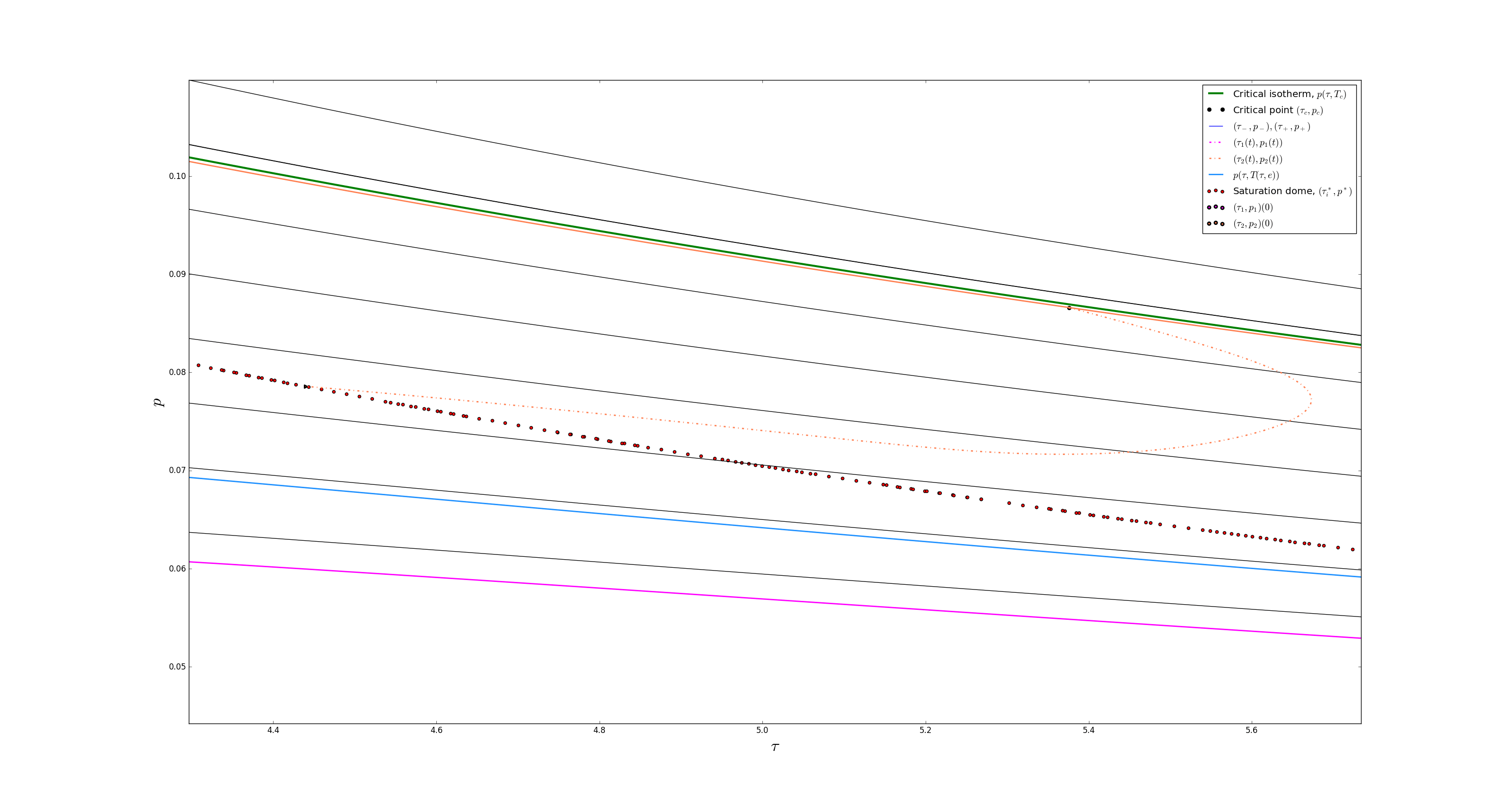}
  \caption{Metastable state and perturbation outside the phase.
   From top to bottom: trajectories of the dynamical system
    \eqref{eq:fraction-model} in the $(\tau,p)$ plane. Starting from an
    initial state $(\tau_1(\mathbf{r}), p(\tau_1(\mathbf{r}), e_1(\mathbf{r})))(0)$
    in the stable liquid region (on the magenta isothermal curve),
    the trajectory $(\tau_1(\mathbf{r}),
    p(\tau_1(\mathbf{r}), e_1(\mathbf{r})))(t)$
    is represented with a dashed magenta line and converges
    towards the point $(\tau,e)$. The
    trajectory $(\tau_2(\mathbf{r}),
    p(\tau_2(\mathbf{r}), e_2(\mathbf{r})))(t)$ is represented in orange.
    Middle and bottom figures: zoom of trajectories $(\tau_1(\mathbf{r}),
    p(\tau_1(\mathbf{r}), e_1(\mathbf{r})))(t)$ and $(\tau_2(\mathbf{r}),
    p(\tau_2(\mathbf{r}), e_2(\mathbf{r})))(t)$ respectively.}
  \label{fig:meta_maxwell_tau_p}
\end{figure}

%------------------------------------------------------------------------------
\section{An homogeneous relaxation model}
\label{sec:homo-model}
%------------------------------------------------------------------------------

The aim of this Section is to investigate the impact of the fluid dynamics
on the stability of metastable states and the apparition of phase
transition. To do so we now consider
the liquid-vapor mixture as a compressible
medium. It is described by its density
$\rho(t,x)$ (and $\tau(t,x)=1/\rho(t,x)$ its specific volume),
its velocity $u(t,x)$ and its internal energy $e(t,x)$, depending on
the time variable $t\in \mathbb R^+$ and the one-dimensional space variable
$x\in \mathbb R$. Since both phases evolve with the same velocity $u$,
we focus on so-called homogeneous models in the spirit of \cite{BH05, Hurisse17}.

The homogeneous model reads
\begin{equation}
  \label{eq:relax-model}
  \begin{cases}
    \p_t(\rho \alpha)+\p_x(\rho u \alpha)=\dfrac{\rho}{\varepsilon} \mathbb F^\alpha(\mathbf{r}),\\
    \p_t(\rho \varphi)+\p_x(\rho u \varphi)=\dfrac{\rho}{\varepsilon}  \mathbb F^\varphi(\mathbf{r}),\\
    \p_t(\rho \xi)+\p_x(\rho u \xi)=\dfrac{\rho}{\varepsilon}  \mathbb F^\xi(\mathbf{r}),\\
    \p_t(\rho )+\p_x(\rho u )=0,\\
    \p_t(\rho u)+\p_x(\rho u^2+ p)=0,\\
    \p_t(\rho E)+\p_x(\rho uE+ up )=0,
  \end{cases}
\end{equation}
where $E=e+u^2/2$ is the total energy.
The last three equations correspond to the Euler's system with
a mixture pressure law $p$ to be define in the sequel.
The first three equations are evolution equations of the fractions
$\mathbf{r}=(\alpha,\varphi,\xi)\in ]0,1[^3$, with relaxation source terms
$(\mathbb F^\alpha, \mathbb F^\varphi, \mathbb F^\xi)$ towards the
Thermodynamic equilibrium, which coincide with the dynamical system
\eqref{eq:fraction-model} studied in the previous section. The
parameter $\varepsilon>0$ stands for a relaxation time towards the
thermodynamic equilibrium. 

%--------------------------------------------------------
% Properties of the Homogeneous Relaxation Model
%--------------------------------------------------------
\subsection{Properties of the homogeneous relaxation model}
\label{sec:prop-homog-model}

First we focus on the convective part of the model
\eqref{eq:relax-model}.
It consists in the Euler system complemented with convection equations of
the fractions $\mathbf{r}(t,x)$; thus it inherits from the wave
structure of the Euler system.
In order to close the system, in agreement with the thermodynamical
constraints presented in the previous sections, the considered
pressure $p$ is a
function of the density $\rho$, the internal energy $e$ and the
fraction vector $\mathbf{r}$. 
Following \cite{BH05, Hurisse14, HelluyHurisse15, Hurisse17}, 
the mixture pressure
law should be derived from the
mixture entropy function $\mathscr S$ defined in \eqref{eq:Sr}.

Highlighting the dependency on $(\tau,e)$, the entropy of the mixture
reads
\begin{equation}
  \label{eq:mixture-entropy}
  \mathcal S(\tau,e,\mathbf{r})=\varphi_1 s(\tau_1(\mathbf{r}),e_1(\mathbf{r}))
  +(1-\varphi) s(\tau_2(\mathbf{r},e_2(\mathbf{r})),
\end{equation}
where $s$ is again the van der Waals EoS and the functions
$\tau_i(\mathbf{r})$ and $e_i(\mathbf{r})$ are defined in
\eqref{eq:volumes_energies_fractions}.
The associated pressure $p$ and the temperature $T$ of the mixture are
deduced from an extended Gibbs relation
\begin{equation}
\label{eq:Gibbs_relation_ext} 
Td   \mathcal S(\tau,e,\mathbf{r})=de+pd\tau+ \dfrac{\p   \mathcal  S}{\p \alpha}d\alpha
+ \dfrac{\p   \mathcal S}{\p \varphi}d\varphi +  \dfrac{\p   \mathcal
  S}{\p \xi}d\xi.
\end{equation}

Then the definitions of the mixture temperature and pressure, as
functions of $(\tau,e\mathbf{r})$, are
\begin{equation}
  \label{eq:mixture_pT} 
\begin{aligned}
  T(\tau,e,\mathbf{r}) & =
  \dfrac{\xi}{T(\tau_1(\mathbf{r}),e_1(\mathbf{r}))} +
  \dfrac{1-\xi}{T(\tau_2(\mathbf{r}),e_2(\mathbf{r}))},\\
\dfrac{p(\tau,e,\mathbf{r})}{  T(\tau,e,\mathbf{r})} & =\alpha 
  \dfrac{p(\tau_1(\mathbf{r}),e_1(\mathbf{r}))}{T(\tau_1(\mathbf{r}),e_1(\mathbf{r}))}
+ (1-\alpha)
\dfrac{p(\tau_2(\mathbf{r}),e_2(\mathbf{r}))}{T(\tau_2(\mathbf{r}),e_2(\mathbf{r}))}.
\end{aligned}
\end{equation}
The sound speed of the system \eqref{eq:relax-model} is 
\begin{equation}
  \label{eq:speed_sound0}
  c^2 =  -\tau^2 \dfrac{\p}{\p \tau} p + \tau^2 p \dfrac{\p}{\p e} p,
\end{equation}
which, using the expression of the mixture pressure
\eqref{eq:mixture_pT}, simplifies to
\begin{equation}
  \label{eq:speed_sound}
  \begin{aligned}
    - \dfrac{c^2}{T\tau^2} &= \dfrac{1}{\varphi}(-\alpha , \xi p) Hs_1
    \begin{pmatrix}
      -\alpha \\ \xi p
    \end{pmatrix}\\
    &+ \dfrac{1}{1-\varphi}(-(1-\alpha ), (1-\xi) p) Hs_2
    \begin{pmatrix}
      -(1-\alpha) \\ (1-\xi) p
    \end{pmatrix},
  \end{aligned}
\end{equation}
where $H_{s_i}$ denotes the hessian matrix of the phasic entropy
$s(\tau_i,e_i)$
\begin{equation}
  \label{eq:Hsi}
    H_{s_i}(\tau_i,e_i) =
  \begin{pmatrix}
    \dfrac{\p^2 s}{\p \tau_i^2} & \dfrac{\p^2 s}{\p \tau_i \p
      e_i}\\
    \dfrac{\p^2 s}{\p e_i \p \tau_i} & \dfrac{\p^2 s}{\p e_i^2}
  \end{pmatrix},
\end{equation}
and the dependency to the variables has been skipped for readability
reasons.

The convective system is hyperbolic if and only if the the right-hand side of
\eqref{eq:speed_sound} is negative. This is the case if the hessian
matrices $H_{s_1}$ and $H_{s_2}$ are negative definite, which is true in
concavity region of the van der Waals entropy, that is outside the
spinodal region $Z_\text{Spinodal}$.
Hence the system is non-strictly hyperbolic.
However, it has been highlighted in \cite{james} in the
isothermal context that the domains of hyperbolicity of
\eqref{eq:relax-model} strongly depend on the attraction basins of the
dynamical system \eqref{eq:fraction-model}. More precisely,
the invariant domains  of hyperbolicity for the relaxed
system are subsets of the attraction basins of the dynamical system. 

The convective part of the model \eqref{eq:relax-model}
inherits the wave structure of the Euler system. 
The fields associated with the fractions $\mathbf{r}$ are linearly
degenerated with the eigenvalue $u$. The momentum and energy
conservation laws are genuinely nonlinear fields with velocities $u\pm
c$ and the mass equation is linearly degenerated with velocity $u$.

The Riemann invariants associated to the wave of velocity $u$ are the
velocity $u$ and the pressure. Moreover the volume fraction, the mass
fraction and the energy fraction are Riemann invariants associated to
the genuinely nonlinear waves.

The positivity of the
fractions is ensured by both 
the positivity property of the
dynamical system, see Proposition
\ref{prop:sys_dyn_prop}-\eqref{it:prop1},
and the form of the convection equations of the
fractions $\mathbf{r}$, see
\cite{hurisse:hal-01976903}.

%-------------------------------------------------------
\subsection{Numerical illustrations}
\label{sec:VF}
%-------------------------------------------------------
Numerous numerical schemes have been proposed for homogeneous models
with relaxation, see again \cite{BH05} and \cite{Hurisse17} for models
involving stiffened gas or tabulated laws. 
We propose here a very standard approach, and take a special interest
to numerical illustrations.

The numerical approximation consists in a fractional step method.

We restrict to regular meshes of size $\Delta x = x_{i+1/2}-x_{i-1/2},
\, i\in \mathbb Z$.
The time step is $\Delta t = t^{n+1}-t^n$, $n\in \mathbb N$.
We focus on the convective part of \eqref{eq:relax-model} with an
initial condition
\begin{equation}
  \label{eq:dtW}
  \begin{cases}
    \p_t W + \p_x F(W) = 0,\\
    W(0,x) = W_0(x), 
  \end{cases}
\end{equation}
with $W=(\rho \alpha, \rho \varphi, \rho \xi, \rho, \rho u, \rho
E)^T$, and $F(W)= u W + p D$, with $D=(0,0,0,0, 1, u)^T$.
Let $W(t^n, x)$ be approximated by 
\begin{equation}
  \label{eq:VF0}
  w_i^n = \dfrac{1}{\Delta x}\int_{x_{i-1/2}}^{x_{i+1/2}} W(t^n, x)
  dx.
\end{equation}
Integrating the system on the space-time domain
$[x_{i-1/2},w_{i+1/2}]\times[t^n, t^{n+1}]$ provides
\begin{equation}
  \label{eq:VF1}
  W_i^{n+1}=W_i^n - \dfrac{\Delta t}{\Delta x}\left( \mathcal
    F_{i+1/2}^n -  \mathcal F_{i-1/2}^n\right).
\end{equation}
We choose the explicit HLLC numerical flux \cite{toro} to define the fluxes $\mathcal
F_{i+1/2}^n$ through the  interface $x_{i+1/2}\times [t^n, t^{n+1}]$.

The source terms of the system \eqref{eq:relax-model} are accounted
for by discretizing
\begin{equation}
  \label{eq:source1}
  \begin{cases}
    \dfrac{d}{dt}\rho(t)= 0,\\
    \dfrac{d}{dt}(\rho u)(t)= 0,\\
    \dfrac{d}{dt}(\rho E)(t)= 0,
  \end{cases}
  \begin{cases}
    \dfrac{d}{dt}(\rho \alpha)(t)= \dfrac{\rho}{\varepsilon} \mathbb F_\alpha(\mathbf{r},
    \rho, e),\\
    \dfrac{d}{dt}(\rho \varphi)(t)=\dfrac{\rho}{\varepsilon} \mathbb F_\varphi(\mathbf{r},
    \rho, e),\\
    \dfrac{d}{dt}(\rho \xi)(t)= \dfrac{\rho}{\varepsilon} \mathbb F_\xi(\mathbf{r},
    \rho, e).
  \end{cases}
\end{equation}
It can be written in an equivalent manner
\begin{equation}
  \label{eq:source2}
    \begin{cases}
    \dfrac{d}{dt} \alpha(t)= \dfrac{1}{\varepsilon} \mathbb F_\alpha(\mathbf{r}(t),
    \rho(0), e(0)),\\
    \dfrac{d}{dt}\varphi(t)= \dfrac{1}{\varepsilon} \mathbb F_\varphi(\mathbf{r}(t),
    \rho(0), e(0)),\\
    \dfrac{d}{dt} \xi(t)=  \dfrac{1}{\varepsilon} \mathbb F_\xi(\mathbf{r}(t),
    \rho(0), e(0)).
  \end{cases}
\end{equation}
The numerical approximation $W^{n+1}$ is an approximated solution of the
system \eqref{eq:source1} at time $t=\Delta t$ with the initial
condition $W^{n+1,*}$ deduced from the convection step.

The numerical method for the convective part has been validated on single-phase test cases
with a real van der Waals EoS proposed in \cite{GHS02}.

In order to capture accurately the thermodynamic equilibrium, one
should ideally consider infinitely fast relaxation with $\varepsilon=0$. The
integration of the source terms \eqref{eq:source2} reduces the the the
projection of the solution on the appropriate equilibrium (described
in Proposition \ref{prop:equilibrium_states}), depending
on the basin of attraction the state $W^{n+1,*}$ belongs to.

Unfortunately, as mentioned in Section \ref{sec:equil-attr}, the
boundaries of the basins of attraction are not explicitly defined.
This is for instance the case of the basins of attraction of the
spinodal zone and the metastable zones.
These basins are either delimited by the saturation dome, which
determination requires the resolution of the nonlinear system
\eqref{eq:egalite}, or by an unstable manifold, which numerical approximation is
intrinsically not reachable.
Hence we consider in the sequel finite but sufficiently small
relaxation time parameter $\varepsilon$ coupled with a 
Runge-Kutta 4 integration method. 

Note that in the isothermal case, studied in \cite{ghazi19}, the
determination of the basins of attraction is precise enough to perform
infinitely fast relaxation with $\varepsilon=0$.

%-------------------------------------------------------
\subsubsection{Single-phase test case}
\label{sec:single-phase-num}
%-------------------------------------------------------

We provide a validation test case which mimics the one proposed in
\cite{GHS02}, for a non-reduced van der Waals equation of state.
The Riemann data correspond to a left stable liquid state and a
right stable vapor state, namely
\begin{equation}
  \label{eq:sod_test}
  \begin{aligned}
    \rho_L = 1.111, \quad u_L = 0., \quad p_L = 0.2, \quad \alpha_L=\varphi_L=\xi_L=10^{-6},\\
    \rho_R = 0.277, \quad u_R = 0, \quad p_R = 0.11 \quad \alpha_R=\varphi_R=\xi_R=10^{-6}.
  \end{aligned}
\end{equation}
This test case corresponds to a single-phase subsonic 1-rarefaction
wave, since the fractions are constant and small. 
The domain $[0,1]$ is decomposed into 500 cells and the discontinuity
is applied at $x_0=0.5$.
 The final time of
computation is $0.4$s and the CFL coefficient is $0.9$.

The global behaviour is coherent with the results provided
\cite{GHS02} .
In particular, the curve profiles around the contact discontinuity
is not precise enough. A more robust numerical flux should be
considered to overcome the problem, which actually disappears as the
grid is refined.

\begin{figure}[htpb]
  \centering
  \includegraphics[width=0.45\linewidth]{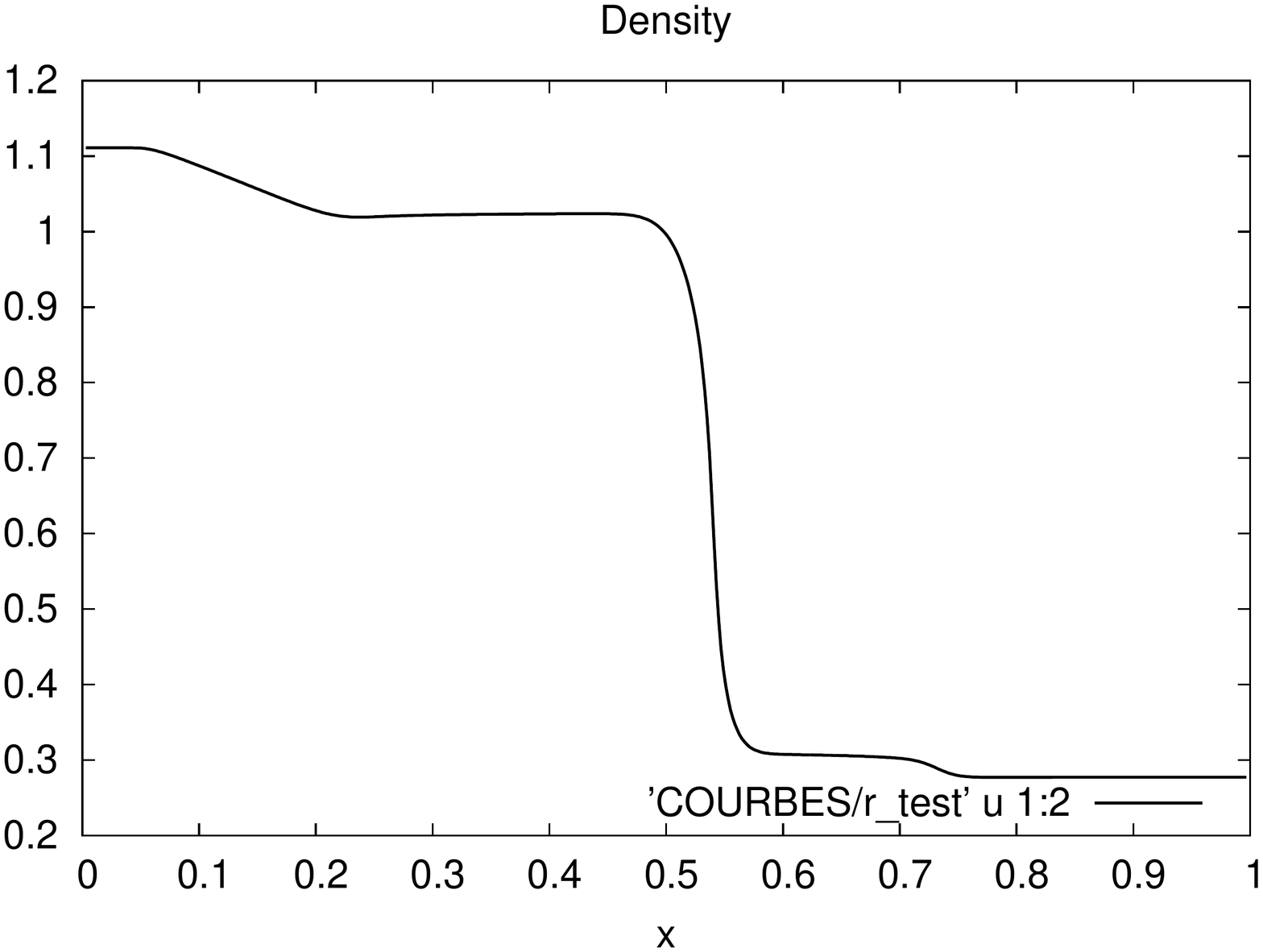}
  \includegraphics[width=0.45\linewidth]{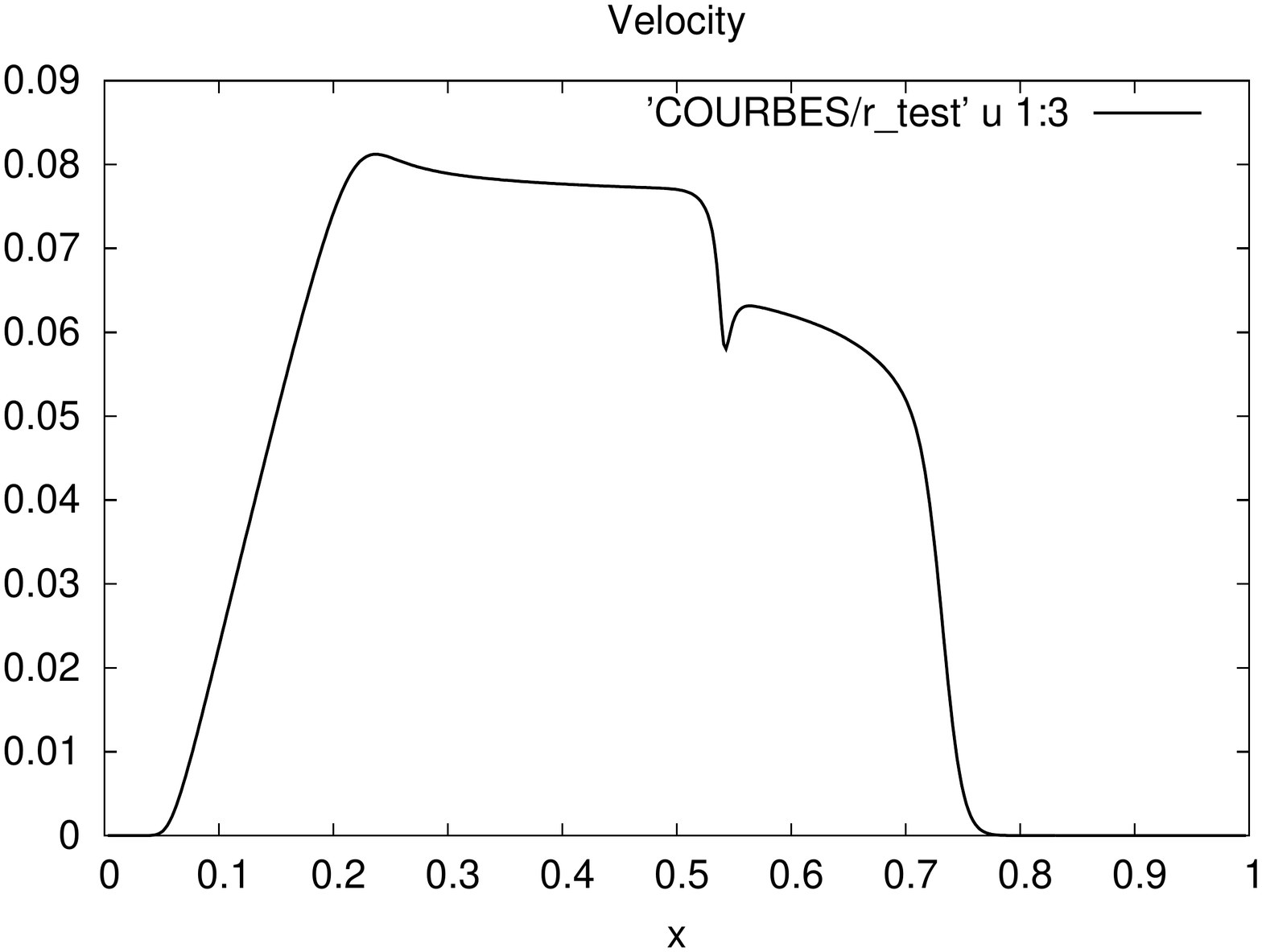}
  \includegraphics[width=0.45\linewidth]{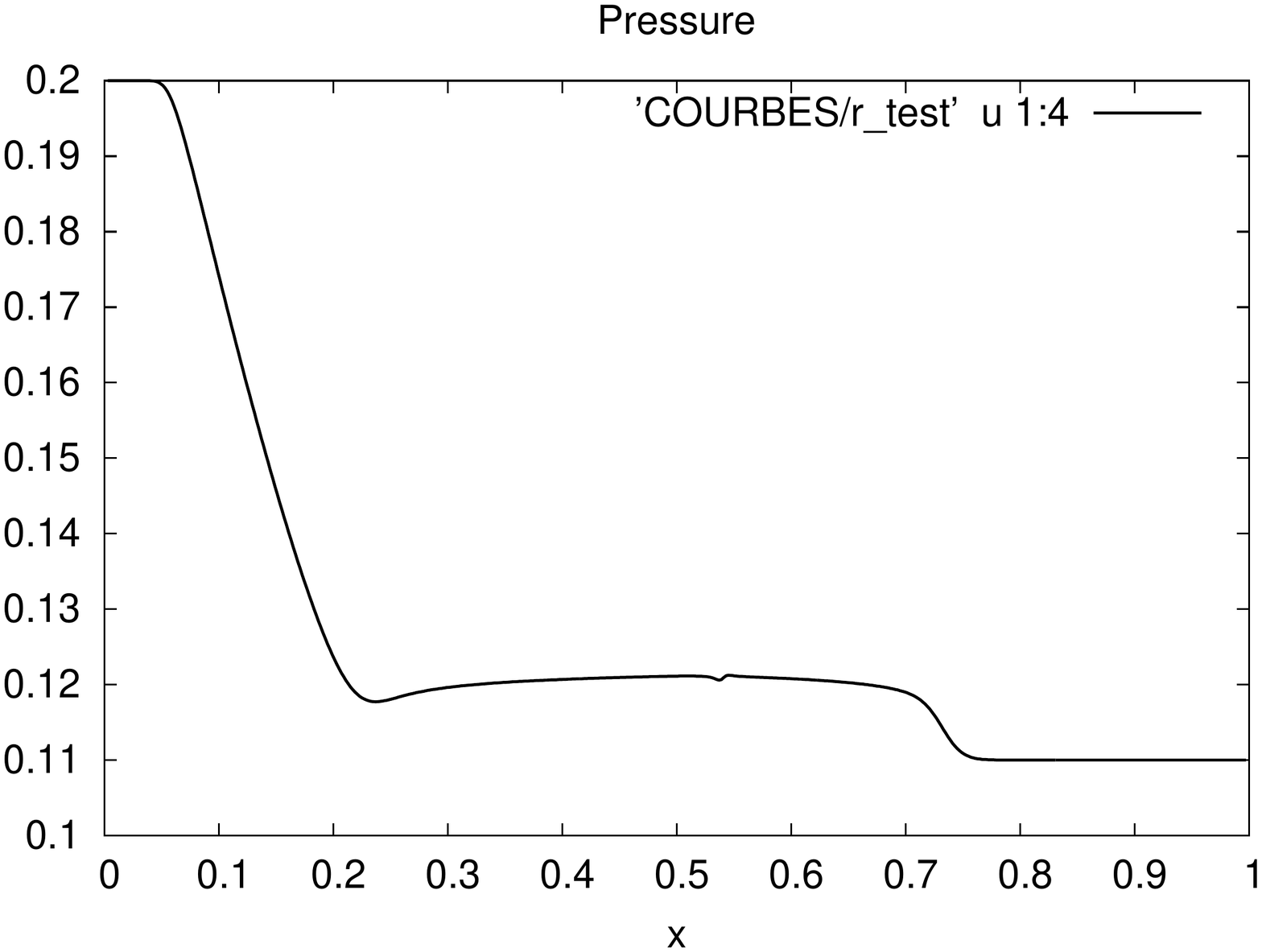}
  \includegraphics[width=0.45\linewidth]{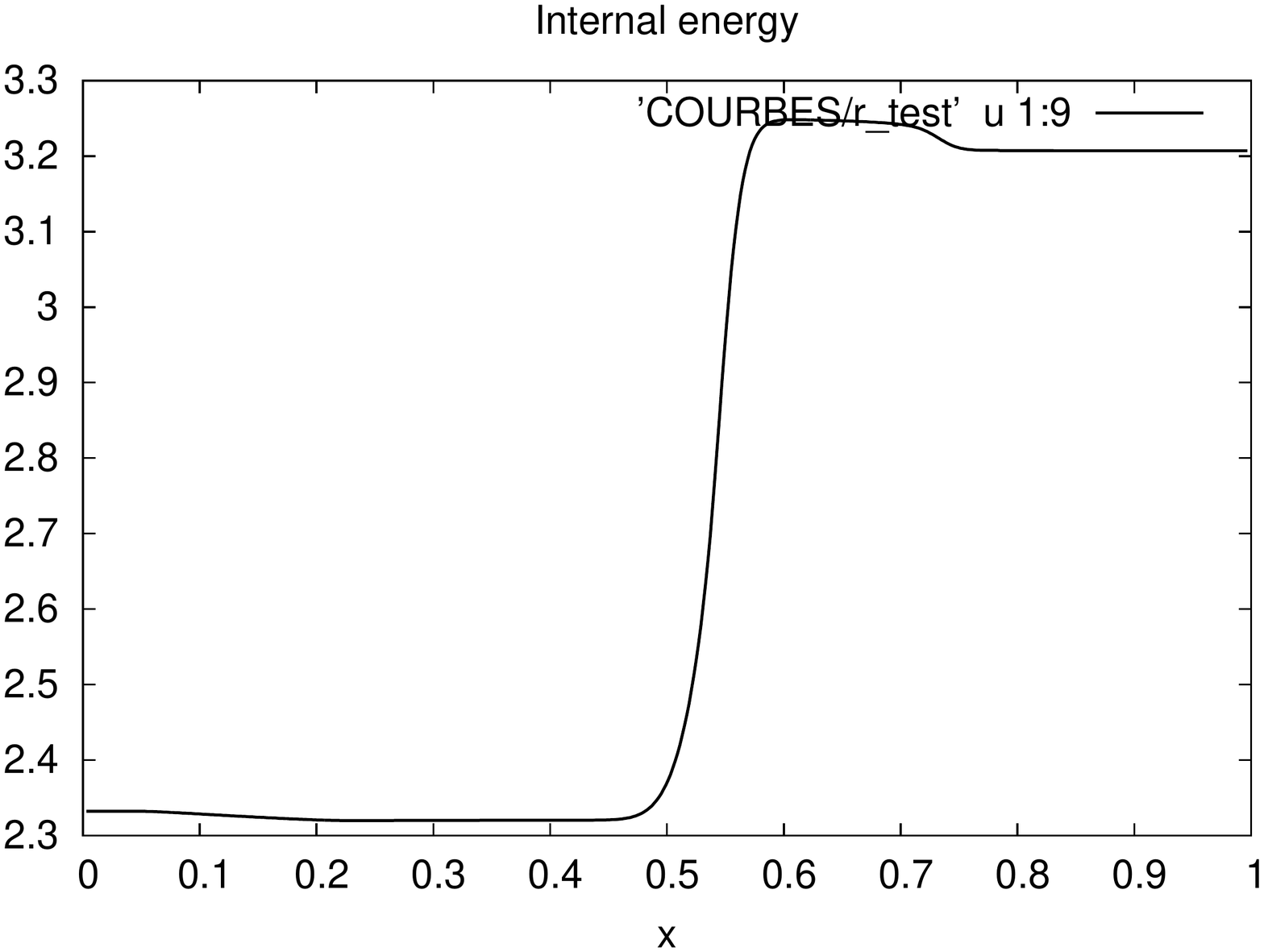}
  \includegraphics[width=0.45\linewidth]{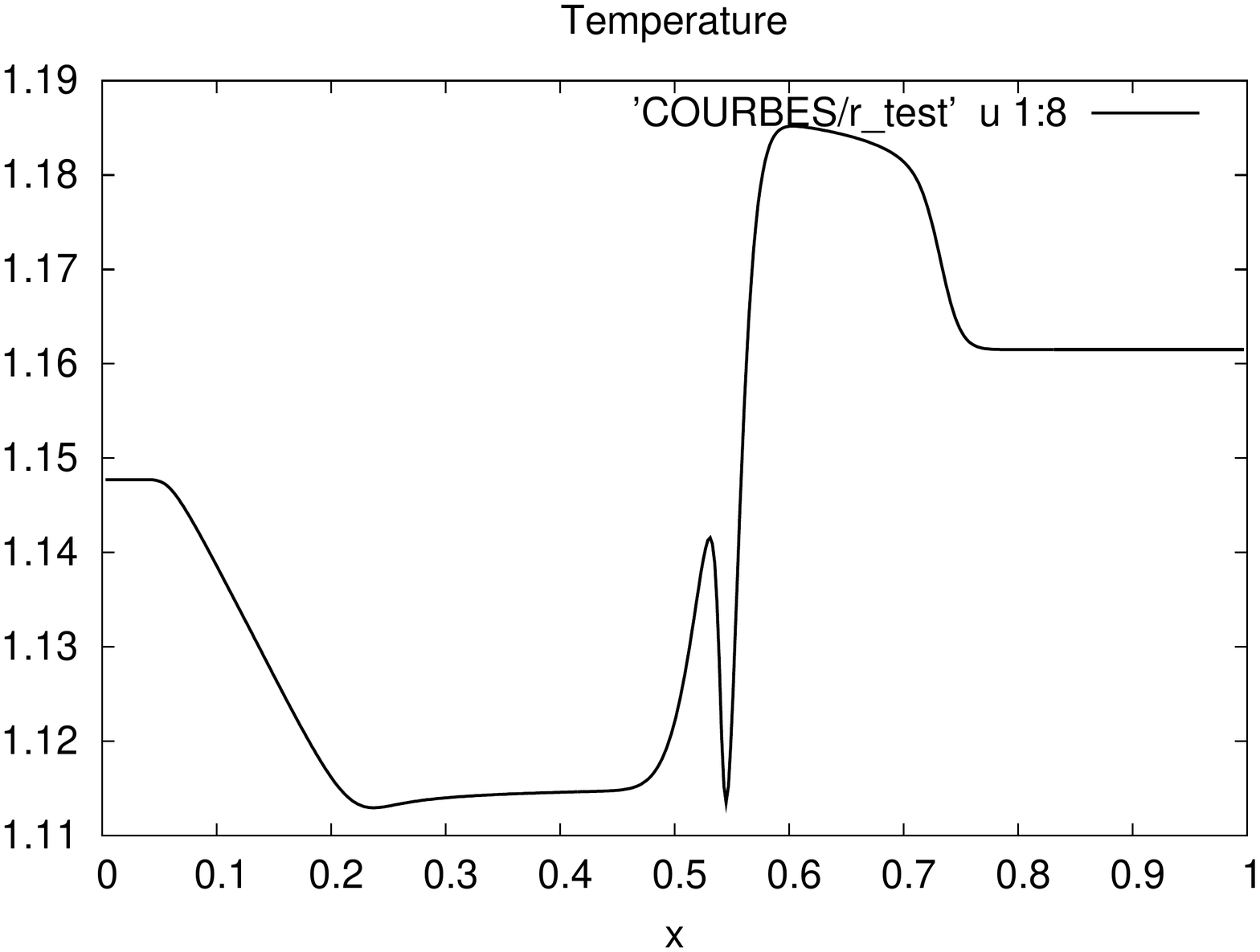}
  \includegraphics[width=0.45\linewidth]{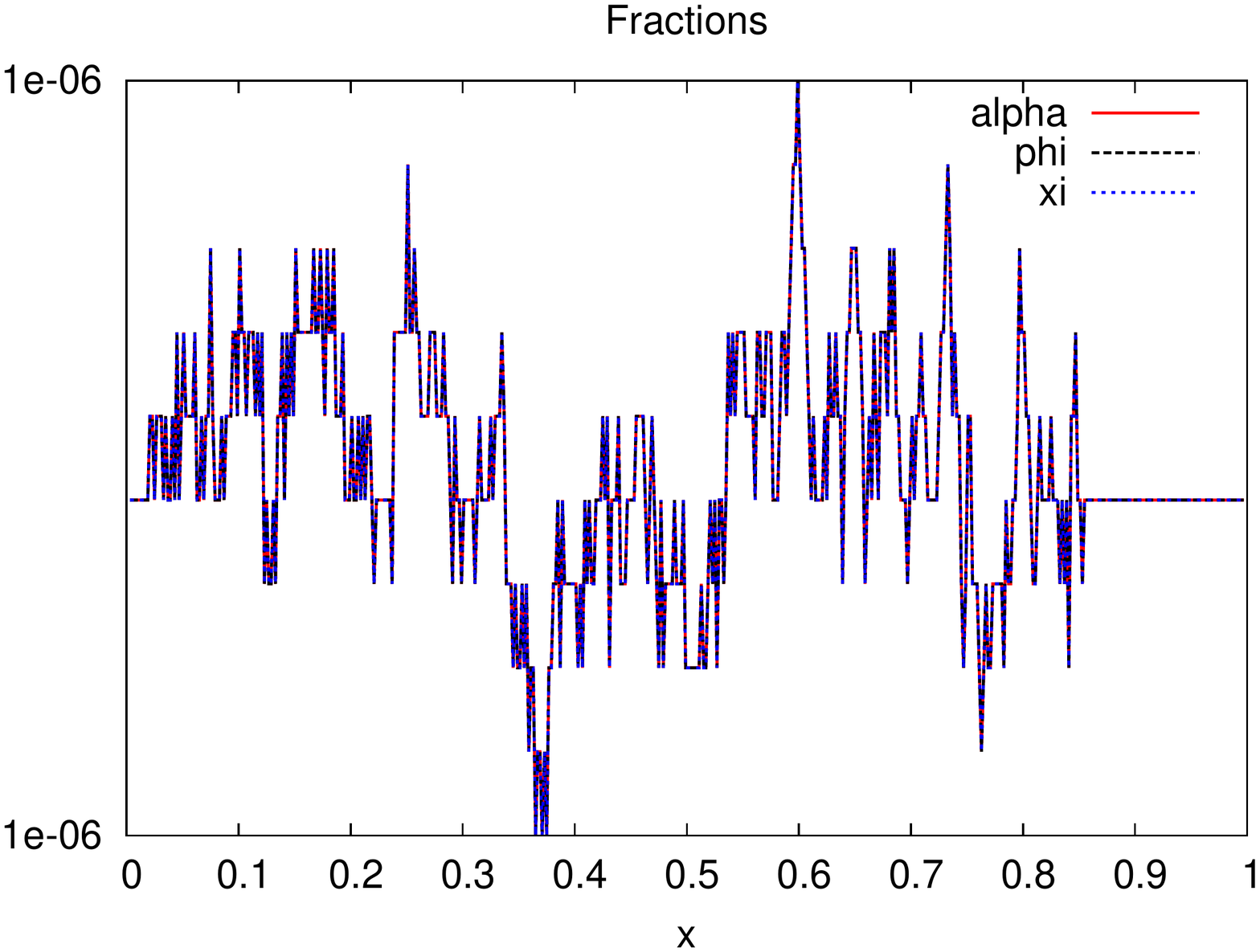}
  \caption{Sod test case. From top left to bottom right: density
    profile, velocity, pressure, internal energy, temperature and
    fractions profile with respect to the space variable.}
  \label{fig:sod}
\end{figure}

%-------------------------------------------------------
\subsubsection{Interaction of a metastable liquid state and a
  saturation state}
\label{sec:meta-satu}
%-------------------------------------------------------
The test case corresponds to a Riemann problem with a left metastable
liquid state and a right saturation state.
The initial data are
\begin{equation}
  \label{eq:meta_saturation}
  \begin{aligned}
    \rho_L = 1.25, \quad u_L = 0., \quad p_L = 0.02, \quad
    \alpha_L=\varphi_L=\xi_L=0.3,\\
    \rho_R = 0.3125, \quad u_R = 0, \quad p_R = 0.0785 \quad
    \alpha_R=0.0907, \quad \varphi_R=0.344, \quad \xi_R=0.2577.
  \end{aligned}
\end{equation}
The right state is at saturation since it holds
\begin{equation}
  \label{eq:meta_saturation2}
  p_{1,R} = p_{2,R}=0.0785, \quad T_{1,R} = T_{2,R}=1.0188, \quad \mu_{1,R}=\mu_{2,R}=2.102.
\end{equation}

\begin{figure}[htpb]
  \centering
  \includegraphics[width=0.45\linewidth]{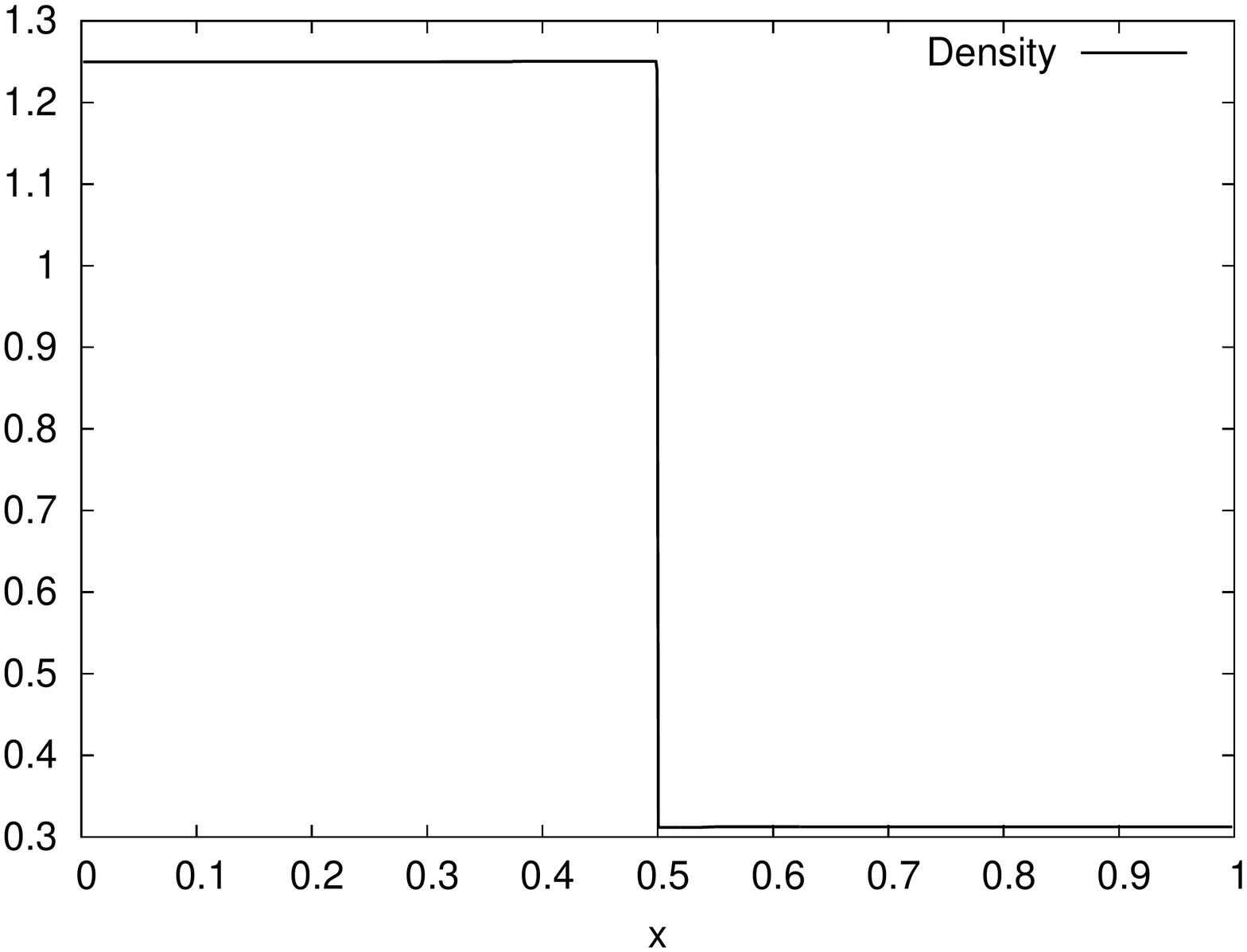}
  \includegraphics[width=0.45\linewidth]{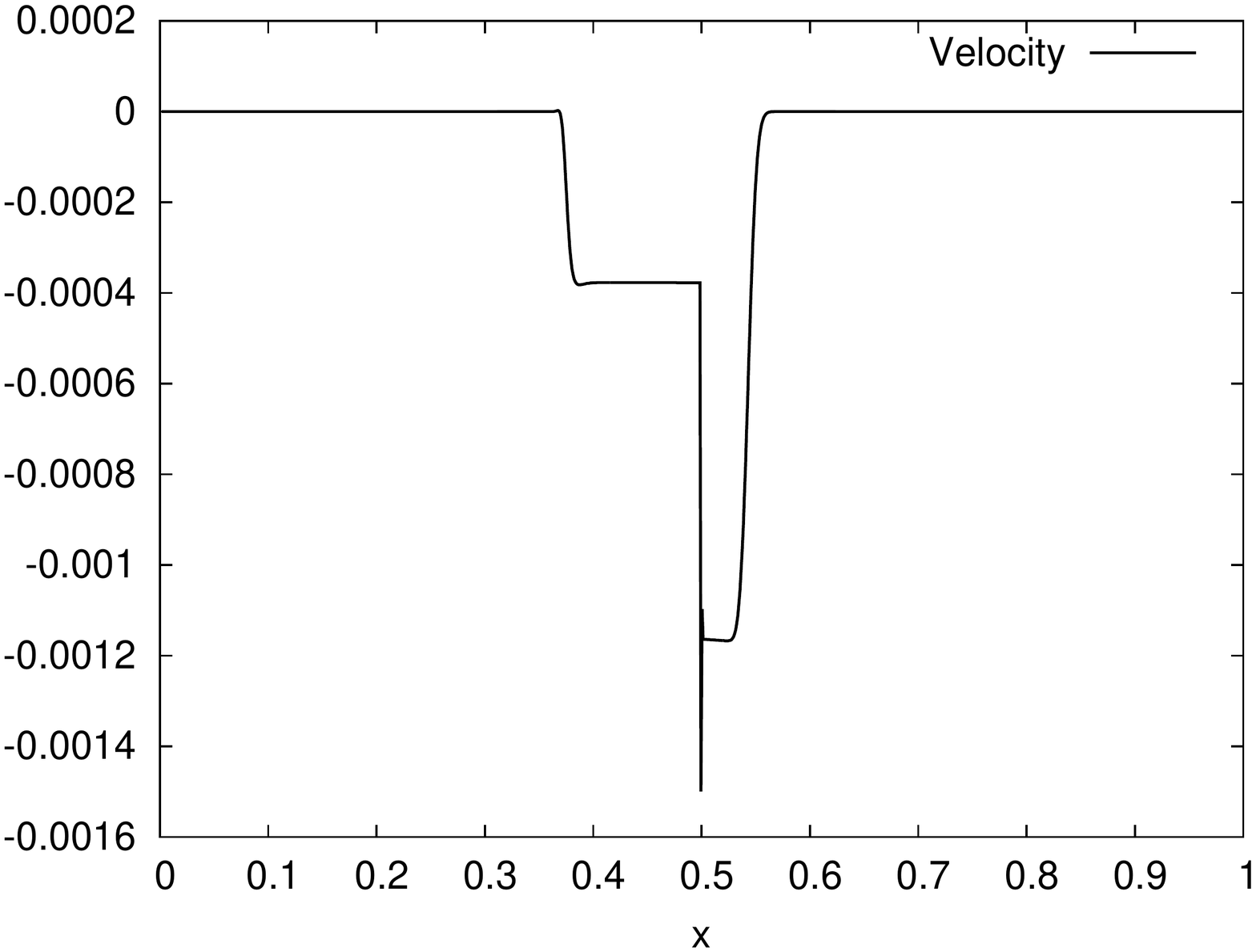}
  \includegraphics[width=0.45\linewidth]{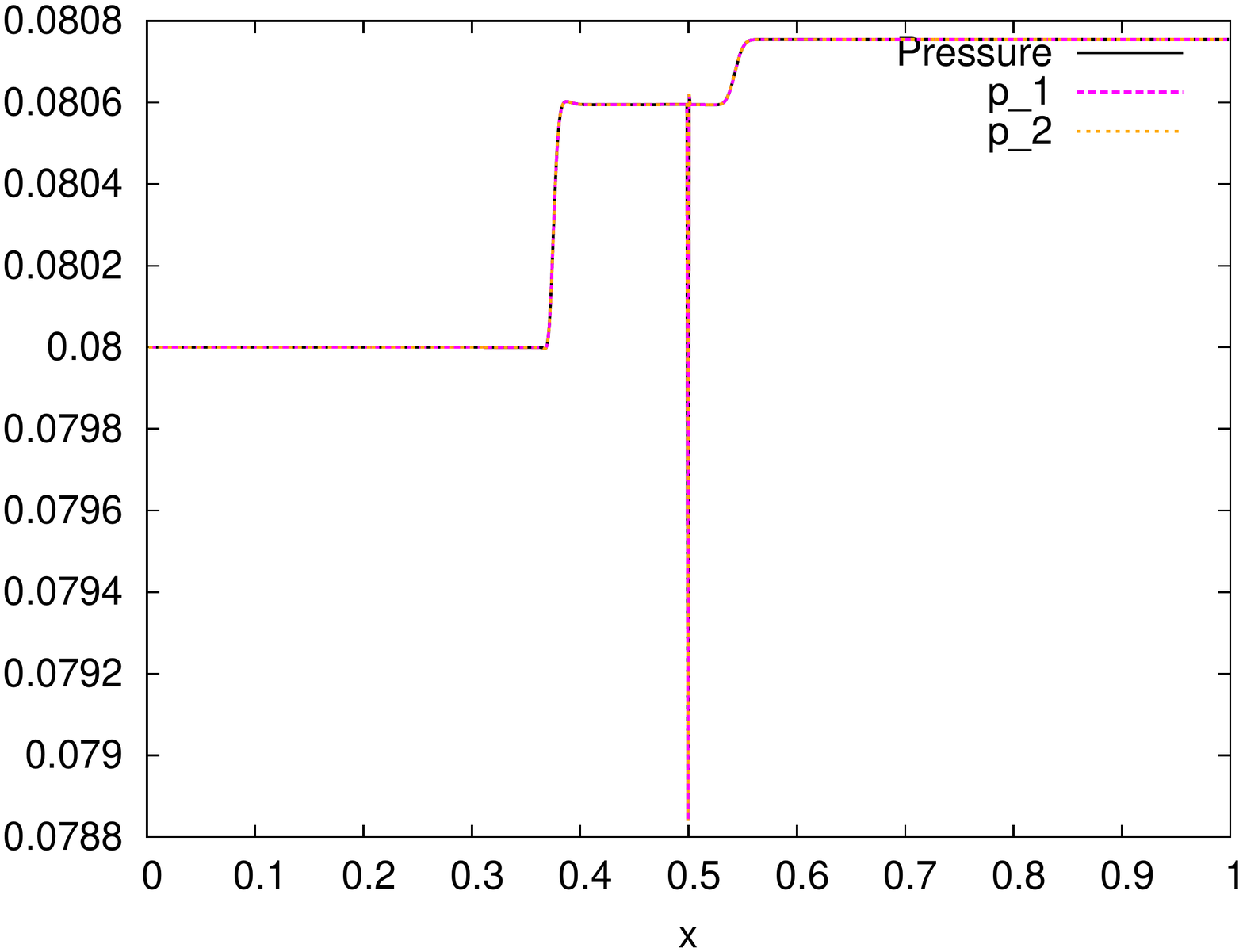}
  \includegraphics[width=0.45\linewidth]{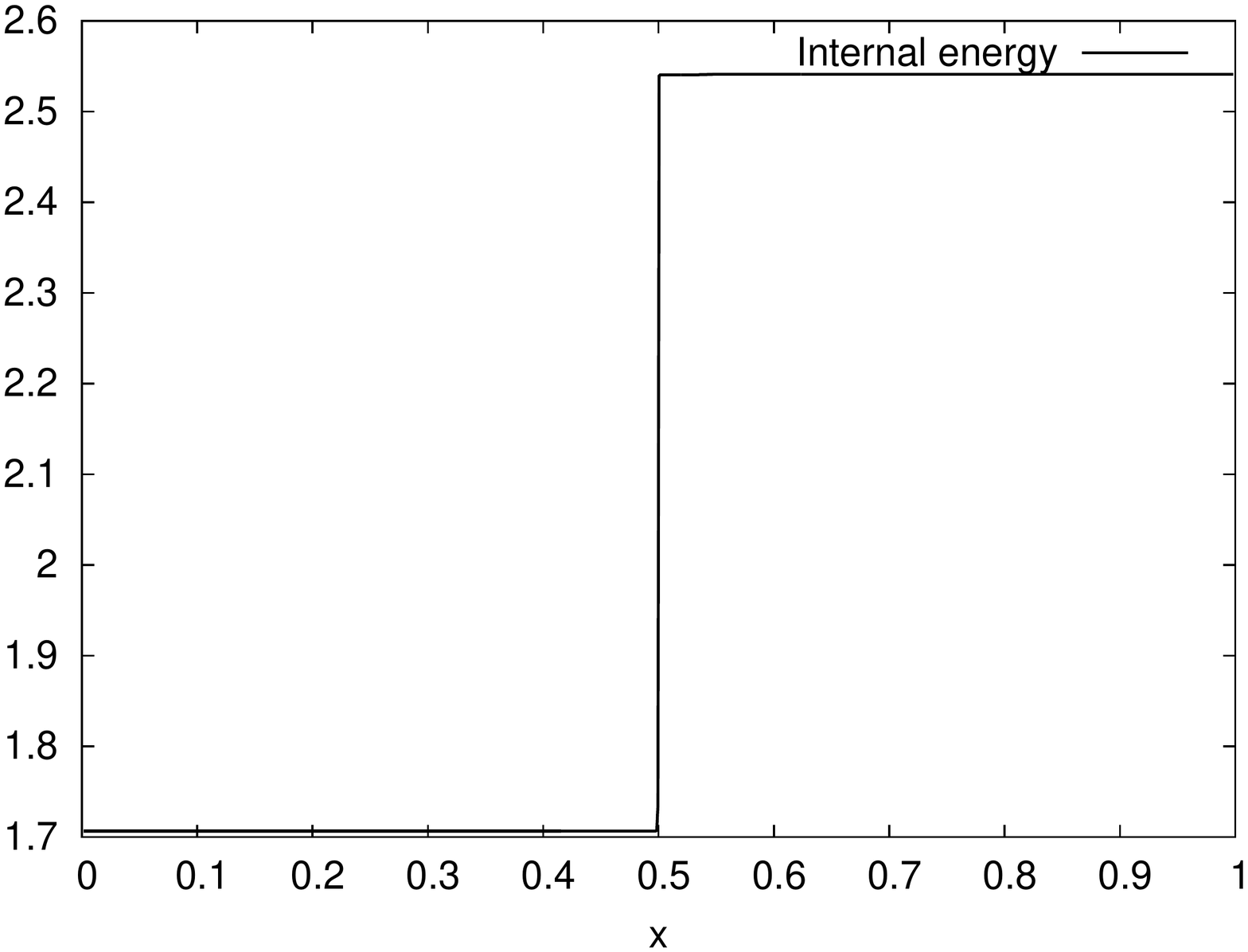}
  \includegraphics[width=0.45\linewidth]{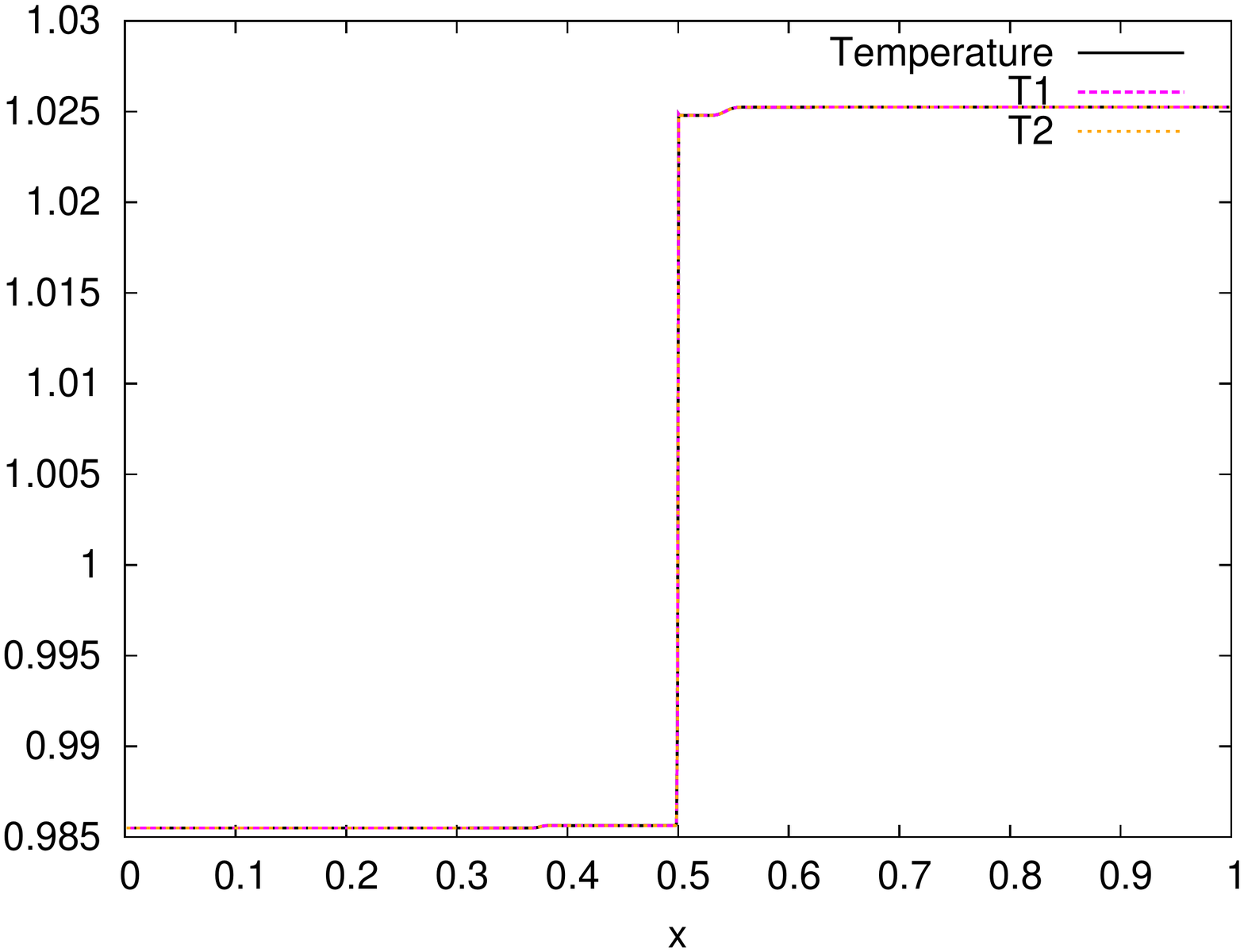}
  \includegraphics[width=0.45\linewidth]{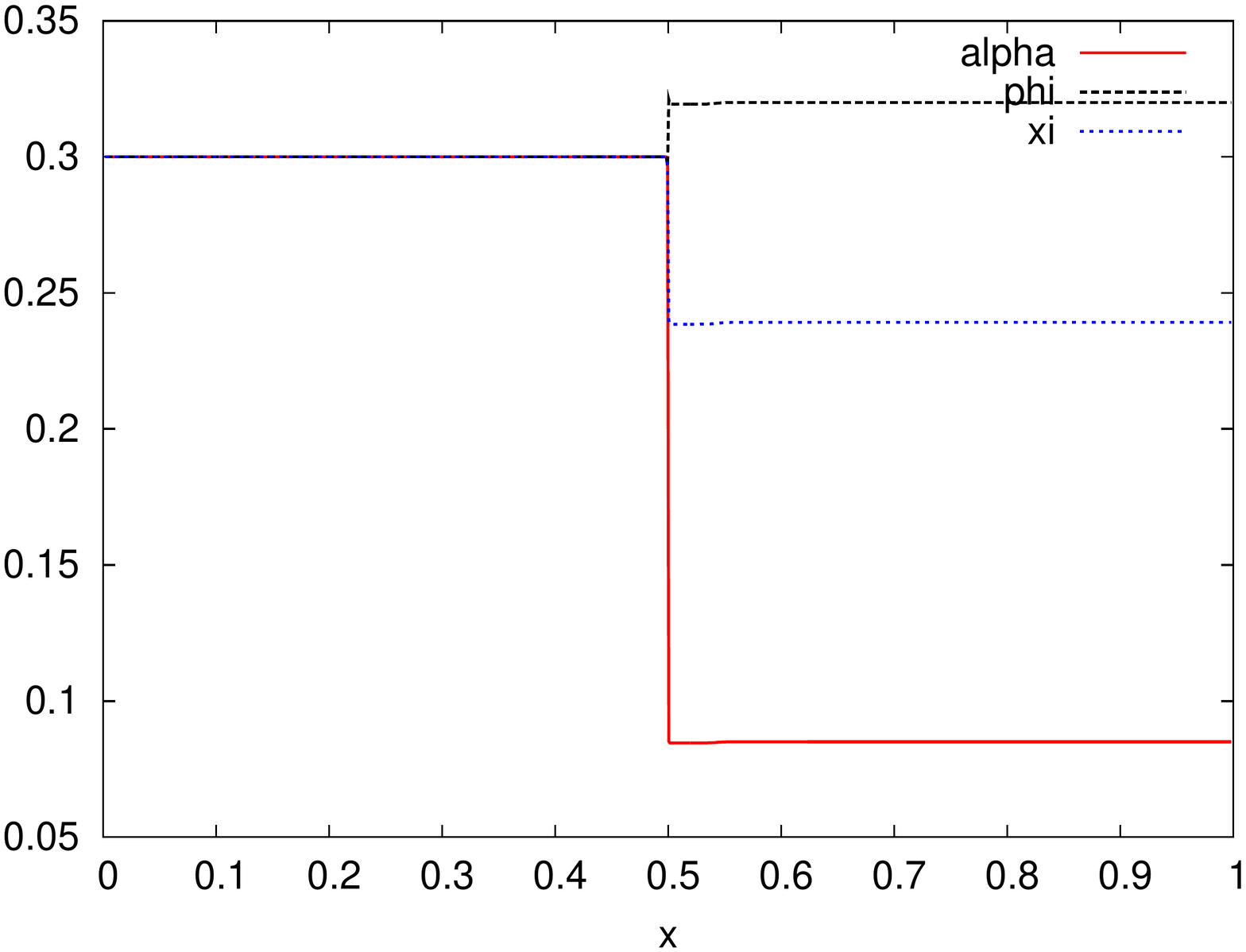}
  \caption{Interaction of a metastable liquid and a saturation state. From top left to bottom right: density
    profile, velocity, pressure, internal energy, temperature and
    fractions profile with respect to the space variable.}
  \label{fig:inter}
\end{figure}

%-------------------------------------------------------
% Conclusion
%-------------------------------------------------------
\section{Conclusion}
\label{sec:conclusion}
This paper concerns the construction of appropriate relaxation source
terms towards thermodynamic equilibrium for a liquid-vapor flow with
the possible appearance of metastable states. Extending the works
\cite{james,ghazi19} in the isothermal context, the two
phases are assumed to follow the same  non convex van der Waals equation of state. 
We provide time evolution equations of the fractions of volume, mass and energy of one of
the phases which guarantee the growth of the mixture entropy. 
The dynamical model admits two major properties. First the attractive
equilibria are either saturation states, characterized by the
equalities of the phasic pressures, temperatures and chemical
potential, or stable or metastable states, for which the two phases
identify.
In the latter case, the equilibrium corresponds to the equality of the
fractions to an asymptotic value between 0 and 1 strictly. 
The fluid is either in a liquid
or vapor, metastable or stable, state, but the fractions do not
cancel, as it is classically the case in the Baer-Nunziato type model.
Second, when considering a mixture state belonging to a metastable
zone, there are two possible equilibria depending the initial
condition on the fractions. The system reaches either a saturation
state or converges toward the metastable initial state characterized
by the identification of the two phases. In contrast with standard models, this does not correspond
to volume fractions equal to 0 or 1.
The method we propose here should be extended to more realistic non convex equations of state. 
Using tabulated laws could be a real issue because of the difficulty of determining the attraction basins.
Another issue is the coupling with fluid dynamics, which is merely illustrated here. It deserves a more
careful study, from both theoretical and numerical viewpoints.

%-------------------------------------------------------
% Bibliographie
%-------------------------------------------------------
\bibliographystyle{plain}
\bibliography{thermo}

\end{document}